\documentclass[11pt]{article}
\usepackage[toc,page]{appendix}
\usepackage[english]{babel}
\usepackage[backref]{hyperref}
\usepackage{amsmath}
\usepackage{amssymb,amsthm,amscd}
\usepackage[latin1]{inputenc}
\usepackage[textwidth=18cm,textheight=24cm]{geometry}
\usepackage{mathrsfs}
\usepackage[thinlines]{easybmat}

\newcommand{\diag}{\operatorname{diag}}

\newcommand{\C}{\mathbb{C}}

\newcommand{\R}{\mathbb{R}}

\newcommand{\Z}{\mathbb{Z}}
\newcommand{\N}{\mathbb{N}}
\newcommand{\T}{\mathbb{T}}
\newcommand{\n}{\mathcal{N}}

\newcommand{\g}{\mathfrak{g}}

\newcommand{\A}{\mathscr A}

\newcommand{\I}{\mathbb{I}}
\renewcommand{\d}{\mathrm{d}}

\newcommand{\Exp}[1]{\operatorname{e}^{#1}}

\newcommand{\im}{\operatorname{i}}

\renewcommand{\L}{\mathscr L}

\newtheorem{pro}{Proposition}
\newtheorem{lemma}{Lemma}
\newtheorem{definition}{Definition}
\newtheorem{theorem}{Theorem}
\newtheorem{cor}{Corollary}
\newcommand{\CMV}{}
\renewcommand{\Re}{\operatorname{Re}}
\begin{document}
\title{Orthogonal Laurent polynomials in unit circle,\\
 extended CMV ordering and 2D Toda type integrable hierarchies}
\author{Carlos \'{A}lvarez-Fern\'{a}ndez$^{1,2,\dag}$ and Manuel Ma\~{n}as$^{1,\ddag}$\\ \\
$^1$Departamento de F\'{\i}sica Te\'{o}rica II, Universidad Complutense, \\
28040-Madrid, Spain\\
$^2$Departamento de M\'{e}todos Cuantitativos, Universidad Pontificia Comillas, \\ 28015-Madrid, Spain }
\date{$^\dag$calvarez@cee.upcomillas.es,  $^\ddag$manuel.manas@fis.ucm.es}
\maketitle
\begin{abstract}
Orthogonal Laurent polynomials in the unit circle and the theory of Toda-like integrable systems are connected using the Gauss--Borel factorization of a Cantero--Moral--Vel\'{a}zquez  moment matrix, which is constructed in terms of a complex quasi-definite measure supported in the unit circle. The  factorization
 of the moment matrix  leads to orthogonal Laurent polynomials in  the unit circle and the corresponding second kind functions. Jacobi operators, 5-term recursion relations and Christoffel--Darboux  kernels, projecting to particular spaces of  truncated Laurent polynomials, and corresponding Christoffel--Darboux formulae are obtained within this point of view in a completely algebraic way. Cantero--Moral--Vel\'{a}zquez sequence of Laurent monomials is generalized and recursion relations,  Christoffel--Darboux  kernels, projecting to general spaces of truncated Laurent polynomials and corresponding Christoffel--Darboux formulae are found in this  extended context. Continuous deformations of the  moment matrix are introduced and is shown how they induce a time dependant orthogonality problem related to a Toda-type integrable system, which is connected with the well known Toeplitz lattice. Using the classical integrability theory  tools the Lax and Zakharov--Shabat equations are obtained. The dynamical system associated with the coefficients of the orthogonal Laurent polynomials is explicitly derived and compared with the classical Toeplitz lattice dynamical system for the Verblunsky coefficients of Szeg\H{o} polynomials for a positive measure. Discrete flows are introduced and related to Darboux transformations. Finally, the representation of the orthogonal Laurent polynomials (and its  second kind functions), using the formalism of Miwa shifts, in terms of $\tau$-functions is presented and bilinear equations are derived.
\end{abstract}
\tableofcontents
\section{Introduction}\label{intro}
This paper puts  focus on orthogonal Laurent polynomials in the unit circle (OLPUC), a  subject with strong links to that of orthogonal polynomials in the unit circle (OPUC),  see \cite{Szego}, \cite{Simon-1} and \cite{Simon-2}. Despite that is well established that this matter is as source of interesting problems and applications in approximation theory, we are mainly interested in its connections with the general theory of integrable systems.

Let us introduce here some notation that will be used along  this article. We will denote the unit circle by $\T:=\{z \in \C:
|z|=1\}$ ,  $\mathbb{D}:=\{z \in \C : |z|<1\}$ stands for the  unit disk and  $\Lambda_{[p,q]}:=\text{span} \{z^{-p},z^{-p+1},\dots,z^q\}$ denotes the linear space of complex Laurent polynomials with their corresponding restrictions on their degrees while $\Lambda_{[\infty]}$ for the infinite set of Laurent polynomials. When $z\in\T$ we will use the parametrization $z=\text{e}^{\text{i}\theta}$ with $\theta\in[0,2\pi)$.

A complex Borel measure $\mu$ supported in $\T$  is said to be positive definite if it maps measurable sets into non-negative numbers. When the measure $\mu$ is absolutely continuous with respect to the Lebesgue measure of the circle $\d\theta$, according to the Radon--Nikodym theorem, it can be always expressed using a complex weight function $w$, so that $\d \mu(\theta)=w({\theta}) \d \theta$. If in addition the measure is positive definite then the weight function should be a non-negative function on $\T$. For notational simplicity we will use, whenever it is convenient, the complex notation $\d \mu(z)=\text{i}\text{e}^{\text{i}\theta}\d \mu(\theta)$. If $\mu$ is a positive Borel measure supported in $\T$, then the OPUC or Szeg\H{o} polynomials are monic polynomials $P_n$ of degree equal or less than $n$ that satisfy the following system of equations, the orthogonality relations,
\begin{align}\label{szego}
\int_{\T}P_n(z) z^{-k} \d \mu(z)&=0, & k&=0,1,\dots,n-1.
\end{align}
It is well known  the deep connections between orthogonal polynomials in the real line (OPRL) in $[-1,1]$ and OPUC  (e.g. \cite{Freud,Berriochoa}).  Let us observe that for this analysis the use of spectral theory techniques requires the study of the operator of multiplication by \emph{z}. The study of the matrix associated to this operator leads to recurrence laws. Both, OPRL and OPUC have recurrence laws, but the main difference is that in the real case three term recurrence laws provide a tridiagonal matrix, the so called Jacobi operator, while in the circle case one is lead to a Hessenberg matrix \cite{Golub}, leading to a  more involved scenario to deal with than the Jacobi case (as it is not a sparse matrix with a finite number of non vanishing diagonals). More precisely, the study of the recurrence relations for the OPUC requires the definition of the reciprocal or reverse Szeg\H{o} polynomials $P^*_l(z):=z^l \bar P_l(z^{-1})$ and the reflection or Verblunsky\footnote{Schur parameters is another usual name. The definition is not unique and $\alpha_l:=-\overline{P_{l+1}(0)}$ is another common definition.} coefficients $\alpha_l:=P_l(0)$. With these elements the recursion relations for the Szeg\H{o} polynomials can be written as
\begin{align}
\begin{pmatrix} P_l \\ P_{l}^* \end{pmatrix}
=
\begin{pmatrix} z & \alpha_l \\ z \bar \alpha_l & 1 \end{pmatrix}
\begin{pmatrix} P_{l-1} \\ P_{l-1}^* \end{pmatrix}.
\end{align}
There has been a relevant number of studies on the zeroes of the OPUC, see for example \cite{Alfaro,Ambrolazde,Barrios-Lopez}, or \cite{Garcia,Godoy,Mhaskar,Totik},  which have interesting applications to signal analysis theory, see \cite{Jones-1,Jones-2,Pan-1,Pan-2}. Despite of that the situation is still far from the corresponding state of the art in the OPRL context. A second important issue is the fact that the set of Szeg\H{o} polynomials is in general not dense in the Hilbert space $L^2(\T,\mu)$.  As it follows from Szeg\H{o}'s theorem it holds that for a nontrivial probability measure on $\d\mu$ with
Verblunsky coefficients $\{\alpha_n\}_{n=0}^\infty$ the  Szeg\H{o}'s polynomials are dense in $L^2(\T,\mu)$ if and
only if $\prod_{n=0}^\infty (1-|\alpha_n)|^2)=0$. If $\d\mu$ is an absolutely continuous probability measure then the Kolmogorov's density theorem follows: polynomials are dense in $L^2(\T,\mu)$ if and only if the so called Szeg\H{o}'s condition $\int_{\T}\log (w(\theta)\d\theta=-\infty$ is fulfilled \cite{Simon-S}.

Orthogonal Laurent polynomials in the real line (OLPRL), where introduced in \cite{Jones-3,Jones-4} in the context of the strong Stieltjes moment problem, i.e finding a positive measure $\mu$ such that its moments
\begin{align}
m_j=\int_{\R} x^{j} \d \mu(x) \quad j=0,\pm 1, \pm 2, \dots
\end{align}
are known in advance. When the moment problem has a solution, there exist polynomials $\{Q_n\}$ such that
\begin{align}
\int_{\R}x^{-n+j}Q_n(x)\d \mu(x)&=0, &j&=0,\dots,n-1,
 \end{align}
 which are called Laurent polynomials or L-polynomials. The theory of Laurent polynomials on the real line was developed in parallel with the theory of orthogonal polynomials, see \cite{Cochran,Diaz,Jones-5} and \cite{Njastad}. Orthogonal Laurent polynomials theory was carried from the real line to the circle \cite{Thron} and subsequent works broadened the matter (e.g. \cite{Barroso-Vera,CMV,Barroso-Daruis,Barroso-Snake}), treating subjects like recursion relations, Favard's theorem, quadrature problems, and Christoffel--Darboux formulae.

The analysis of OLPUC, and specially the use of the Cantero--Moral--Vel\'{a}zquez  (CMV) \cite{CMV} representation is very helpful in the study of a number properties of Szeg\H{o} polynomials. Different reasons support this statement, for example as we mentioned while the OLPUC are always dense in $L^2(\T,\mu)$ this is not true in general for the OPUC, see  \cite{Bul} and \cite{Barroso-Vera}, and  also the bijection between OLPUC in the CMV representation and the ordinary Szeg\H{o} polynomials allows to replace the complicated recurrence relations with a five-term recurrence relation more alike to the structure of the OPRL. The representation of the operator of multiplication by $z$ is much more natural using CMV matrices than using Hessenberg matrices. This is the main motivation for us in order to take CMV matrices as a essential element in our scheme. Other papers have reviewed and broadened the study of CMV matrices, see for example \cite{CMV-Simon,Killip}. Alternative or generic orders in the base used to span $\Lambda_{[\infty]}$ can be found in \cite{Barroso-Snake}

The approach to the integrable hierarchies that we use here  is based on the Gauss--Borel factorization. The seminal paper of M. Sato \cite{sato} and further developments performed by the Kyoto school \cite{date1}-\cite{date3} settled the Lie-group theoretical description of the integrable hierarchies. It was M. Mulase, in the key paper \cite{mulase}, the one who made the connection between factorization problems, dressing procedures and integrability. In this context, K. Ueno and T. Takasaki \cite{ueno-takasaki} performed an analysis of the Toda-type hierarchies and their soliton-like solutions. In a series of papers, M. Adler and P. van Moerbeke \cite{adler}-\cite{adler-vanmoerbeke-5}, made clear the connection between the Lie-group factorization, applied to Toda-type hierarchies (what they call discrete K-P) and the Gauss--Borel factorization applied to a moment matrix that comes from an orthogonality problems; thus, the corresponding orthogonal polynomials  are closely related to specific  solutions of the integrable hierarchy. See  \cite{manas-martinez-alvarez} and \cite{cum} for further developments in relation with the factorization problem, multicomponent Toda lattices and generalized orthogonality. In the paper \cite{Adler-Van-Moerbecke-Toeplitz} a profound study of the OPUC and the Toda-type associated lattice, called the Toeplitz lattice (TL), was performed. A relevant reduction of the equations of the TL has been found by L. Golinskii \cite{Golinskii} in the context of Schur flows when the measure is invariant under conjugation, another interesting paper on this subject is \cite{Mukaihira}. The Toeplitz Lattice has been proved equivalent to the Ablowitz--Ladik lattice (ALL), \cite{a-l-1,a-l-2}, and that work has been generalized to the link between matrix orthogonal polynomials and the non-Abelian ALL in \cite{Cafasso}. Both of them have to deal with the Hessenberg operator for the multiplication by $z$.

Our aim is to explore the connection between Toda-type integrable systems and orthogonality in the circle from a different point of view. As we proof in this paper the  CMV representation is a bridge between the factorization techniques used in \cite{afm-2} and the circular case. We will see that many results obtained in \cite{afm-2} about Christoffel--Darboux (CD) formulae, continuous and discrete deformations, and $\tau$-function theory can be extended to the circular case under the suitable choice of moment matrices and shift operators.

Let us recall the reader that measures and linear functionals are closely connected;  given a linear functional $\L$ on $\Lambda_{[\infty]}$ we define the corresponding moments of $\L$ as $c_n:=\L[z^n]$ for all the possible integer values of $n \in \Z$. The functional $\L$ is said to be Hermitian whenever $c_{-n}=\overline{c_n}$, $\forall n \in \Z$. Moreover, the functional $\L$ is defined as quasi-definite (positive definite) when the principal submatrices of the Toeplitz moment matrix $(\Delta_{i,j})$, $\Delta_{i,j}:=c_{i-j}$, associated to the sequence $c_n$ is non-singular (positive definite), i.e. $\forall n \in \Z, \Delta_n:=\det(c_{i-j})_{i,j=0}^n \neq 0 (>0)$. Some aspects on quasi-definite functionals and their perturbations are studied in \cite{alvarez,marcellan}.

It is known \cite{geronimus-2} that when the linear functional $\L$ is Hermitian and positive definite there exist a finite positive  Borel measure with a support lying on $\T$ such that $\L[f]=\int_{\T} f \d \mu$, $\forall f\in \Lambda_{[\infty]}$. In addition, an Hermitian positive definite linear functional $\L$ defines a sesquilinear form $\langle {\cdot},{\cdot}\rangle_{\L}: \Lambda_{[\infty]} \times \Lambda_{[\infty]} \mapsto \C$ as $\langle f,g\rangle_{\L}=\L[f \bar g]$, $\forall f ,g \in \Lambda_{[\infty]}$. Two Laurent polynomials $\{f,g\}\subset\Lambda_{[\infty]}$ are said to be orthogonal with respect to $\L$ if $\langle f,g \rangle_{\L}=0$. From the properties of $\L$ is easy to see that $\langle {\cdot},{\cdot} \rangle_{\L} $ is a scalar product and if $\mu$ is the positive finite Borel measure associated to $\L$ we are lead to the corresponding Hilbert space $L^2(\T,\mu)$, the closure of $\Lambda_{[\infty]}$. As we mentioned before $\{P_l\}_{l=0}^\infty$ denotes the set of monic
orthogonal polynomials, $\deg P_l\leq l$, for a positive measure $\mu$ satisfying \eqref{szego} and therefore $\{P_l\}_{l=0}^q$ is an orthogonal basis of the space of truncated polynomials $\Lambda_{[0,q]}$.

In this work we allow for a  more general setting assuming that $\L$ is just quasi-definite, which is associated to a corresponding quasi-definite complex measure $\mu$, see \cite{gautschi}. As  before a sesquilinear form $\langle {\cdot},{\cdot}\rangle_{\L}$ is defined for any such linear functional $\L$; thus, we just have the linearity (in the first entry) and skew-linearity (in the second entry) properties. However,  we have no symmetry allowing the interchange of the two arguments. We formally broaden the notion of orthogonality and say that $f$ is orthogonal to $g$ if $\langle f,g \rangle_{\L}=0$, but we must be careful as in this more general situation it could  happen that $ \langle f,g \rangle_{\L}=0 $ but $ \langle g,f \rangle_{\L} \neq 0 $.

The layout of this paper goes as follows. In \S \ref{introCMV} we present the application of the Gauss--Borel factorization of a  CMV moment matrix to the construction of OLPUC, associated second kind functions,  5-term  recursion relations and CD formulae. In \S \ref{genCMV} we perform  a similar work using a more general sequence that the one used in \cite{CMV}. This allows us to study snake-shaped (as are denoted in \cite{Barroso-Snake}) recursion formulae. Moreover, the CD formulae we derive for these extended cases is the  kernel for the orthogonal projection to the general space of truncated Laurent polynomials, $\Lambda_{[p,q]}$ 
with $p,q\in\N$. This is a large generalization of the CMV situation as the possible spaces  of truncated Laurent polynomials is very particular, namely either $\Lambda_{[l,l]}$ or $\Lambda_{[l+1,l]}$
with $l\in\N$. To conclude, in \S \ref{Toda} we study the deformations of the moment matrices to obtain the integrable equations associated, the representation of the OLPUC and its associated second kind functions using $\tau$-functions and corresponding bilinear identities.

\section{Orthogonal Laurent polynomials in the circle, $LU$ factorization and the CMV ordering}
\label{introCMV}

In this section we  use the Gauss--Borel, also known as $LU$ or Gaussian,  factorization of an infinite dimensional matrix, that we refer to as moment matrix,  to derive bi-orthogonal Laurent polynomials in the unit circle (BOLPUC) to the given measure $\mu$ --that, as we will see, are closely related to the Szeg\H{o} polynomials--, associated second kind functions in terms of the Fourier series of the measure, their recursion relations and corresponding CD formulae very much in the spirit of \cite{afm-2}. The key idea is to use the results of   \cite{CMV} to construct a very specific moment matrix that will lead to the mentioned results.

\subsection{Biorthogonal Laurent polynomials}

We are ready to, using a CMV moment type matrix,  find a set of bi-orthogonal Laurent polynomials, its connection with Szeg\H{o} polynomials and corresponding determinantal expressions. For this aim, let us  consider the basic object fixing the CMV order of the Fourier family $\{z^j\}_{j\in\Z}$. This order allows us to work in realm of the semi-infinite matrices avoiding in this way the less convenient scenario of bi-infinite matrices, very much as in the multiple OPRL situation \cite{afm-2}.
\begin{definition}
  We denote
  \begin{align*}
\chi_{1}(z)&:=(1,0,z,0,z^2,0,\dots)^{\top},&
\chi_{2}(z)&:=(0,1,0,z,0,z^2,\dots)^{\top},&
\chi_{a}^*(z)&:=z^{-1}\chi_{a}(z^{-1}), & a=1,2.
\end{align*}
\begin{align*}
\chi_{\CMV}(z)&:=\chi_1(z)+\chi_2^*(z)=(1,z^{-1},z,z^{-2},\dots)^{\top},&
\chi_{\CMV}^*(z)&:=\chi^*_1(z)+\chi_2(z)=(z^{-1},1,z^{-2},z,\dots)^{\top}.
\end{align*}
\end{definition}

With these sequences at hand and with a given quasi-definite complex Borel measure $\mu$ supported on $\T$  we define the CMV moment matrix
\begin{definition}
 The CMV moment matrix is the following  semi-infinite complex-valued matrix\footnote{The reader should notice  that if $\mu$ is a positive measure then $g_{\CMV}$ is a definite positive Hermitian matrix; i.e.,  $g_{\CMV}={g_{\CMV}}^\dagger$.}
\begin{align}\label{def.CMV.g}
g_{\CMV}:=\oint_{\T} \chi_{\CMV}(z) \chi_{\CMV}(z)^{\dagger} \d \mu(z).
\end{align}
\end{definition}

The  $LU$ factorization, which will play an important role in what follows,  is
\begin{equation}
  g_{\CMV}=S_1^{-1} S_2,
\end{equation}
where $S_1$ is a normalized\footnote{The coefficients in the main diagonal are equal to the unity.} lower triangular matrix and $S_2$ is an upper
triangular matrix. Hence we write
\begin{align*}
S_1&=\begin{pmatrix}
  1& 0 & 0 &\dots\\
  (S_1)_{10}&1& 0 &\dots\\
  (S_1)_{20}&(S_1)_{21}&1&\\
  \vdots&\vdots& & \ddots
\end{pmatrix},&
S_2&=\begin{pmatrix}
  (S_2)_{00}&  (S_2)_{01}& (S_2)_{02} &\dots\\
   0 & (S_2)_{11} & (S_2)_{12}  &\dots\\
  0 & 0 & (S_2)_{22} &\\
  \vdots&\vdots& & \ddots
\end{pmatrix}.
\end{align*}
This Gaussian factorization for the moment matrix makes sense if all the principal minors are non-singular, which is precisely the quasi-definiteness condition for the measure $\mu$. More on the algebraic Gauss--Borel (or $LU$) factorization and it's connection with integrable systems can be read in \cite{felipe}.

With the aid of these matrices we consider the sequences
\begin{align*}
  \Phi_{1,1}&:=S_1 \chi_1,&\Phi_{1,2}&:=S_1\chi_2^*, &\Phi_{2,1}&:=(S_2^{-1})^\dagger\chi_1, &\Phi_{2,2}&:=(S_2^{-1})^\dagger\chi_2^*,
\end{align*}
which can be written as semi-infinite vectors
\begin{align*}
\Phi_{1,a}(z)&=
\begin{pmatrix}
\varphi_{1,a}^{(0)}(z)\\
\varphi_{1,a}^{(1)}(z) \\
\vdots
\end{pmatrix},&
\Phi_{2,a}(z)&=
\begin{pmatrix}
 \varphi_{2,a}^{(0)}(z)\\
 \varphi_{2,a}^{(1)}(z) \\
\vdots
\end{pmatrix},
\end{align*}
for $a=1,2$. The corresponding components $\varphi_{b,1}^{(l)}$ are polynomials of degree $l$ in variable $z$, while $\varphi_{b,2}^{(l)}$ are polynomials of degree $l+1$ in the variable $z^{-1}$ which vanish at $z=\infty$.  Inspired by the multiple orthogonal case \cite{afm-2} we define  the following sequences
of Laurent polynomials
\begin{align*}
\Phi_{1}(z):&=S_1\chi_{\CMV}(z)=\Phi_{1,1}+\Phi_{1,2}, &  \Phi_{2}(z):&=(S_2^{-1})^{\dagger} \chi_{\CMV}(z)=\Phi_{2,1}+\Phi_{2,2},
\end{align*}
which are semi-infinite vectors that we write in the form
\begin{align*}
\Phi_{1}(z)&=
\begin{pmatrix}
\varphi_{1}^{(0)}(z)\\
\varphi_{1}^{(1)}(z) \\
\vdots
\end{pmatrix},&
\Phi_{2}(z)&=
\begin{pmatrix}
 \varphi_{2}^{(0)}(z)\\
 \varphi_{2}^{(1)}(z) \\
\vdots
\end{pmatrix},
\end{align*}
with its coefficients being Laurent polynomials.

As we said in \S\ref{intro} the measure $\mu$ has an associated sesquilinear form $\langle {\cdot},{\cdot} \rangle_{\L}$ acting on any pair of Laurent polynomials in $\T$, $\varphi_1(z)$ and $ \varphi_2(z)$, as
\begin{align} \label{def.esc.prod}
\langle \varphi_1 ,\varphi_2 \rangle_{\L} := \oint_{\T} \varphi_1(z) \bar \varphi_2(z^{-1}) \d \mu(z).
\end{align}
From the definition is clear that $g$ (whose principal minors should not vanish) is the matrix associated to $\langle {\cdot},{\cdot} \rangle_{\L }$. The reader can check that the principal minors of $g$ are exactly the Toeplitz minors $\Delta_n$ (one matrix is obtained from the other using permutations). As we are going to use quasi-definite measures  the factorization condition will hold in the subsequent work.

We recall the reader that given two linear spaces $V,V'$ and a sesquilinear form
\begin{align*}
\begin{array}{cccc}
 \langle {\cdot},{\cdot} \rangle_{\L}: & V \times V' & \rightarrow & \C \\
 & (x,y) & \mapsto & \langle x,y \rangle_{\L}
\end{array}
\end{align*}
we say that sets $X \subset V$ and $Y \subset V'$ are bi-orthogonal if $\langle x,y \rangle_{\L} =0$ for all $ x \in X$ and $y \in Y$.
\begin{theorem}
The  sets of Laurent polynomials
$\{\varphi_{1}^{(l)}\}_{l=0}^\infty$ and
$\{\varphi_{2}^{(l)}\}_{l=0}^\infty$ are bi-orthogonal in the
unit circle with respect to the sesquilinear form defined in \eqref{def.esc.prod}, that is
\begin{align}\label{biorth}
\langle \varphi_1^{(l)} ,\varphi_2^{(k)} \rangle_{\L}=\oint_{\T} \varphi_{1}^{(l)}(z) \bar \varphi_{2}^{(k)}(z^{-1}) \d \mu(z) &=\delta_{l,k} &
l,k=0,1,\dots
\end{align}
\end{theorem}
\begin{proof}
We compute
\begin{align*}
\oint_{\T} \Phi_{1}(z) \bar \Phi_{2} (z^{-1})^{\top} \d \mu(z) &= \oint_{\T} \Phi_{1}(z) \Phi_{2}(z)^{\dagger} \d \mu(z)\\&=S_1\Big[\oint_{\T} \chi_{\CMV}(z) \chi_{\CMV}(z)^{\dagger} \d \mu(z) \Big] S_2^{-1}\\&=\I.
\end{align*}
\end{proof}
Orthogonality relations \eqref{biorth} can  alternatively   be expressed as follows
\begin{align}\label{orth}
\begin{aligned}
\langle \varphi_1^{(2l)} ,z^k \rangle_{\L}&=\oint_{\T} \varphi_{1}^{(2l)}(z) z^{-k} \d \mu(z)=0, &&k=-l,\dots,l-1, \\
\langle \varphi_1^{(2l+1)} ,z^k \rangle_{\L}&=\oint_{\T} \varphi_{1}^{(2l+1)}(z) z^{-k} \d \mu(z)=0,& &k=-l,\dots,l, \\
\langle z^k,\varphi_2^{(2l)}  \rangle_{\L}&=\oint_{\T} \bar \varphi_{2}^{(2l)}(z^{-1}) z^{k} \d \mu(z)=0,& &k=-l,\dots,l-1, \\
\langle z^k, \varphi_2^{(2l+1)} \rangle_{\L}&=\oint_{\T} \bar \varphi_{2}^{(2l+1)}(z^{-1}) z^{k} \d \mu(z)=0,& &k=-l,\dots,l.
\end{aligned}
\end{align}

\begin{pro}\label{pro.1}
  Given a positive definite measure $\mu$ there exists  $h_l\in\R$, $l=0,1,\dots,$ such that
  \begin{align*}
    \varphi^{(l)}_{2}=h_l^{-1}\varphi_{1}^{(l)}
  \end{align*}
  \end{pro}
\begin{proof}See Appendix \ref{proofs}.
 \end{proof}

Therefore, are proportional Laurent polynomials and  bi-orthogonality \eqref{biorth} implies that
$\{\varphi^{(l)}_{1}\}_{l=0}^\infty$ and $\{\varphi^{(l)}_{2}\}_{l=0}^\infty$ are sets of orthogonal
Laurent polynomials for the positive measure $\mu$, that is
\begin{align}\label{orthogonality}
\begin{aligned}
\langle \varphi_1^{(l)} ,\varphi_1^{(k)} \rangle_{\L}=\oint_{\T} \varphi_{1}^{(l)}(z) \bar \varphi_{1}^{(k)}(z^{-1}) \d \mu(z) &=\delta_{l,k}h_l &
l,k=0,1,\dots \\
\langle \varphi_2^{(l)} ,\varphi_2^{(k)} \rangle_{\L}=\oint_{\T} \varphi_{2}^{(l)}(z) \bar \varphi_{2}^{(k)}(z^{-1}) \d \mu(z) &=\delta_{l,k}h_l^{-1} &
l,k=0,1,\dots
\end{aligned}
\end{align}

We are now ready to write the relationship between these CMV Laurent polynomials and the
Szeg\H{o} polynomials $P_l$ introduced previously.
\begin{pro}\label{ident.CMV.szego}
  If the measure $\mu$ is positive  definite we have the following identifications between the CMV Laurent polynomials, the Szeg\H{o} polynomials and their reciprocals
  \begin{align}\begin{aligned}\label{rec.relations}
\varphi_{1}^{(2l)}(z)&= z^{-l}P_{2l}(z),&
\varphi_{1}^{(2l+1)}(z)&=z^{-l-1}P_{2l+1}^*(z).\end{aligned}
  \end{align}
\end{pro}
\begin{proof}
  To check it just observe that
\begin{align*}
\oint_{\T} z^l \varphi_{1}^{(2l)}(z) z^{-k} \d \mu(z)&=0, & k=0,\dots,2l-1.
\end{align*}
Hence, $z^l\varphi_{1}^{(2l)}(z)$ has the same orthogonality
relations that $P_{2l}$ and both
 are monic polynomials of degree $2l$; uniqueness leads to their identification. In a similar way we proceed for the odd polynomials. Indeed,
\begin{align*}
\oint_{\T} z^{l+1} \varphi_{1}^{(2l+1)}(z) z^{-k} \d \mu(z)&=0,
& k=1,\dots,2l+1,
\end{align*}
that is, $z^{l+1} \varphi_{1}^{(2l+1)}(z)$ has the same
orthogonality relations that the polynomial $P^*_{2l+1}$ (that makes them proportional) and both
are equal to 1 at $z=0$, consequently they are the same.
\end{proof}
Using the Verblunsky coefficients\footnote{See the Introduction. } we can write
  \begin{align*}
    \varphi_{1}^{(2l)}(z)&=\alpha_{2l} z^{-l}+\cdots+z^l,\\
    \varphi_{1}^{(2l+1)}(z)&= z^{-l-1}+\cdots+\bar \alpha_{2l+1}z^l.
  \end{align*}
For later use, and in addition to the reflection coefficients, it is also useful to define the sequence $\rho_l:=\sqrt{1-|\alpha_l|^2}$, related to $\{h_l\}_{l=0}^{\infty}$ by
\begin{align*}
\rho_l^2=\frac{h_l}{h_{l-1}},
\end{align*}
valid for $l>0$, with  $\rho_0=0$.

In the general quasi-definite case, we can perform a very similar construction. We write
\begin{align}&\begin{aligned}
    \varphi_{1}^{(2l)}(z)&=\alpha_{2l}^{(1)} z^{-l}+\cdots+z^l,\\
    \varphi_{1}^{(2l+1)}(z)&= z^{-l-1}+\cdots+\bar \alpha_{2l+1}^{(2)} z^l, \end{aligned}\\
& \begin{aligned}
    \varphi_{2}^{(2l)}(z)&=\bar h_{2l}^{-1} \alpha_{2l}^{(2)} z^{-l}+\cdots+\bar h_{2l}^{-1}z^l,\\
    \varphi_{2}^{(2l+1)}(z)&= \bar h_{2l+1}^{-1}z^{-l-1}+\cdots+\bar h_{2l+1}^{-1} \bar \alpha_{2l+1}^{(1)} z^l,\end{aligned}
  \end{align}
 and also $\rho_0^2:=0,
  \rho_l^2:=\frac{h_l}{h_{l-1}}, l=1,2,\dots$. The Hermitian case can be considered a particular reduction where $\alpha_l^{(2)}=\alpha_l^{(1)}$. The reason for the notation stands in the following fact. Given a quasi-definite measure $\mu$ we can find two families of monic orthogonal polynomials such that
  \begin{align}\begin{aligned}
  \int_{\T}P_l^{(1)}(z) z^{-k} \d \mu(z) &=0& k&=0,1,\dots,l-1, \\
  \int_{\T}z^{k} \bar P_l^{(2)}(z^{-1})  \d \mu(z) &=0& k&=0,1,\dots,l-1. \end{aligned}
  \end{align}
If we call $ \alpha_l^{(1)}=P_l^{(1)}(0)$, and $\alpha_l^{(2)}=P_l^{(2)}(0)$ then these coefficients are related with the Laurent Polynomial coefficients and we can obtain the quasi-definite version of Proposition \ref{ident.CMV.szego}, that is
 \begin{align}\begin{aligned}\label{quasi.rec.relations}
\varphi_{1}^{(2l)}(z)&= z^{-l}P_{2l}^{(1)}(z), &
\varphi_{1}^{(2l+1)}(z)&=z^{-l-1}P_{2l+1}^{(2)*}(z),\\
\varphi_{2}^{(2l)}(z)&= \bar h_{2l}^{-1} z^{-l}P_{2l}^{(2)}(z),&
\varphi_{2}^{(2l+1)}(z)&=\bar h_{2l+1}^{-1} z^{-l-1}P_{2l+1}^{(1)*}(z).
\end{aligned}
  \end{align}
In what follows  $g^{[l]}:=\sum_{i,j=0}^{l-1}g_{i,j}E_{i,j}$ denotes the $l\times l$ truncated moment matrix and $\chi^{[l]}$ is the truncated vector consisting of the $l$ first components of $\chi$. With this notation we can express the bi-orthogonal Laurent polynomials in different ways:
\begin{pro}\label{pro.det.lp}
The following expressions hold true
 \begin{align}\label{syslaurant1}
  \varphi_{1}^{(l)}(z)&=\chi^{(l)}-\begin{pmatrix}
    g_{l,0}&g_{l,1}&\cdots &g_{l,l-1}
  \end{pmatrix}(g^{[l]})^{-1}\chi^{[l]}\\ \label{syslaurant2}
  &= (S_2)_{ll}\begin{pmatrix}
    0 &0 &\dots &0& 1
  \end{pmatrix}  (g^{[l+1]})^{-1}\chi^{[l+1]}\\
    \label{detlaurant} &=\frac{1}{\det g^{[l]}}\det
  \left(\begin{BMAT}{cccc|c}{cccc}
    g_{0,0}&g_{0,1}&\cdots&g_{0,l-1}&\chi^{(0)}(z)\\
     g_{1,0}&g_{1,1}&\cdots&g_{1,l-1}&\chi^{(1)}(z)\\
     \vdots &\vdots&            &\vdots&\vdots\\
          g_{l,0}&g_{l,1}&\cdots&g_{l,l-1}&\chi^{(l)}(z)
\end{BMAT}\right),& l&\geq 1,
\end{align}
and
\begin{align}\label{sysduallaurant1}
  \bar \varphi_2^{(l)}(\bar z)&=
  (S_{2})_{ll}^{-1}\Big((\chi^{(l)})^{\dag}-(\chi^{[l]})^\dag(g^{[l]})^{-1}\begin{pmatrix}
    g_{0,l}\\g_{1,l}\\\vdots\\g_{l-1,l}
  \end{pmatrix}\Big)  \\ \label{sysduallaurant2}
  &=(\chi^{[l+1]})^\dag
  (g^{[l+1]})^{-1}\begin{pmatrix}
    0 \\0\\\vdots \\0\\ 1
  \end{pmatrix}\\
\label{detduallaurant} &=\frac{1}{\det g^{[l+1]}}\det
  \left(\begin{BMAT}{cccc}{cccc|c}
    g_{0,0}&g_{0,1}&\cdots& g_{0,l}\\
     g_{1,0}&g_{1,1}&\cdots& g_{1,l} \\
     \vdots &\vdots&   &\vdots\\
        g_{l-1,0}&g_{l-1,1}&\cdots&g _{l-1,l}\\
        \overline{\chi^{(0)}(z)}& \overline{\chi^{(1)}(z)} & \dots & \overline{\chi^{(l)}(z)}
        \end{BMAT}\right), & l\geq 1.
\end{align}
\end{pro}
\begin{proof}
See Appendix \ref{proofs}.

\end{proof}

Similar expressions hold for $\varphi_{a,1}^{(l)}$ replacing $\chi$ by $\chi_1$, $a=1,2$, and for $\varphi_{a,2}^{(l)}$ replacing $\chi$ by $\chi_2^*$, $a=1,2$.

\subsection{Second kind functions}

In this subsection we introduce the second kind functions associated with the orthogonal Laurent polynomials discussed before. First we present  determinantal expressions, then we connect them with the Fourier series of the measure and also with corresponding Cauchy/Caratheorody transforms.

\begin{definition}
The partial second kind  sequences are
given by
\begin{align*}
  C_{1,1}(z)&:=(S_1^{-1})^\dagger \chi^*_1(z), & C_{1,2}(z)&:=(S_1^{-1})^\dagger \chi_2(z),&
  C_{2,1}(z)&:=S_2\chi^*_1(z),& C_{2,2}(z)&:=S_2\chi_2(z).
\end{align*}
and the second kind sequences
\begin{align*}
  C_{1}(z)&:=(S_1^{-1})^\dagger \chi^*(z), &
  C_{2}(z)&:=S_2\chi^*(z).
\end{align*}
\end{definition}
Observe that
\begin{align}\label{C1221}
  C_1&=C_{1,1}+C_{1,2},&
  C_2&=C_{2,1}+C_{2,2}.
\end{align}
We have the expressions as semi-infinite vectors, being its coefficients what we call second kind functions
\begin{align*}
C_{1,a}(z)&:=
\begin{pmatrix}
C_{1,a}^{(0)}(z)\\
C_{1,a}^{(1)}(z) \\
\vdots
\end{pmatrix},&
C_{2,a}(z)&:=
\begin{pmatrix}
 C_{2,a}^{(0)}(z)\\
 C_{2,a}^{(1)}(z) \\
\vdots
\end{pmatrix},&
C_{1}(z)&:=\begin{pmatrix}
C_{1}^{(0)}(z)\\
C_{1}^{(1)}(z) \\
\vdots
\end{pmatrix},&
C_{2}(z)&:=
\begin{pmatrix}
 C_{2}^{(0)}(z)\\
 C_{2}^{(1)}(z) \\
\vdots
\end{pmatrix}.
\end{align*}

The expressions just given for the second kind functions are, in principle,  formal, as the matrix products lead to series, not necessarily convergent, instead of finite sums. In fact (as we will show  in Proposition \ref{proC}) they are well defined in terms of the bi-orthogonal Laurent polynomials and truncated Fourier series of the measure, with convergence ensured in some annulus centered at the origin of the complex plane. The coefficients of these sequences are called second kind functions.

One can find analogous determinantal type expressions as those for the OLPUC given in Proposition \ref{pro.det.lp}, if we define
\begin{align*}
  \Gamma_{1,j}^{(l)}&:=\sum_{k \geq l} g_{jk}\chi^{*(k)},&\Gamma_{2,j}^{(l)}&:=\sum_{k \geq l} g^{\dag}_{jk}\chi^{*(k)}.
\end{align*}
\begin{pro}\label{det.Cauchy}
The second kind functions have the following determinantal expressions for $l \geq 1$
\begin{align}
\overline{ C_1^{(l)}(z)}=\frac{1}{\det g^{[l+1]}}\det
  \left(\begin{BMAT}{cccc}{cccc|c}
    g_{0,0}&g_{0,1}&\cdots& g_{0,l}\\
     g_{1,0}&g_{1,1}&\cdots& g_{1,l} \\
     \vdots &\vdots&   &\vdots\\
        g_{l-1,0}&g_{l-1,1}&\cdots&g _{l-1,l}\\
        \overline{ \Gamma_{2,0}^{(l)}(z)}& \overline{ \Gamma_{2,1}^{(l)}(z)}& \dots & \overline{\Gamma_{2,l}^{(l)}(z)}
\end{BMAT}\right),
\end{align}
and
\begin{align}
C_2^{(l)}(z)=\frac{1}{\det g^{[l]}}\det
  \left(\begin{BMAT}{cccc|c}{cccc}
    g_{0,0}&g_{0,1}&\cdots&g_{0,l-1}&\Gamma_{1,0}^{(l)}(z)\\
     g_{1,0}&g_{1,1}&\cdots&g_{1,l-1}&\Gamma_{1,1}^{(l)}(z)\\
     \vdots &\vdots&            &\vdots&\vdots\\
          g_{l,0}&g_{l,1}&\cdots&g_{l,l-1}&\Gamma_{1,l}^{(l)}(z)
\end{BMAT}\right).
\end{align}
\end{pro}
\begin{proof}
See Appendix \ref{proofs}.
\end{proof}

Now, we introduce the following definition
\begin{definition}
  \begin{enumerate}
    \item For the  bi-orthogonal Laurent polynomials $\varphi^{(l)}_1$ and $\varphi^{(l)}_2$ we use the notation $\varphi^{(l)}_2(\Exp{\im \theta})=\sum_{|k|\ll\infty}\varphi^{(l)}_{2,k}\Exp{\im k\theta}$ and $\varphi^{(l)}_1(\Exp{\im \theta})=\sum_{|k|\ll\infty}\varphi^{(l)}_{1,k}\Exp{\im k\theta}$,\footnote{ Here $\sum_{|k|\ll\infty}$ is used just to indicate that the sum is finite.}
        \item Let
 \begin{align*}
  c_n&=\frac{1}{2\pi}\int_0^{2\pi}\Exp{-\im n\theta}\d\mu(\theta),&F_\mu(u)&=\sum_{n=-\infty}^\infty c_n u^n, &
\end{align*}
 be the Fourier coefficients and the Fourier series  of the measure $\mu$.
 \item
 For each integer $k$ we introduce the following truncated Fourier series
 \begin{align*}
   F^{(+)}_{\mu,k}(z) &:=\sum_{n\geq -k} c_{n}z^{n},&
  F^{(-)}_{\mu,k}(z) &:=\sum_{n< -k} c_{n}z^{n}.
 \end{align*}
  \end{enumerate}
\end{definition}
\paragraph{Observations}
\begin{enumerate}
\item It holds that  $c_n(\bar\mu)=\overline{c_{-n}(\mu)}$, hence $\overline{c_{-n}}=c_n$ for real measures. Consequently,
\begin{align*}
  F^{(+)}_{\bar\mu,k}(z)&=\bar F^{(-)}_{\mu,-k-1}(z^{-1}) , & F^{(-)}_{\bar\mu,k}(z)&=\bar F^{(+)}_{\mu,-k-1}(z^{-1}) ,& F_{\bar \mu}(z)&=\bar F_\mu(z^{-1}).
\end{align*}
  \item The Fourier series always converges in $\mathcal D'(\T)$, the space of distributions in the circle, so that $\int_0^{2\pi} F_\mu (\theta) f(\theta)\d\theta=\int_{0}^{2\pi}f(\theta)\d\mu(\theta)$, $\forall f\in\mathcal D(\T)$, here $\mathcal D(\T)$ denotes the linear space of test functions on the circle. For an absolutely continuous measure $\d\mu(\theta)=w(\theta)\d\theta$ we can write $\d\mu(\theta)=F_\mu(\theta)\d\theta$.
      \item We will also consider the Laurent series $F_\mu(z)=\sum_{n=-\infty}^\infty c_nz^n$, for $z\in\C$. Notice that $ F^{(+)}_{\mu,k}+F^{(-)}_{\mu,k}=F_\mu$.
      \item Let $D(0;r,R)=\{z\in\C: r<|z|<R\}$ denote the annulus around $z=0$ with interior and exterior radii $r$ and $R$ and $R_\pm:=(\limsup_{n\to\infty}\sqrt[n]{|c_{\pm n}|})^{\mp 1}$. Then, according to the Cauchy--Hadamard theorem, we have
      \begin{itemize}
        \item The series $F^{(+)}_{\mu,k}(z)$ converges uniformly in any compact set $K$, $K\subset D(0;0,R_+)$.
          \item The series $F^{(-)}_{\mu,k}(z)$ converges uniformly in any compact set $K$, $K\subset D(0;R_-,\infty)$.
            \item The series $F_{\mu}(z)$ converges uniformly in any compact set $K$, $K\subset D(0;R_-,R_+)$.
        \end{itemize}
        \item According to F. and M. Riesz theorem if $c_n=0$, $n<0$, then $\mu$ is absolutely continuous with respect to the Lebesgue measure on $\T$ \cite{riesz}. In fact if $f(z)=\frac{1}{2\pi}\int_0^{2\pi} \frac{\d\mu(\theta)}{1-z\Exp{-\im\theta}}$, then $w(\theta)=\lim_{r\to1 }f(r\Exp{\im \theta})\in L^1(\T)$ and $\d\mu=w(\theta)\d\theta$; therefore, in this case we have $f(z)=\frac{1}{2\pi\im}\int_\T \frac{f(u)}{u-z}\d u$, a holomorphic function in $\mathbb D$. Moreover, the set of $w(\theta)=f(\Exp{\im\theta})\in L^2(\T)$ with $c_n=0$, $n<0$, is isometric to the set $H^2$ of holomorphic functions in $\mathbb D$ with limits when $r\to 1$ in $L^2(\T)$, observe that $w=\sum_{n=0}^\infty c_n\Exp{n\im\theta}=F_\mu$.
\end{enumerate}
\begin{pro}\label{proC}
 The partial second kind functions can be expressed as
 \begin{align*}
    C_{1,1}^{(l)}&= 2\pi \sum_{|k|\ll\infty}\varphi^{(l)}_{2,k}z^{-k-1}\bar F^{(-)}_{\mu,-k-1}(z),& R_-&<|z|<\infty,&
     C_{1,2}^{(l)}&=  2\pi \sum_{|k|\ll\infty}\varphi^{(l)}_{2,k}z^{-k-1}  \bar F^{(+)}_{\mu,-k-1}(z),& 0&<|z|<R_+,\\
     C_{2,1}^{(l)}&= 2\pi \sum_{|k|\ll\infty}\varphi^{(l)}_{1,k}z^{-k-1}F^{(+)}_{\mu,k}(z^{-1}),&R_+^{-1}&<|z|<\infty,&
     C_{2,2}^{(l)}&=  2\pi \sum_{|k|\ll\infty}\varphi^{(l)}_{1,k}z^{-k-1}  F^{(-)}_{\mu,k}(z^{-1}),&0&<|z|<R_-^{-1},
  \end{align*}
  and the second kind functions as
  \begin{align}\label{skf}
    C_1^{(l)}&=2\pi\varphi^{(l)}_2(z^{-1})z^{-1}\bar F_{\mu}(z), & R_-&<|z|<R_+,&C_2^{(l)}&=
     2\pi\varphi^{(l)}_1(z^{-1})z^{-1}F_{\mu}(z^{-1}), & R_+^{-1}&<|z|<R_-^{-1}.
  \end{align}
\end{pro}
\begin{proof}
From the formal definition of $C_{a,b}$, $a,b=1,2$, and the aid of the Gaussian factorization of the moment matrix $g$,  we have
  \begin{align*}
    C_{1,1}&=(S_2^{-1})^\dagger\oint_\T\chi(u)\chi(u)^\dagger\overline{\d\mu (u)}\chi_1^*(z),&
     C_{1,2}&=  (S_2^{-1})^\dagger\oint_\T\chi(u)\chi(u)^\dagger\overline{\d\mu (u)}\chi_2(z)\\
    C_{2,1}&=  S_1\oint_\T\chi(u)\chi(u)^\dagger\d\mu (u)\chi_1^*(z),&
     C_{2,2}&=   S_1\oint_\T\chi(u)\chi(u)^\dagger\d\mu (u)\chi_2(z),
  \end{align*}
We recall that $((S_2^{-1})^\dagger\chi(u))^{(l)}=\varphi^{(l)}_2(u)$ and $(S_1\chi(u))^{(l)}=\varphi^{(l)}_1(u)$ and expand the matrix products involved, without any interchange of integrals and summation symbols, to get
  \begin{align*}
    C_{1,1}^{(l)}&= \sum_{|k|\ll\infty}\varphi^{(l)}_{2,k}\Big(\sum_{n=0}^\infty\Big[\int_0^{2\pi}\Exp{\im (k-n)\theta}\d\bar\mu(\theta)\Big]z^{-n-1}\Big),&
     C_{1,2}^{(l)}&=  \sum_{|k|\ll\infty}\varphi^{(l)}_{2,k}\Big(\sum_{n=0}^\infty\Big[\int_0^{2\pi}\Exp{\im (k+n+1)\theta}\d\bar\mu(\theta)\Big]z^{n}\Big)\\
    C_{2,1}^{(l)}&=  \sum_{|k|\ll\infty}\varphi^{(l)}_{1,k}\Big(\sum_{n=0}^\infty\Big[\int_0^{2\pi}\Exp{\im (k-n)\theta}\d\mu(\theta)\Big]z^{-n-1}\Big),&
     C_{2,2}^{(l)}&=  \sum_{|k|\ll\infty}\varphi^{(l)}_{1,k}\Big(\sum_{n=0}^\infty\Big[\int_0^{2\pi}\Exp{\im (k+n+1)\theta}\d\mu(\theta)\Big]z^{n}\Big).
  \end{align*}
  Using the Fourier coefficients $c_n$ we write
    \begin{align}\label{CC}\begin{aligned}
      C_{1,1}^{(l)}&=2\pi \sum_{|k|\ll\infty}\varphi^{(l)}_{2,k}\Big(\sum_{n=0}^\infty\overline{c_{k-n}}z^{-n-1}\Big),&
     C_{1,2}^{(l)}&=2\pi   \sum_{|k|\ll\infty}\varphi^{(l)}_{2,k}\Big(\sum_{n=0}^\infty \overline{c_{k+n+1}}z^{n}\Big),\\
    C_{2,1}^{(l)}&= 2\pi  \sum_{|k|\ll\infty}\varphi^{(l)}_{1,k}\Big(\sum_{n=0}^\infty c_{n-k}z^{-n-1}\Big),&
     C_{2,2}^{(l)}&= 2\pi  \sum_{|k|\ll\infty}\varphi^{(l)}_{1,k}\Big(\sum_{n=0}^\infty c_{-k-n-1}z^{n}\Big),
    \end{aligned}
  \end{align}
  from where the desired result follows.
\end{proof}

\paragraph{Example.} Assume that \begin{align*}
      c_n=\frac{1}{2|n|!}+\delta_{n,0}\Big(\operatorname{e}+\frac{1}{2}\Big)
    \end{align*}
  so that
\begin{align*}
\d\mu(\theta)&=w(\theta)\d\theta, &  w&=\operatorname{e}+\sum_{n=0}^\infty \frac{\cos n\theta}{n!}=\operatorname{e}+\Exp{\cos \theta}\cos(\sin \theta),
\end{align*}
where the weight $w$ is a smooth positive $2\pi$-periodic function. Then
    \begin{align*}
   F^{(+)}_{\mu,k}&=\begin{cases}\operatorname{e}+
    \frac{1}{2}\Big( \sum_{n=0}^k \frac{z^{-n}}{n!}+\Exp{z}\Big),    &k\geq 0,\\
        \frac{1}{2}\Big(- \sum_{n=0}^{|k|-1}\frac{z^{n}}{n!} +\Exp{z}\Big)&k< 0,
      \end{cases}\\
 F^{(-)}_{\mu,k}&=\begin{cases}
    \frac{1}{2}\Big( -\sum_{n=0}^{k}\frac{z^{-n}}{n!} +\Exp{z^{-1}}\Big)& k\geq 0,
   \\ \operatorname{e}+ \frac{1}{2}\Big(\sum_{n=0}^{|k|-1}\frac{z^{n}}{n!}+\Exp{z^{-1}}\Big)
      &k<0,
    \end{cases}
    \end{align*}
    with annulus of convergence $D=\{z\in \C: 0<|z|<1\}$. Notice that in this example
    \begin{align*}
F_\mu(z)=\operatorname{e}+\frac{\Exp{z}+\Exp{z^{-1}}}{2}
    \end{align*}
    and therefore
    \begin{align*}
      C^{(l)}_1(z)&=2\pi z^{-1}\varphi_2^{(l)}(z^{-1})F_\mu(z),&
      C^{(l)}_2(z)&=2\pi z^{-1}\varphi_1^{(l)}(z^{-1})F_\mu(z).
    \end{align*}
\begin{pro}\label{pro.conv.gamma}
The formal series for $\Gamma_{a,j}^{l}(z)$ (where $a=1,2$) defined in Proposition \ref{det.Cauchy} can be expressed in terms of the Fourier series of $\mu$ and consequently are convergent in corresponding annulus in the complex plane. More precisely
\begin{align*}
\Gamma^{(l)}_{1,j}(z)&=2 \pi z^{-J(j)-1} \Big(F^{(+)}_{J(j)-l,\mu}(z^{-1})+F^{(-)}_{J(j)+l,\mu}(z^{-1})\Big), & R_+^{-1}&<|z|<R_-^{-1},\\
 \Gamma^{(l)}_{2,j}(z)&=2 \pi z^{-J(j)-1} \Big(\bar F^{(-)}_{l-J(j)-1,\mu}(z)+\bar F^{(+)}_{-l-J(j)-1,\mu}(z)\Big),& R_-&<|z|<R_+.
\end{align*}
where\footnote{The reader can check that $z^{J(j)}=\chi^{(j)}(z)$.} $J(j)=[(-1)^{a(j)-1}\frac{j}{2}]$, being $[p]$ the integer part of $p$.
\end{pro}
\begin{proof}
See Appendix \ref{proofs}.
\end{proof}
Notice that for $l=0$ we have\begin{align*}
\Gamma^{(0)}_{1,j}(z)&=2 \pi z^{-J(j)-1} F_{\mu}(z^{-1}), & R_+^{-1}&<|z|<R_-^{-1},\\
 \Gamma^{(0)}_{2,j}(z)&=2 \pi z^{-J(j)-1} \bar F_{\mu}(z),& R_-&<|z|<R_+.
\end{align*}

Now we will justify the name we have given to these functions and show a Cauchy integral representation of them. We will prove it in two different scenarios depending on the measure. First, with a more measure theory taste, assuming that $\mu$ is positive and using the Lebesgue dominated convergence theorem. Second, with a more complex analysis taste, considering absolutely continuous complex measures  $\d\mu=w(\theta)\d\theta$ when $w$ is a continuous complex function.
\begin{theorem}\label{pro:cauchy}
Assume a positive measure $\d\mu(\theta)$ or a complex measure $\d\mu(\theta)=w(\theta)\d\theta$ with $w$ a continuous function. Then, the second kind functions can be written as the following Cauchy integrals
\begin{align*}
 C_{1,1}^{(l)}&=z^{-1}\oint_\T\frac{u\varphi_2^{(l)}(u)}{u-z^{-1}}\overline{\d\mu (u)},
 &C_{2,1}^{(l)}&=z^{-1}\oint_\T\frac{u\varphi_1^{(l)}(u)}{u-z^{-1}}\d\mu (u),& |z|&>1,\\
  C_{1,2}^{(l)}&=-z^{-1}\oint_\T\frac{u\varphi_2^{(l)}(u)}{u-z^{-1}}\overline{\d\mu (u)},&
  C_{2,2}^{(l)}&=-z^{-1}\oint_\T\frac{u\varphi_1^{(l)}(u)}{u- z^{-1}}\d\mu (u),& |z|&<1.
\end{align*}
\end{theorem}
\begin{proof}
  From the  definition of $C_{a,b}$ and the aid of the Gaussian factorization of the moment matrix $g$, $(S_1^{-1})^\dagger= (S_2^{-1})^\dagger g^\dagger $ or $S_2=S_1g$ we get
  \begin{align*}
    C_{1,1}&=  (S_2^{-1})^\dagger\oint_\T\chi(u)\chi(u)^\dagger\overline{\d\mu (u)}\chi_1^*(z)=\sum_{n=0}^\infty\Big(\oint_\T \Phi_2(u)u^{-n}z^{-n-1}\overline{\d\mu (u)}\Big),\\
     C_{1,2}&=  (S_2^{-1})^\dagger\oint_\T\chi(u)\chi(u)^\dagger\overline{\d\mu (u)}\chi_2(z)=\sum_{n=0}^\infty\Big(\oint_\T \Phi_2(u)u^{n+1}z^{n}\overline{\d\mu (u)}\Big),\\
  C_{2,1}&=  S_1\oint_\T\chi(u)\chi(u)^\dagger\d\mu (u)\chi_1^*(z)= \sum_{n=0}^\infty\Big(\oint_\T \Phi_1(u)u^{-n}z^{-n-1}{\d\mu (u)}\Big),\\
     C_{2,2}&=   S_1\oint_\T\chi(u)\chi(u)^\dagger\d\mu (u)\chi_2(z)=\sum_{n=0}^\infty\Big(\oint_\T \Phi_1(u)u^{n+1}z^{n}{\d\mu (u)}\Big).\end{align*}
  For the series in these expressions we have:
   \begin{enumerate}
    \item The series  $\sum_{n=0}^\infty u^{-n}z^{-n-1}$ converges uniformly in the $u$ variable in any compact set $K\subset\{u\in\C: |u|>|z|^{-1}\}$ to $(z-u^{-1})^{-1}$ and if $|z|>1$ then we can take $K$ such that  $\T\subset K$.
    \item The series  $\sum_{n=0}^\infty u^{n+1}z^{n}$ converges uniformly in the $u$ variable in any compact set $K\subset\{u\in\C: |u|<|z|^{-1}\}$ to $-(z-u^{-1})^{-1}$ and if $|z|<1$  then we can take $K$ such that    $\T\subset K$.
  \end{enumerate}
 Let us assume a positive measure. The corresponding $m$-th partial sums are
  \begin{align*}
   \sum_{n=0}^m u^{-n}z^{-n-1}&=z^{-1}\frac{1-(uz)^{-m-1}}{1-(uz)^{-1}},&
 \sum_{n=0}^m u^{n+1}z^{n}&=u\frac{1-(uz)^{m+1}}{1-(uz)}.
  \end{align*}
 If we write $u=\Exp{\im\theta}$ and $z=|z|\Exp{\im \arg z}$  we have
  \begin{align*}
    \Big|\frac{1-(uz)^{-m-1}}{1-(uz)^{-1}}\Big|^2=\frac{1-2|z|^{-(m+1)}\cos((m+1)(\theta+\arg z))+|z|^{-2(m+1)}}{1-2|z|^{-1}\cos(\theta+\arg z)+|z|^{-2}}.
  \end{align*}
  For $|z|^{-1}<1$ we have the following inequalities
  \begin{align*}
0<1-2|z|^{-(m+1)}\cos((m+1)(\theta+\arg z))+|z|^{-2(m+1)}&\leq (1+|z|^{-(m+1)})^2<4,\\
0<1-2|z|^{-1}\cos(\theta+\arg z)+|z|^{-2})&\geq (1-|z|^{-1})^2,
  \end{align*}
  so that, for $u\in\T$, we infer
    \begin{align*}
    \Big|z^{-1}\frac{1-(uz)^{-m-1}}{1-(uz)^{-1}}\Big|&<\frac{2|z|^{-1}}{1-|z|^{-1}}, & |z|>1.
  \end{align*}
  Similarly, we conclude that
   \begin{align*}
    \Big|u\frac{1-(uz)^{m+1}}{1-uz}\Big|&<\frac{2}{1-|z|}, & |z|<1.
  \end{align*}
  Thus, for $u\in\T$ and $a=1,2$, we have the control bounds
  \begin{align*}
  \Big|  \sum_{n=0}^m \varphi_a^{(l)}(u)u^{-n}z^{-n-1}\Big|&<\Big|\frac{2}{1-|z|}\varphi_a^{(l)}(u)\Big|, & |z|>1,\\
  \Big|  \sum_{n=0}^m \varphi_a^{(l)}(u)u^{n+1}z^{n}\Big|&<\Big|\frac{2}{1-|z|}\varphi_a^{(l)}(u)\Big|, & |z|<1.
  \end{align*}
  Consequently, as the Laurent polynomials  $\varphi_a$ are measurable functions in $\T$, the Lebesgue dominated convergence theorem leads to the sated result.

 Finally, if we assume that $\d\mu=w(\theta)\d \theta$, with $w$ a continuous complex function; we can always write $w(\theta)\d\theta=F(u)\frac{\d u}{\im u}$, $u=\Exp{\im \theta}$, with $F$ a continuous function on $\T$. Then, recalling the uniform convergence of the geometrical series involved and the fact that the Laurent polynomials $\varphi_a$ are continuous functions on $\T$, we can interchange integral and series symbols arriving to the expressions
  \begin{enumerate}
    \item $\int_\T\Big(\sum_{n=0}^\infty \Phi_a(u)u^{-n}z^{-n-1}\Big)\d \mu(u)=\sum_{n=0}^\infty \Big(\int_\T \Phi_a(u)u^{-n}z^{-n-1})\d \mu(u)\Big)$ for $|z|>1$.
    \item $\int_\T\Big(\sum_{n=0}^\infty  \Phi_a(u)u^{n+1}z^{n}\Big)\d \mu(u)=\sum_{n=0}^\infty \Big(\int_\T  \Phi_a(u)u^{n+1}z^{n})\d \mu(u)\Big)$for $|z|<1$.
  \end{enumerate}

\end{proof}

The result motivates the name given to these functions \cite{marcellan2}. These expressions can be also written as Geronimus transforms, for that aim we just need to recall that for $u\in\T$
\begin{align*}
  \frac{1}{z- u^{-1}}=\frac{1}{2z}\Big(1+\frac{u+z^{-1}}{u-z^{-1}}\Big),
\end{align*}
and therefore for $l \geq 1$
\begin{align*}
 C_{1,1}^{(l)}&=\frac{1}{2z}\oint_\T\frac{u+z^{-1}}{u-z^{-1}}\varphi_2^{(l)}(u)\overline{\d\mu (u)},&C_{2,1}^{(l)}&=\frac{1}{2z}\oint_\T\frac{u+z^{-1}}{u-z^{-1}}\varphi_1^{(l)}(u)\d\mu (u),& |z|&>1,\\
  C_{1,2}^{(l)}&=-\frac{1}{2z}\oint_\T\frac{u+z^{-1}}{u-z^{-1}}\varphi_2^{(l)}(u)\overline{\d\mu (u)},&
  C_{2,2}^{(l)}&=-\frac{1}{2z}\oint_\T\frac{u+z^{-1}}{u-z^{-1}}\varphi_1^{(l)}(u)\d\mu (u),& |z|&<1.
\end{align*}
For $l=0$, we just obtain (up to constants)
\begin{align*}
  \frac{1}{2z}(|\mu|+\mathcal C(z^{-1})),
\end{align*}
where $\mathcal C(z)$ is the Carath\'{e}odory transform of the measure:
\begin{align*}
  \mathcal C(z):=\int_0
  ^{2\pi}\frac{\Exp{\im\theta}+z}{\Exp{\im\theta}-z}\d\mu(\theta).
\end{align*}
\paragraph{Example.} Let us assume that for $|z|>1$ we have
\begin{align}\label{FF}
 C_{2,1}^{(l)}= z^{-1}\oint_\T\frac{u\varphi_1^{(l)}(u)}{u-z^{-1}}F(u)\frac{\d u}{\im u},
\end{align}
and that the only singularity of $F$ inside $\mathbb D$  can lay exclusively  at $z=0$. On the one hand, as $F$ has a Laurent expansion of the form $F(z)=\sum_{k\in\Z} F_k z^k$ and there no singularities different from $z=0$ we conclude that the coefficients of the Laurent series are just the Fourier coefficients of the measure\footnote{This property does not hold if $F$ has different singularities from $z=0$ inside $\mathbb D$, as we can not choose  $\T$ as the integration circuit to apply the Cauchy formula for the coefficients of the Laurent series.}
\begin{align*}
   F_k=\frac{1}{2\pi\im}\int_\T\frac{F(u)}{u^{k+1}}\d u=\frac{1}{2\pi}\int_0^{2\pi}F(\exp{\im\theta})\Exp{-\im k\theta}\d\theta=c_k.
\end{align*}
 On the other hand, as the Laurent polynomial $\varphi_1$ has its singularities at $z=0$, and $z^{-1}\in\mathbb D$, then the set of singularities of the integrand that lay in $\mathbb D$  in \eqref{FF} is $\{0,z^{-1}\}$. Hence,  the residue theorem gives
\begin{align*}
  C_{2,1}^{(l)}=2\pi z^{-1}\Big(\text{Res}_{u=z^{-1}}\Big(\frac{\varphi_1^{(l)}(u)F(u)}{u-z^{-1}}\Big)+\text{Res}_{u=0}\Big(\frac{\varphi_1^{(l)}(u)F(u)}{u-z^{-1}}\Big)\Big).
\end{align*}
We now evaluate the residues, the first one may be computed noticing that $\varphi_1^{(l)}(u)F(u)$ is analytic at $u=z^{-1}$
\begin{align*}
  2\pi z^{-1}\text{Res}_{u=z^{-1}}\Big(\frac{\varphi_1^{(l)}(u)F(u)}{u-z^{-1}}\Big)&=\varphi_1^{(l)}(z^{-1})F(z^{-1})=C_{1}^{(l)},
\end{align*}
where we have used Proposition \ref{proC}. For the residue at $u=0$ we use the Laurent series around this point
\begin{align*}
  \varphi_1^{(l)}(u)&=\sum_k\varphi_{1,k}^{(l)}u^k, & F(u)&=\sum_j c_j u^j, & \frac{1}{u-z^{-1}}=-\sum_{n=0}^\infty z^{n+1} u^n,
\end{align*}
so that
\begin{align*}
  2\pi z^{-1}\text{Res}_{u=0}\Big(\frac{\varphi_1^{(l)}(u)F(u)}{u-z^{-1}}\Big)=-\sum_{\substack{k+j+n=-1\\n\geq 0}}\varphi_{1,k}^{(l)}c_jz^{n+1}=-
  \sum_{k}\varphi_{1,k}^{(l)}\sum_{n=0}^\infty c_{-k-n-1}z^{n+1}=-C_{2,2}^{(l)}
\end{align*}
and we finally get
\begin{align*}
  C_{2,1}&=C_2-C_{2,2},
\end{align*}
in agreement with \eqref{C1221}.

The application of the residue theorem to the formulae in Theorem \ref{pro:cauchy} leads to expressions for the second kind functions in terms of residues:
\begin{pro}
    Let us assume  $\d\mu(\theta)=F(u)\frac{\d u}{\im u}$  then:
    \begin{enumerate}
      \item    When $F$ is an analytic function in $\C\setminus\mathbb D$ but for a set of isolated singularities, and  if denote by $\{z_{j,-}\}_{j=0}^{p_-}$ the set of distinct  points obtained from the union of $z_{0,-}=0$ and the set of singularities of $F$ at $\C\setminus\mathbb D$, then, for $z\not\in\T$,
       \begin{align*}
         C_{1,1}^{(l)}(z)&=2\pi z^{-1}\Big[\varphi^{(l)}_2(z^{-1})\bar F(z)\theta(|z|-1)-
      \operatorname{Res}_{u=0}\big(\varphi^{(l)}_2(u)\bar F(u^{-1})\big)z+\sum_{j=1}^{p_-}\Big(\frac{\varphi^{(l)}_2(\bar z_{j,-}^{-1})}{\bar z_{j,-}^{-1}-z^{-1}}
      \operatorname{Res}_{u=\bar z_{j,-}^{-1}}\bar F(u^{-1})\Big)\Big],\\
    C_{1,2}^{(l)}(z)&=2\pi z^{-1} \Big[\varphi^{(l)}_2(z^{-1})\bar F(z)\theta(1-|z|)+  \operatorname{Res}_{u=0}\big(\varphi^{(l)}_2(u)\bar F(u^{-1})\big)z-\sum_{j=1}^{p_-}\Big(\frac{\varphi^{(l)}_2(\bar z_{j,-}^{-1})}{z_{j,-}^{-1}-z^{-1}}\operatorname{Res}_{u= \bar z_{j,-}^{-1}}\bar F(u^{-1})\Big)
\Big]
        \end{align*}
       \item If $F$ is an analytic function in $\mathbb D$ but for a set of isolated singularities, and denote by $\{z_{j,+}\}_{j=0}^{p_+}$ the set of different points obtained from the union of $z_0=0$ and the set of singularities of $F$ at $\mathbb D$. Then, for $z\not\in\T$,
        \begin{align*}
      C_{2,1}^{(l)}(z)&=2\pi z^{-1}\Big[\varphi^{(l)}_1(z^{-1})F(z^{-1})\theta(|z|-1)-
      \operatorname{Res}_{u=0}\big(\varphi^{(l)}_1(u)F(u)\big)z+\sum_{j=1}^{p_+}\Big(\frac{\varphi^{(l)}_1(z_{j,+})}{z_{j,+}-z^{-1}}
      \operatorname{Res}_{u=z_j}F(u)\Big)\Big],\\
    C_{2,2}^{(l)}(z)&=2\pi z^{-1} \Big[\varphi^{(l)}_1(z^{-1})F(z^{-1})\theta(1-|z|)+  \operatorname{Res}_{u=0}\big(\varphi^{(l)}_1(u)F(u)\big)z-\sum_{j=1}^{p_+}\Big(\frac{\varphi^{(l)}_1(z_{j,+})}{z_{j,+}-z^{-1}}\operatorname{Res}_{u=z_j}F(u)\Big)
\Big].
    \end{align*}
    \end{enumerate}
\end{pro}

To conclude this section we give some summation rules that are derived using the geometrical series
\begin{pro}\label{scf}
  The OLPUC and its corresponding partial second kind functions satisfy
  \begin{align*}
    \sum_{l=0}^\infty \overline{ C_{a,1}^{(l)}}(z)\varphi_{a,1}^{(l)}(z')
&=\frac{1}{z-z'}, & |z'|>|z|,\\
  \sum_{l=0}^\infty \overline{ C_{a,2}^{(l)}}(z)\varphi_{a,2}^{(l)}(z')
&=-\frac{1}{z-z'}, & |z'|<|z|,\\    \sum_{l=0}^\infty \overline{ C_{a,1}^{(l)}}(z)\varphi_{a,2}^{(l)}(z')
=\sum_{l=0}^\infty \overline{ C_{a,2}^{(l)}}(z)\varphi_{a,1}^{(l)}(z')&=0.
\end{align*}
for $a=1,2$.
\end{pro}
\begin{proof} See Appendix \ref{proofs}.
   \end{proof}

\subsection{Recursion relations}

We are about to derive, using the Gaussian factorization, the CMV recursion relations and obtain in this way the well known CMV five diagonal Jacobi type matrix for the recursion of the Szeg\H{o} polynomials. Let us begin with the following
\begin{definition}Given the canonical basis for semi-infinite matrices $E_{i,j}, i,j \in \Z_+$, we define the projections
 \begin{align*}
   \Pi_{1}&:=\sum_{j=0}^\infty E_{2j,2j},&  \Pi_{2}&:=\sum_{j=0}^\infty E_{2j+1,2j+1},
 \end{align*}
 and the  matrices
\begin{align*}
\Lambda_{1}&:=\sum_{j=0}^{\infty}E_{2j,2+2j},&
\Lambda_{2}&:=\sum_{j=0}^{\infty} E_{1+2j,3+2j},\\
\Lambda&:=\sum_{j=0}^{\infty}E_{j,j+1},&\Upsilon&:=\Lambda_1+\Lambda_2^{\top}+E_{1,1}\Lambda^{\top}.
\end{align*}
\end{definition}
The matrix $\Upsilon$,
\begin{align*}\Upsilon=\left(
\begin{BMAT}{cc:cc:cc:cc:ccc}{cc:cc:cc:cc:ccc}
0 & 0 & 1 & 0 & 0 & 0 & 0 & 0 & 0 & 0 &\cdots  \\
1 & 0 & 0 & 0 & 0 & 0 & 0 & 0 & 0 & 0 &\cdots  \\
0 & 0 & 0 & 0 & 1 & 0 & 0 & 0 & 0 & 0 &\cdots  \\
0 & 1 & 0 & 0 & 0 & 0 & 0 & 0 & 0 & 0 &\cdots  \\
0 & 0 & 0 & 0 & 0 & 0 & 1 & 0 & 0 & 0 &\cdots  \\
0 & 0 & 0 & 1 & 0 & 0 & 0 & 0 & 0 & 0 &\cdots  \\
0 & 0 & 0 & 0 & 0 & 0 & 0 & 0 & 1 & 0 &\cdots  \\
0 & 0 & 0 & 0 & 0 & 1 & 0 & 0 & 0 & 0 &\cdots  \\
0 & 0 & 0 & 0 & 0 & 0 & 0 & 0 & 0 & 0 &\cdots  \\
0 & 0 & 0 & 0 & 0 & 0 & 0 & 1 & 0 & 0 &\cdots  \\
\vdots & \vdots & \vdots & \vdots & \vdots & \vdots &
\vdots & \vdots & \vdots & \vdots &\ddots
\end{BMAT}\right),
\end{align*}
is a central object in this paper, as its \emph{dressing} --its orbit by conjugations-- gives the pentadiagonal CMV Jacobi type matrix.

It is immediate to check that
\begin{pro}
The following relations hold
\begin{align}\label{shifts}
&\begin{aligned}
\Lambda_{1} \chi_{\CMV}(z)&=z\Pi_{1}  \chi_{\CMV}(z),&
\Lambda_{2}  \chi_{\CMV}(z)&=z^{-1}\Pi_{2}  \chi_{\CMV}(z), \\
\Lambda_{1}^{\top} \chi_{\CMV}(z)&=(z^{-1}\Pi_1-E_{0,0}\Lambda) \chi_{\CMV}(z),&
\Lambda_{2}^{\top} \chi_{\CMV}(z)&=(z\Pi_{2} -E_{1,1}\Lambda^{\top}) \chi_{\CMV}(z),
\end{aligned}\\
&\begin{aligned}
 \Upsilon \chi_{\CMV}(z)&=z\chi_{\CMV}(z),\\
  \Upsilon^{\top} \chi_{\CMV}(z)&=z^{-1}\chi_{\CMV}(z).\end{aligned}
\end{align}
\end{pro}
With the aid of these  conditions we characterize the moment matrix $g$ as verifying a symmetry constraint, which we  call \emph{string} equation, from where the recursion as well as the CD formulae will be derived. The symmetry is detailed in the following
\begin{pro}
  The CMV moment matrix fulfils the following condition
  \begin{align} \label{symz}
\Upsilon g_{\CMV}=g_{\CMV}\Upsilon.
\end{align}
\end{pro}
\begin{proof}
  To prove the previous result we proceed as follows
\begin{align*}
\Upsilon g_{\CMV}&=\oint_{\T}(\Lambda_1+\Lambda_2^{\top}+E_{1,1}\Lambda^{\top})\chi(z)\chi(z)^{\dag} \d \mu(z)=
 \oint_{\T}z\chi_{\CMV}(z)\chi_{\CMV}(z)^{\dag} \d \mu(z)\\
 &=\oint_{\T}\chi_{\CMV}(z)(z^{-1}\chi_{\CMV}(z))^{\dag} \d \mu(z)\\&=\oint_{\T}\chi_{\CMV}(z)((\Lambda_1^{\top}+\Lambda_2+E_{0,0}\Lambda)\chi_{\CMV}(z))^{\dag} \d \mu(z) \\
 &=g_{\CMV}(\Lambda_1+\Lambda_2^{\top}+\Lambda^{\top}E_{0,0})\\&=g_{\CMV}(\Lambda_1+\Lambda_2^{\top}+E_{1,1}\Lambda^{\top})\\&=g_{\CMV}\Upsilon.
\end{align*}
\end{proof}
We now proceed to dress the $\Upsilon$ matrix in two ways
\begin{definition}
We use the  notation
\begin{align*} J_{1}&:=S_1 \Upsilon S_1^{-1}, &
J_{2}&:= S_2 \Upsilon S_2^{-1}.
\end{align*}
\end{definition}
The following proposition is trivially derived from the above definition and \eqref{symz}
\begin{pro}
\begin{align*}
J_{1}= J_{2}.
\end{align*}
\end{pro}
Consequently, we introduce the CMV  Jacobi type matrix
\begin{definition}
  We define $J_{\CMV}:=J_{1}=J_{2}$.
\end{definition}

The matrix $J$ has a five diagonal structure; as easily follows when one observes that $J_1$ has zero coefficients over the third upper-diagonal and that $J_2$ has all its coefficients equal to zero under the third lower-diagonal. More specifically, the structure  is
\begin{align*}J_{\CMV}=\left(
\begin{BMAT}{cc:cc:cc:cc:ccc}{cc:cc:cc:cc:ccc}
* & * & 1 & 0 & 0 & 0 & 0 & 0 & 0 & 0 &\cdots  \\
+ & * & * & 0 & 0 & 0 & 0 & 0 & 0 & 0 &\cdots  \\
0 & * & * & * & 1 & 0 & 0 & 0 & 0 & 0 &\cdots  \\
0 & + & * & * & * & 0 & 0 & 0 & 0 & 0 &\cdots  \\
0 & 0 & 0 & * & * & * & 1 & 0 & 0 & 0 &\cdots  \\
0 & 0 & 0 & + & * & * & * & 0 & 0 & 0 &\cdots  \\
0 & 0 & 0 & 0 & 0 & * & * & * & 1 & 0 &\cdots  \\
0 & 0 & 0 & 0 & 0 & + & * & * & * & 0 &\cdots  \\
0 & 0 & 0 & 0 & 0 & 0 & 0 & * & * & * &\cdots  \\
0 & 0 & 0 & 0 & 0 & 0 & 0 & + & * & * &\cdots  \\
\vdots & \vdots & \vdots & \vdots & \vdots & \vdots &
\vdots & \vdots & \vdots & \vdots &\ddots
\end{BMAT}\right)
\end{align*}
where, $*$ is a possibly non-vanishing term and $+$ is a positive term. In fact, using the $LU$ factorization problem we are able to completely characterize  $J$ in terms of the Verblunsky coefficients.
\begin{pro}\label{explicit-J}
\begin{enumerate}
  \item The non-vanishing coefficients of $J$ are
  \begin{align*}
J_{2k,2k-1}&=-\rho^2_{2k}\alpha_{2k+1}^{(1)},&J_{2k,2k}&=-\bar \alpha_{2k}^{(2)}\alpha_{2k+1}^{(1)},&J_{2k,2k+1}&=-\alpha_{2k+2}^{(1)},&J_{2k,2k+2}&=1,\\
J_{2k+1,2k-1}&=\rho^2_{2k+1}\rho^2_{2k},&J_{2k+1,2k}&=\rho^2_{2k+1} \bar \alpha_{2k}^{(2)},&
J_{2k+1,2k+1}&=-\bar\alpha_{2k+1}^{(2)}\alpha_{2k+2}^{(1)},&
J_{2k+1,2k+2}&=\bar \alpha_{2k+1}^{(2)}.
\end{align*}
\item We have  the recursion relations
\begin{align}\label{5terms.even}
z\varphi_{1}^{(2k)} &=\varphi_{1}^{(2k+2)}-\alpha_{2k+2}^{(1)}\varphi_{1}^{(2k+1)}-\bar \alpha_{2k}^{(2)}\alpha_{2k+1}^{(1)}\varphi_{1}^{(2k)}-\rho^2_{2k}\alpha_{2k+1}^{(1)}\varphi_{1}^{(2k-1)} ,\\ \label{5terms.odd}
z\varphi_{1}^{(2k+1)}
 &=\bar \alpha_{2k+1}^{(2)}\varphi_{1}^{(2k+2)}-\bar\alpha_{2k+1}^{(2)}\alpha_{2k+2}^{(1)}\varphi_{1}^{(2k+1)}+\rho^2_{2k+1} \bar \alpha_{2k}^{(2)}\varphi_{1}^{(2k)}+\rho^2_{2k+1}\rho^2_{2k}\varphi_{1}^{(2k-1)},
\end{align}
\end{enumerate}
\end{pro}
\begin{proof}
\begin{enumerate}
  \item See Appendix \ref{proofs}.
  \item We have
   \begin{align*}
z\varphi_{1}^{(2k)}
&=J_{2k,2k+2} \varphi_{1}^{(2k+2)}+J_{2k,2k+1}\varphi_{1}^{(2k+1)}+J_{2k,2k}\varphi_{1}^{(2k)}+J_{2k,2k-1}\varphi_{1}^{(2k-1)},\\
z\varphi_{1}^{(2k+1)}
&=J_{2k+1,2k+2} \varphi_{1}^{(2k+2)}+J_{2k+1,2k+1}\varphi_{1}^{(2k+1)}+J_{2k+1,2k}\varphi_{1}^{(2k)}+J_{2k+1,2k-1}\varphi_{1}^{(2k-1)}.
\end{align*}
\end{enumerate}

\end{proof}
 As $J_{1} \Phi_{1} = S_1 \Upsilon S_1^{-1}S_1 \chi_{\CMV}(z)=z \Phi_{1}$, the sequences $\Phi_{1}, \Phi_{2}$ have a five term recurrence formula. However, although there are five non-vanishing diagonals the recurrence relations do not have more than four non-zero terms, explicitly
\begin{align}\label{5terms0}\begin{aligned}
z\varphi_{1}^{(2k)}
&=J_{2k,2k+2} \varphi_{1}^{(2k+2)}+J_{2k,2k+1}\varphi_{1}^{(2k+1)}+J_{2k,2k}\varphi_{1}^{(2k)}+J_{2k,2k-1}\varphi_{1}^{(2k-1)}\\
 &=\varphi_{1}^{(2k+2)}-\alpha_{2k+2}^{(1)}\varphi_{1}^{(2k+1)}-\bar \alpha_{2k}^{(2)}\alpha_{2k+1}^{(1)}\varphi_{1}^{(2k)}-\rho^2_{2k}\alpha_{2k+1}^{(1)}\varphi_{1}^{(2k-1)} ,\\
z\varphi_{1}^{(2k+1)}
&=J_{2k+1,2k+2} \varphi_{1}^{(2k+2)}+J_{2k+1,2k+1}\varphi_{1}^{(2k+1)}+J_{2k+1,2k}\varphi_{1}^{(2k)}+J_{2k+1,2k-1}\varphi_{1}^{(2k-1)}\\
 &=\bar \alpha_{2k+1}^{(2)}\varphi_{1}^{(2k+2)}-\bar\alpha_{2k+1}^{(2)}\alpha_{2k+2}^{(1)}\varphi_{1}^{(2k+1)}+\rho^2_{2k+1} \bar \alpha_{2k}^{(2)}\varphi_{1}^{(2k)}+\rho^2_{2k+1}\rho^2_{2k}\varphi_{1}^{(2k-1)},\end{aligned}
\end{align}
to them we can add the truncated relations for $k=0,1$
\begin{align*}
z\varphi_1^{(0)}&=\varphi_1^{(2)}-\alpha_2^{(1)}\varphi_1^{(1)}-\alpha_1^{(1)}\varphi_1^{(0)},\\
z \varphi_1^{(1)}&=\bar \alpha_1^{(2)}\varphi_1^{(2)}-\bar \alpha_1^{(2)}\alpha_2^{(1)}\varphi_1^{(1)}+\rho_1^2 \varphi_1^{(0)}.
\end{align*}

It is also possible to build recursion relations multiplying by $z^{-1}$, thus we have
\begin{pro} The OLPUC have the following recursion relations
\begin{align}\label{pro.inv.relations.even}
z^{-1}\varphi_1^{(2k)}&=\alpha_{2k}^{(1)}\varphi_1^{(2k+1)}-\alpha_{2k}^{(1)}\bar \alpha_{2k+1}^{(2)}\varphi_1^{(2k)}+\rho^2_{2k}\alpha_{2k-1}^{(1)}\varphi_1^{(2k-1)}+\rho_{2k-1}^2\rho_{2k}^2\varphi_1^{(2k-2)},\\\label{pro.inv.relations.odd}
z^{-1}\varphi_1^{(2k+1)}&=\varphi_1^{(2k+3)}-\bar \alpha_{2k+3}^{(2)}\varphi_1^{(2k+2)}-\alpha_{2k+1}^{(1)}\bar \alpha_{2k+2}^{(2)}\varphi_1^{(2k+1)}-\rho_{2k+1}^2 \bar \alpha_{2k+2}^{(2)}\varphi_1^{(2k)},
\end{align}
where we have to add the truncated relation
\begin{align*}
z^{-1}\varphi_1^{(0)}=\varphi_1^{(1)}-\bar \alpha_1^{(2)}\varphi_1^{(0)}.
\end{align*}
\end{pro}
\begin{proof}
Using that $S_1 \Upsilon^{\top} S_1^{-1}\Phi_{1}=S_1 \Upsilon^{\top} S_1^{-1}S_1 \chi_{\CMV}(z)=z^{-1} \Phi_{1}$ we need to calculate the coefficients of $S_1 \Upsilon^{\top} S_1^{-1}=S_2 \Upsilon^{\top} S_2^{-1}$ as we did with $J_1$, to obtain the desired result.
\end{proof}
With the previous result we can get
\begin{pro}\label{pro.alpha.non.hermitian}
The coefficients $\rho_{l}^2$ verify the following relations
\begin{align*}
\rho_{k}^2&=1-\alpha_{k}^{(1)}\bar \alpha_{k}^{(2)},& k&\geq 0.
\end{align*}
\end{pro}
\begin{proof}
See Appendix \ref{proofs}.
\end{proof}

The following results were found previously in \cite{CMV} using an alternative derivation. Recalling the explicit form of the CMV Jacobi type matrix $J$ provided in Proposition \ref{explicit-J} we get
 \begin{pro}
   When  the measure $\mu$ is positive the recursion relations for the OLPUC can be expressed in terms of the Verblunsky coefficients  as follows
   \begin{align}\begin{aligned}\label{5terms0.1}
z\varphi_{1}^{(2k)}&=\varphi_{1}^{(2k+2)}-\alpha_{2k+2}\varphi_{1}^{(2k+1)}-\bar\alpha_{2k}\alpha_{2k+1}\varphi_{1}^{(2k)}
-\rho_{2k}^2\alpha_{2k+1}\varphi_{1}^{(2k-1)},\\
z\varphi_{1}^{(2k+1)}&=\bar \alpha_{2k+1}\varphi_{1}^{(2k+2)}-\bar\alpha_{2k+1}\alpha_{2k+2}\varphi_{1}^{(2k+1)}+\rho_{2k+1}^2\bar \alpha_{2k}\varphi_{1}^{(2k)}+\rho_{2k+1}^2\rho_{2k}^2\varphi_{1}^{(2k-1)},\end{aligned}
\end{align}
   \begin{align}\begin{aligned}\label{5terms0.2}
z^{-1}\varphi_1^{(2k)}&=\alpha_{2k}\varphi_1^{(2k+1)}-\alpha_{2k}\bar\alpha_{2k+1}\varphi_1^{(2k)}+\rho^2_{2k}\alpha_{2k-1}\varphi_1^{(2k-1)}+\rho_{2k-1}^2\rho_{2k}^2\varphi_1^{(2k-2)},\\
z^{-1}\varphi_1^{(2k+1)}&=\varphi_1^{(2k+3)}-\bar \alpha_{2k+3}\varphi_1^{(2k+2)}-\alpha_{2k+1}\bar \alpha_{2k+2}\varphi_1^{(2k+1)}-\rho_{2k+1}^2 \bar \alpha_{2k+2}\varphi_1^{(2k)},
\end{aligned}
\end{align}
where  $k \geq 0$.
 \end{pro}

Using the Szeg\H{o} polynomials and their reciprocals recursion relations \eqref{5terms.even} and \eqref{5terms.odd} can be expressed like
\begin{align}\label{5terms1}
zP_{2k}&=z^{-1}( P_{2k+2}-\alpha_{2k+2}(k)P_{2k+1}^*)+-\bar\alpha_{2k}\alpha_{2k+1}P_{2k}-(1-|\alpha_{2k}|^2)\alpha_{2k+1}P_{2k-1}^*,\\
zP_{2k+1}^*&=\bar \alpha_{2k+1}P_{2k+2}-\bar\alpha_{2k+1}\alpha_{2k+2}P_{2k+1}^*+z((1-|\alpha_{2k+1}|^2)\bar \alpha_{2k}P_{2k}+(1-|\alpha_{2k+1}|^2)(1-|\alpha_{2k}|^2)P_{2k-1}^*),\label{5terms2}
\end{align}
relations that can be obtained also with the classical Szeg\H{o} recurrence formulae and their reciprocals in $\T$.

\subsection{Projection operators and the Christoffel--Darboux kernel}

Here we discuss the CD kernel; i.e., the integral kernel of the \emph{quasi-orthogonal projection}, according to the sesquilinear form $\langle {\cdot},{\cdot} \rangle_{\L}$ defined by the measure $\mu$, to the space of OLPUC.
\begin{definition}\label{CD.projection}
We use the following notation
    \begin{align}\label{truncated Laurent CMV}\begin{aligned}
\Lambda^{[l]}_{\CMV}&:=\C\big\{\chi_{\CMV}^{(0)},\dots,\chi_{\CMV}^{(l-1)}\big\}=\begin{cases}\Lambda_{[k,k-1]}, &l=2k,\\
   \Lambda_{[k,k]}, &l=2k+1.
\end{cases} \end{aligned}
\end{align}
\end{definition}
Notice that
\begin{align*}
\Lambda^{[l]}_{\CMV}&=\C\{\varphi_{1}^{(0)},\dots,\varphi_{1}^{(l-1)}\}=\C\{\varphi_{2}^{(0)},\dots,\varphi_{2}^{(l-1)}\}.
\end{align*}
Associated with these spaces of truncated Laurent polynomials we consider the following related spaces, \emph{quasi-orthogonal complements},
\begin{align*}
(\Lambda^{[l]})^{\bot_2}&:=\Big\{\sum_{l\leq k \ll  \infty} c_k \varphi_1^{(k)},c_k\in\C\Big\},& (\Lambda^{[l]})^{\bot_1}&:=\Big\{\sum_{l\leq k \ll  \infty} c_k \varphi_2^{(k)},c_k\in\C\Big\}.
\end{align*}
Formally, we can express the following \emph{bi-quasi-orthogonality} relations
\begin{align*}
\langle \Lambda^{[l]}_{\CMV},(\Lambda^{[l]})^{\bot_1} \rangle_{\L}&=0, & \langle(\Lambda^{[l]})^{\bot_2},\Lambda^{[l]}_{\CMV}\rangle_{\L}&=0,
\end{align*}
and the corresponding splittings
\begin{align*}
 \Lambda_{[\infty]}&=\Lambda^{[l]}_{\CMV}\oplus (\Lambda^{[l]})^{\bot_1}=\Lambda^{[l]}_{\CMV}\oplus (\Lambda^{[l]})^{\bot_2},
\end{align*}
induce the associated \emph{quasi-orthogonal projections}
\begin{align*}
  \pi_1^{(l)}&:\Lambda_{[\infty]} \to \Lambda^{[l]}_{\CMV},&  \pi^{(l)}_2&:\Lambda_{[\infty]} \to \Lambda^{[l]}_{\CMV}.
\end{align*}
The reader should notice that we cannot properly talk of an orthogonal complement and an orthogonal projection if the measure is not positive and consequently we do not have an scalar product. If $\mu$ is a positive measure then $ (\Lambda^{[l]})^{\bot_1}= (\Lambda^{[l]})^{\bot_2}= (\Lambda^{[l]})^{\bot}$ and both projections are truly orthogonal and coincide.
\begin{definition}
The  CD kernel is defined by\footnote{In case that we have a positive measure $\mu$ then we can define the orthonormal Laurent polynomials $\tilde \varphi^{(l)}=(h_l)^{-\frac{1}{2}} \varphi_1^{(l)}=(h_l)^{\frac{1}{2}} \varphi_2^{(l)}$ so that  $K^{[l]}(z,z')=\sum_{k=0}^{l-1}h_k^{-1}\varphi_{1}^{(k)}(z')\bar \varphi_{1}^{(k)}(\bar z)=\sum_{k=0}^{l-1}h_k\varphi_{2}^{(k)}(z')\bar \varphi_{2}^{(k)}(\bar z)=\sum_{k=0}^{l-1} \tilde \varphi^{(k)}(z') \overline{ \tilde \varphi^{(k)}}(\bar z).$}
\begin{align}
  \label{def.CD}
  K_{\CMV}^{[l]}(z,z')&:=\sum_{k=0}^{l-1}\varphi_{1}^{(k)}(z')\bar \varphi_{2}^{(k)}(\bar z).
\end{align}
\end{definition}
As the CD kernel is expressed in terms of  Laurent polynomials the definition makes sense as long as $z,z'\neq 0$. This is the kernel of the integral representation of the projections   $\pi^{(l)}_1,\pi^{(l)}_2$:
\begin{pro}
The integral representation
  \begin{align*}
 ( \pi^{(l)}_{\CMV 1}f)(z')&=\oint_{\T}  K_{\CMV}^{[l]}(z,z')f(z)\d \mu(z), & \forall f\in\Lambda_{[\infty]},\\
  \overline {( \pi^{(l)}_{\CMV 2} f)(z)}&=\oint_{\T}  K_{\CMV}^{[l]}(z,z')\bar f(\bar z')\d \mu(z'), &\forall f\in\Lambda_{[\infty]},
\end{align*}
holds.
\end{pro}
\begin{proof}
It follows from the bi-orthogonality condition \eqref{biorth}.
\end{proof}
This CD kernel   has the reproducing property
\begin{pro} \label{reproducing.CMV}
  The kernel $K^{[l]}(z,z')$ fulfills
  \begin{align*}
K^{[l]}_{\CMV}(z,z')=\oint_{\T}  K^{[l]}_{\CMV}(z,u)K^{[l]}_{\CMV}(u,z')\d\mu( u).
\end{align*}
\end{pro}
\begin{proof}
See Appendix \ref{proofs}.

\end{proof}

\subsection{Associated Laurent polynomials}
In order to find CD formulae for the CD kernel just discussed we need the following definitions introdcing what we call associated Laurent polynomials
\begin{definition}\label{def.CMV.associated} In the $l$ even case the associated Laurent polynomials are
\begin{align*}
\varphi_{1,+1}^{(l)}&:=\chi^{(l)}-\begin{pmatrix}
    g_{l,0}&g_{l,1}&\cdots &g_{l,l-1}
  \end{pmatrix}(g^{[l]})^{-1}\chi^{[l]}, & \varphi_{1,-2}^{(l-1)}&:=e_{l-1}^{\top}(g^{[l]})^{-1}\chi^{[l]},\\
  \varphi_{2,+2}^{(l)}&:=\chi^{(l+1)}-\begin{pmatrix}
    \bar g_{0,l+1} & \bar g_{1,l+1} & \cdots & \bar g_{l-1,l+1}
  \end{pmatrix}((g^{[l]})^{-1})^{\dag}\chi^{[l]}, & \varphi_{2,-1}^{(l-1)}&:=e_{l-2}^{\top}((g^{[l]})^{-1})^{\dag}\chi^{[l]} ,
\end{align*}
while for the $l$ odd case they are
\begin{align*}
\varphi_{1,+1}^{(l)}&:=\chi^{(l+1)}-\begin{pmatrix}
    g_{l+1,0}&g_{l+1,1}&\cdots &g_{l+1,l-1}
  \end{pmatrix}(g^{[l]})^{-1}\chi^{[l]}, & \varphi_{1,-2}^{(l-1)}&:=e_{l-2}^{\top}(g^{[l]})^{-1}\chi^{[l]}, \\
\varphi_{2,+2}^{(l)}&:=\chi^{(l)}-\begin{pmatrix}
    \bar g_{0,l} &  \bar g_{1,l} & \cdots & \bar g_{l-1,l}
  \end{pmatrix}((g^{[l]})^{-1})^{\dag} \chi^{[l]}, &  \varphi_{2,-1}^{(l-1)}&:=e_{l-1}^{\top}((g^{[l]})^{-1})^{\dag}\chi^{[l]}.
\end{align*}
\end{definition}
The associated polynomials can be expressed in terms of the Laurent polynomials in different alternative manners
\begin{theorem} \label{exp.CMV.associated}If $\mu$ is a positive measure the associated Laurent polynomials in the $l$ even case are
\begin{align}
\varphi_{1,+1}^{(l)}(z)&=\varphi_{1}^{(l)}(z), & \varphi_{1,-2}^{(l-1)}(z)&=h_{l-1}^{-1}\varphi_1^{(l-1)}(z),\notag \\
\varphi_{2,+2}^{(l)}(z)&=z^{-1} h_l \bar \varphi_{2}^{(l)}(z^{-1}), & \varphi_{2,-1}^{(l-1)}(z)&=z^{-1}\bar \varphi^{(l-1)}_{2}(z^{-1}),\label{ape1}\\
&=z^{-1}(\bar \alpha_l h_l \varphi_{2}^{(l)}(z)+h_{l-1}\rho_l^2\varphi_{2}^{(l-1)}(z)),& &=z^{-1}(\rho_l^2\varphi_{2}^{(l)}(z)-\alpha_l\varphi_{2}^{(l-1)}(z)),\label{ape2}\\
&=h_{l+1} \varphi_2^{(l+1)}(z)-h_{l}\bar \alpha_{l+1} \varphi_2^{(l)}(z), &&=\alpha_{l-1}\varphi_2^{(l-1)}(z)+\varphi_2^{(l-2)}(z),\label{ape3}
\end{align}
and in the $l$ odd case
\begin{align}
\varphi_{2,+2}^{(l)}(z)&=h_{l}^{-1} \varphi_2^{(l)}(z), & \varphi_{2,-1}^{(l-1)}(z)&= \varphi_2^{(l-1)}(z),\notag \\
\varphi_{1,+1}^{(l)}(z)&=\bar\varphi_1^{(l)}(z^{-1}),& \varphi_{1,-2}^{(l-1)}(z)&=h^{-1}_{l-1}\bar\varphi_1^{(l-1)}(z^{-1}),\label{apo1}\\
&=z(\alpha_l \varphi_{1}^{(l)}(z)+\rho_l^2\varphi_{1}^{(l-1)}(z)), & &=z(h^{-1}_l\rho_l^2\varphi_{1}^{(l)}(z)-\bar \alpha_lh^{-1}_{l-1}\varphi_{1}^{(l-1)}(z))\label{apo2}\\
&=\varphi_{1}^{(l+1)}(z)-\alpha_{l+1} \varphi_1^{(l)}(z) , &&= h_{l-1}^{-1}\bar \alpha_{l-1} \varphi_{1}^{(l-1)}(z)+h^{-1}_{l-2}\varphi_1^{(l-2)}(z)\label{apo3}.\end{align}
\end{theorem}
\begin{proof}
\begin{enumerate}
\item To prove \eqref{ape1} and \eqref{apo1} we proceed as follows. On the one hand, when $l$ is even  \eqref{ort.ass.CMV.even.1} implies
\begin{align*}
\oint_{\T} z \varphi_{2,+2}^{(l)}(z) z^{-j} \d \mu(z)&=0, & j=-\frac{l}{2}+1,\dots,\frac{l}{2},
\end{align*}
and on the other hand, due to the Hermitian property of the scalar product, it follows for $\varphi_2^{(l)}$ that
\begin{align*}
\oint_{\T} \bar \varphi_{2}^{(l)}(z^{-1}) z^{-j} \d \mu(z)&=0, & j=-\frac{l}{2}+1,\dots,\frac{l}{2}.
\end{align*}
Hence, $z \varphi_{2,+2}^{(l)}\in\Lambda_{[l/2,l/2]}$ and solves the same linear system of equations that $\bar \varphi_{2}^{(l)}(z^{-1})\in\Lambda_{[l/2,l/2]}$ does. Consequently both Laurent polynomials are proportional. The equality is obtained from the coefficients in the power $z^{-\frac{l}{2}}$. In a similar way, from \eqref{ort.ass.CMV.even.2} we see that
\begin{align*}
\oint_{\T} z^j z^{-1} \bar \varphi_{2,-1}^{(l-1)} (z^{-1}) \d \mu(z) &=0, & j&=-\frac{l}{2}+1, \dots, \frac{l}{2}-1, & \oint_{\T} z^{\frac{l}{2}} z^{-1} \bar \varphi_{2,-1}^{(l-1)}(z^{-1}) \d \mu(z)&=1,
\end{align*}
 this means that $ z \varphi^{(l-1)}_{2,-1}(z)\in\Lambda_{[l/2-1,l/2]}$ has the same orthogonality relations and normalization condition that $\bar \varphi^{(l-1)}_{2}(z^{-1})\in\Lambda_{[l/2-1,l/2]}$, so they coincide.
Analogously, in the odd case  we obtain \eqref{apo1}.
\item For \eqref{ape2} and \eqref{apo2} we argue  in the following manner. Using orthogonality relations for $\bar \varphi_{1}^{(l)}(z^{-1})$ and $\bar \varphi_{2}^{(l-1)}(z^{-1})$ we conclude that
\begin{align*}
\bar \varphi_{1}^{(l)}(z^{-1}) & \in \text{span}\{\varphi_1^{(l)},\varphi_1^{(l-1)}\},\\
\bar \varphi_{2}^{(l-1)}(z^{-1}) & \in \text{span}\{\varphi_2^{(l)},\varphi_2^{(l-1)}\},
\end{align*}
and identifying coefficients
\begin{align*}
z\varphi_{2,+2}^{(l)}(z)&=h_l \bar \varphi_{2}^{(l)}(z^{-1})= \bar \varphi_{1}^{(l)}(z^{-1})=\bar \alpha_l \varphi_{1}^{(l)}(z)+\rho_l^2 \varphi_{1}^{(l-1)}(z)=\bar \alpha_l h_l \varphi_{2}^{(l)}(z)+\rho_l^2 h_{l-1}\varphi_{2}^{(l-1)}(z),\\
z\varphi_{2,-1}^{(l-1)}(z)&=\bar \varphi_{2}^{(l-1)}(z^{-1})=\rho_l^2\varphi_{2}^{(l)}(z)-\alpha_l\varphi_{2}^{(l-1)}(z),
\end{align*}
that concludes the proof of \eqref{ape2}, \eqref{apo2} follows similarly.

\item Finally, we proceed now  to prove \eqref{ape3} and \eqref{apo3}. For the even case we  compute the following integral
\begin{align*}
\oint_{\T}  \chi^{[l]}(z) \bar \varphi_{2,+2}^{(l)}(\bar z) \d \mu(z) &=\oint_{\T} \chi^{[l]}(z)(\chi^{(l+1)}(z))^{\dag} \d \mu(z)-\\
&-\oint_{\T}\chi^{[l]}(z)\chi^{[l]}(z)^{\dag} \d \mu(z)(g^{[l]})^{-1}\begin{pmatrix}
    g_{0,l+1} \\ g_{1,l+1} \\ \cdots \\ g_{l-1,l+1}
  \end{pmatrix} \\
& = \begin{pmatrix} g_{0,l+1} & g_{1,l+1} & \dots & g_{l-1,l+1} \end{pmatrix}^{\top} -\begin{pmatrix} g_{0,l+1} & g_{1,l-1} \dots & g_{l-1,l+1} \end{pmatrix}^{\top}\\
&= \begin{pmatrix} 0 & 0 & \dots & 0 \end{pmatrix}^{\top},
\end{align*}
which written componentwise reads
\begin{align}\label{ort.ass.CMV.even.1}
\oint_{\T} z^{j} \bar \varphi_{2,+2}^{(l)}(\bar z)  \d \mu(z)&=0, & j=-\frac{l}{2},\dots,\frac{l}{2}-1.
\end{align}
It also follows from the definition that $\bar \varphi_{2,+2}^{(l)}(z^{-1}) \in \Lambda_{[\frac{l}{2}-1,\frac{l}{2}+1]}$, and $(\bar \varphi_{2,+2}^{(l)}-z^{\frac{l}{2}+1}) \in \Lambda_{[\frac{l}{2}-1,\frac{l}{2}]}$.
For the other associated Laurent polynomials, the orthogonality relations are
\begin{align}\label{ort.ass.CMV.even.2}
\oint_{\T} z^j \bar \varphi_{2,-1}^{(l-1)}(\bar z) \d \mu(z) &=0, & j&=-\frac{l}{2}, \dots, \frac{l}{2}-2, & \oint_{\T} z^{\frac{l}{2}-1} \bar \varphi_{2,-1}^{(l-1)}(\bar z) \d \mu(z)&=1.
\end{align}

To get this result we proceed as before
\begin{align*}
\oint_{\T} \chi^{[l]}(z) \bar \varphi_{2,-1}^{(l-1)}(\bar z) \d \mu(z)= \left(\oint_{\T} \chi^{[l]}(z)\chi^{[l]}(z)^{\dag} \d \mu(z) \right)  (g^{[l]})^{-1}e_{l-2}=e_{l-2},
\end{align*}
that is the matrix version of the orthogonality relations.

These orthogonality relations lead to
\begin{align*}
\varphi_{2,+2}^{(l)}&=a_l\varphi_2^{(l+1)}+b_l\varphi_2^{(l)}, & \varphi_{2,-1}^{(l-1)}&=c_{l-1}\varphi_2^{(l-1)}+d_{l-1}\varphi_2^{(l-2)}.
\end{align*}

 Let us prove this statement. As $\varphi_{2,+2}^{(l)} \in \Lambda_{[\frac{l}{2}+1,\frac{l}{2}-1]}$ then $\varphi_{2,+2}^{(l)} \in \text{span}\{\varphi_2^{(0)}, \varphi_2^{(1)}, \dots ,\varphi_2^{(l+1)}\}$, but due to the orthogonality relations all the coefficients vanish except for the ones corresponding to $\varphi_2^{(l+1)}$ and $\varphi_2^{(l)}$. Comparing the coefficients of $z^{-\frac{l}{2}-1}$ and $z^{\frac{l}{2}}$ we get the system of equations
 \begin{align*}
 1&=(S_2)_{l+1,l+1}^{-1}a_l+0, & 0&=a_l(S_2)_{l+1,l+1}^{-1}\bar \alpha_{l+1}+(S_2)^{-1}_{l,l}b_l,
 \end{align*}
 from where we conclude $a_l=(S_2)_{l+1,l+1}$ and $b_l=-\bar \alpha_{l+1}(S_2)_{l,l}$. Now, we notice that $\varphi_{2,-1}^{(l-1)} \in \text{span}\{\varphi_2^{(0)}, \varphi_2^{(1)}, \dots ,\varphi_2^{(l-1)}\}$ and also that   the orthogonality relations imply that  \begin{align*}
 \varphi_{2,-1}^{(l-1)}\bot\text{span}\{\varphi_2^{(0)}, \varphi_2^{(1)}, \dots ,\varphi_2^{(l-3)}\}.
  \end{align*}
  Therefore, $\varphi_{2,-1}^{(l-1)} \in \text{span}\{\varphi_2^{(l-2)},\varphi_2^{(l-1)}\}$ and we only need to find the expression of the associated Laurent polynomials as a linear combination of these two Laurent polynomials. For that aim, we take the complex conjugate, multiply by $\varphi_1^{(l-1)}$ and $\varphi_1^{(l-2)}$, and integrate to  obtain
\begin{align*}
\bar c_{l-1} &= \oint_{\T}\varphi_1^{(l-1)}(z) \bar \varphi_{2,-1}^{(l-1)}(\bar z) \d \mu(z)= \oint_{\T} z^{-\frac{l}{2}} \bar \varphi_{2,-1}^{(l-1)}(\bar z) \d \mu(z) + \bar \alpha_{l-1} \oint_{\T} z^{\frac{l}{2}-1} \bar \varphi_{2,-1}^{(l-1)}(\bar z) \d \mu(z)=0+\bar \alpha_{l-1},\\
\bar d_{l-1} &= \oint_{\T}\varphi_1^{(l-2)}(z) \bar \varphi_{2,-1}^{(l-1)}(\bar z) \d \mu(z) = \oint_{\T} z^{\frac{l}{2}-1} \bar \varphi_{2,-1}^{(l-1)}(\bar z) \d \mu(z)=1,
\end{align*}
so we conclude $c_{l-1}= \alpha_{l-1}$ and $d_{l-1}=1$. For the odd case one proceeds in an analogous form.

A different proof for the same formula can be found in Appendix \ref{proofs}.
\end{enumerate}
\end{proof}

Finally we give determinantal expressions for these polynomials

\begin{pro} \label{dets} The associated Laurent polynomials have the following determinantal expressions
 \begin{align}\label{det.ass}
  \varphi_{1,+a}^{(l)}(z)&=\frac{1}{\det g^{[l]}}\det
  \left(\begin{BMAT}{cccc|c}{cccc|c}
    g_{0,0}&g_{0,1}&\cdots&g_{0,l-1}&\chi^{(0)}(z)\\
     g_{1,0}&g_{1,1}&\cdots&g_{1,l-1}&\chi^{(1)}(z)\\
     \vdots &\vdots&            &\vdots&\vdots\\
        g_{l-1,0}&g_{l-1,1}&\cdots&g_{l-1,l-1}&\chi^{(l-1)}(z)\\
          g_{l_{+a},0}&g_{l_{+a},1}&\cdots&g_{l_{+a},l-1}&\chi^{(l_{+a})}(z)
\end{BMAT}\right),& l&\geq 1.
\end{align}
 \begin{align}\label{det.ass.minus}
  \varphi_{1,-a}^{(l)}(z)&=\frac{(-1)^{l+l_{-a}}}{\det g^{[l+1]}}\det
  \left(\begin{BMAT}{cccc:ccc|c}{cccc}
    g_{0,0}&g_{0,1}&\cdots&g_{0,l_{-a}-1}&g_{0,l_{-a}+1}& \cdots &g_{0,l}&\chi^{(0)}(z)\\
     g_{1,0}&g_{1,1}&\cdots&g_{1,l_{-a}-1}&g_{1,l_{-a}+1}& \cdots &g_{1,l}&\chi^{(1)}(z)\\
     \vdots &\vdots&   & \vdots & \vdots  &    &\vdots&\vdots\\
          g_{l,0}&g_{1}&\cdots&g_{l,l_{-a}-1}&g_{l,l_{-a}+1}& \cdots &g_{l_{+a},l}&\chi^{(l)}(z)
\end{BMAT}\right),& l&\geq 1.
\end{align}
and
\begin{align} \label{det.dualass.plus}
  \bar \varphi_{2, +a}^{(l)}(\bar z)&=\frac{1}{\det g^{[l]}}\det
  \left(\begin{BMAT}{cccc|c}{cccc|c}
    g_{0,0}&g_{0,1}&\cdots&g_{0,l-1}& g_{0,l_{+a}}\\
     g_{1,0}&g_{1,1}&\cdots&g_{1,l-1}& g_{1,l_{+a}} \\
     \vdots &\vdots&   &\vdots&\vdots\\
        g_{l-1,0}&g_{l-1,1}&\cdots&g_{l-1,l-1}&g _{l-1,l_{+a}}\\
        (\chi^{(0)}(z))^\dag & (\chi^{(1)}(z))^\dag & \dots & (\chi^{(l-1)}(z))^\dag & (\chi^{(l_{+a})}(z))^\dag
\end{BMAT}\right), & l\geq 1.
\end{align}
\begin{align} \label{det.dualass}
  \bar \varphi_{2, -a}^{(l)}(\bar z)&=\frac{(-1)^{l+l_{-a}}}{\det g^{[l+1]}}\det
  \left(\begin{BMAT}{cccc}{cccc:ccc|c}
    g_{0,0}&g_{0,1}&\cdots&g_{0,l}\\
     g_{1,0}&g_{1,1}&\cdots&g_{1,l} \\
     \vdots &\vdots&   &\vdots\\
     g_{l_{-a}-1,0}&g_{l_{-a}-1,1}&\cdots& g_{l_{-a}-1,l} \\
     g_{l_{-a}+1,0}&g_{l_{-a}+1,1}&\cdots& g_{l_{-a}+1,l} \\
     \vdots &\vdots&   &\vdots\\
        g_{l,0}&g_{l,1}&\cdots&g _{l,l}\\
        (\chi^{(0)}(z))^\dag& (\chi^{(1)}(z))^\dag & \dots & (\chi^{(l)}(z))^\dag
\end{BMAT}\right), & l\geq 1.
\end{align}
\end{pro}
\begin{proof}
See Appendix \ref{proofs}.

\end{proof}

\subsection{The Christoffel--Darboux formula}
To obtain   CD formula in this context we need a number of  preliminary lemmas. First we consider a version of the Aitken--Berg--Collar theorem \cite{Simon}
\begin{lemma} \label{ABC.th.cmv} The following  ABC type formula
\begin{align*}
K^{[l]}_{\CMV}(z,z')&=\chi^{[l]}_{\CMV}(z)^{\dag}(g_{\CMV}^{[l]})^{-1}\chi_{\CMV}^{[l]}(z')
\end{align*}
is fulfilled.
\end{lemma}
\begin{proof}
See the Appendix \ref{proofs}.

\end{proof}

The CD formula can be obtained using the previous expressions for the CD kernel.
\begin{lemma} \label{primera.CD.CMV} For the CD kernel one has
\begin{align*}
(z'-\bar z^{-1})K^{[l]}(z,z')&=\chi^{[l]}_{\CMV}(z)^{\dag}(g_{\CMV}^{[l]})^{-1}z'\chi^{[l]}_{\CMV}(z')-\bar z^{-1} \chi_{\CMV}^{[l]}(z)^{\dag}(g_{\CMV}^{[l]})^{-1}\chi_{\CMV}^{[l]}(z')\\
&=(\chi_{\CMV}^{[l]}(z)^{\dag}(g_{\CMV}^{[l]})^{-1}g_{\CMV}^{[l,\geq l]}-\chi_{\CMV}^{[\geq l]}(z)^{\dag})\Upsilon_{\CMV}^{[\geq l, l]}(g_{\CMV}^{[l]})^{-1}\chi_{\CMV}^{[l]}(z')-\\
&-\chi_{\CMV}^{[l]}(z)^{\dag}(g_{\CMV}^{[l]})^{-1}\Upsilon_{\CMV}^{[l, \geq l]}(g_{\CMV}^{[\geq l,l]}(g_{\CMV}^{[l]})^{-1}\chi_{\CMV}^{[l]}(z')-\chi_{\CMV}^{[\geq l]}(z')).
\end{align*}
\end{lemma}
\begin{proof}
See Appendix \ref{proofs}.

\end{proof}

The reader can easily check
\begin{lemma} \label{upsilons.CMV} If $l$ is an even number
\begin{align*}
\Upsilon_{\CMV}^{[l,\geq l]}&=E_{l-2,l-l}=e_{l-2}e_{0}^{\top}, & \Upsilon_{\CMV}^{[\geq l, l]}&=E_{l+1-l,l-1}=e_{1}e_{l-1}^{\top},
\end{align*}
while for the $l$ odd case he have
\begin{align*}
\Upsilon_{\CMV}^{[l,\geq l]}&=E_{l-1,l+1-l}=e_{l-1}e_{1}^{\top}, & \Upsilon_{\CMV}^{[\geq l, l]}&=E_{l-l,l-2}=e_{0}e_{l-2}^{\top}.
\end{align*}
\end{lemma}

\begin{theorem}\label{main.CD.th.CMV}
The following CD formula holds
\begin{align}\label{gen.CMV.CD}
K^{[l]}_{\CMV}(z,z')&=\frac{\bar \varphi_{2,+2}^{(l)}(\bar z) \bar z \varphi_{1,-2}^{(l-1)}( z') - \varphi_{1,+1}^{(l)}(z') \bar z \bar
\varphi_{2,-1}^{(l-1)}(\bar z)  }{(1-z'\bar z)}, &z'\bar z &\neq 1.
\end{align}
\end{theorem}
\begin{proof}
The proof of \eqref{gen.CMV.CD} relies in Lemmas \ref{primera.CD.CMV} and  \ref{upsilons.CMV}. Let us first study the $l$ even case; with Lemma \ref{primera.CD.CMV} and \ref{upsilons.CMV} we obtain a more explicit expression for the CD kernel given by
\begin{align*}
\begin{aligned}
(\bar z^{-1}-z')K^{[l]}(z,z')&=(\chi^{(l+1)}(z)^{\dag}-\chi^{[l]}(z)^{\dag}(g^{[l]})^{-1}g^{[l,\geq l]}e_{1})e_{l-1}^{\top}(g^{[l]})^{-1}\chi^{[l]}(z')-\\
&-\chi^{[l]}(z)^{\dag}(g^{[l]})^{-1}e_{l-2}(\chi^{(l)}(z')-e_{0}^{\top}g^{[\geq l,l]}(g^{[l]})^{-1}\chi^{[l]}(z')),
\end{aligned}
\end{align*}
then using Definition \ref{def.CMV.associated} for the associated Laurent polynomials and Proposition \ref{exp.CMV.associated} we conclude our claim.
that leads to \eqref{gen.CMV.CD}. 

For the $l$ odd case, reasoning again with Lemmas \ref{primera.CD.CMV} and  \ref{upsilons.CMV} we obtain the expression
\begin{align*}
\begin{aligned}
(\bar z^{-1}-z')K^{[l]}(z,z')&=(\chi^{(l)}(z)^{\dag}-\chi^{[l]}(z)^{\dag}(g^{[l]})^{-1}g^{[l,\geq l]}e_{0})e_{l-2}^{\top}(g^{[l]})^{-1}\chi^{[l]}(z')-\\
&-\chi^{[l]}(z)^{\dag}(g^{[l]})^{-1}e_{l-1}(\chi^{(l+1)}(z')-e_{1}^{\top}g^{[\geq l,l]}(g^{[l]})^{-1}\chi^{[l]}(z')).
\end{aligned}
\end{align*}
 and recalling with Definition \ref{def.CMV.associated} we immediately get the claimed result.
\end{proof}
Recalling the different expressions for the associated Laurent polynomials in Theorem \ref{exp.CMV.associated}, one easily notice that
\begin{cor}
  For a positive measure $\mu$ the CD kernel can be written in the following alternative forms. In the $l$ even case it can be written as
\begin{align}\label{CMV.CD.Bis.Aux.Even}
K^{[l]}_{\CMV}(z,z')&=\frac{\varphi_{1}^{(l)}(\bar z^{-1}) \varphi_{2}^{(l-1)}(z') -\varphi_{1}^{(l)}(z') \varphi_{2}^{(l-1)}(\bar z^{-1})}{(1-z'\bar z)}\\\label{CMV.CD.Bis.Even}
&=\frac{(\overline{\bar \alpha_l\varphi_{1}^{(l)}(z)+\rho_l^2\varphi_{1}^{(l-1)}(z)}) \varphi_{2}^{(l-1)}(z') -\varphi_{1}^{(l)}(z') (\overline{\rho_l^2\varphi_{2}^{(l)}(z)-\alpha_l\varphi_{2}^{(l-1)}(z)})}{(1-z'\bar z)},\\&=\frac{\overline{z(\varphi_{1}^{(l+1)}(z)-\bar  \alpha_{l+1} \varphi_{1}^{(l)}(z) )} \varphi_{2}^{(l-1)}(z') -\varphi_{1}^{(l)}(z') \overline {z( \alpha_{l-1} \varphi_{2}^{(l-1)}(z)+\varphi_{2}^{(l-2)}(z))}}{(1-z'\bar z)}\label{CMV.CD.even}
\end{align}
while in the $l$ odd case it can be written as
\begin{align}
\label{CMV.CD.Bis.Aux.Odd}K^{[l]}_{\CMV}(z,z')&=
\frac{\bar z \bar \varphi_{1}^{(l)}(\bar z) \bar \varphi_{2}^{(l-1)}(z'^{-1}) - \bar \varphi_{1}^{(l)}(z'^{-1}) \bar z \bar \varphi_{2}^{(l-1)}(\bar z)}{(1-z' \bar z)}\\\label{CMV.CD.Bis.Odd}&=\bar z z'
\frac{ \bar \varphi_{1}^{(l)}(\bar z) (\rho_l^2\varphi_{2}^{(l)}(z')-\bar \alpha_l \varphi_{2}^{(l-1)}(z')) - (\alpha_l\varphi_{1}^{(l)}(z')+\rho_l^2 \varphi_{1}^{(l-1)}(z')) \bar \varphi_{2}^{(l-1)}(\bar z)}{(1-z' \bar z)}\\\label{CMV.CD.odd}
&=
\frac{\overline{z \varphi_{1}^{(l)}(z)} (\bar \alpha_{l-1} \varphi_{2}^{(l-1)}(z')+\varphi_{2}^{(l-2)}(z')) - (\varphi_{1}^{(l+1)}(z')-\alpha_{l+1} \varphi_{1}^{(l)}(z'))\overline{z \varphi_{2}^{(l-1)}(z)}}{(1-z'\bar z)}.
\end{align}
\end{cor}
 Formulae \eqref{CMV.CD.Bis.Even} and \eqref{CMV.CD.Bis.Odd} where found\footnote{However, the authors use the orthonormal sequence instead of the dual monic orthogonal sequences} in \cite{Barroso-Vera}, however the other expressions, to the best of our knowledge, are new.

\section{Extended CMV ordering and orthogonal Laurent polynomials}\label{genCMV}

This section is devoted to an extension of the CMV ordering that allows to extend CMV  results to more general situations. The main result is an extension of the CD formula allowing for projecting kernels to general spaces of Laurent polynomials, avoiding the CMV restriction on degrees.

\subsection{Extending the CMV ordering}

Let us consider a vector $\vec n\in\Z_+^2$, $\vec n=(n_+,n_-)$, and the associated alternated sequence
with $n_+$ positive increasing powers of $z$ followed by $n_-$ negative decreasing powers of $z$,that is
\begin{align*}
\chi_{\vec n}(z):=(1,z,\dots,z^{n_+-1},z^{-1},z^{-2},\dots,z^{-n_-},z^{n_+},z^{n_++1}\dots,)^\top.
\end{align*}
The CMV case presented above corresponds to the particular choice $n_+=n_-=1$. Given $l\geq 0$ if $\chi_{\vec n}^{(l)}$ is a non-negative power of $z$ we say that $a(l)=1$ and if $\chi_{\vec n}^{(l)}$ is a negative power of $z$ we define $a(l)=2$. In addition, for any $l\geq 0$ we will denote $\nu_+(l)$ $(\nu_-(l))$ as the number of elements in the set $\{\chi_{\vec n}^{(l')}, 0 \leq l' \leq l\}$ with $a(l')=1$ $(a(l')=2)$. That is

\begin{align}\label{def.nu+-}
\nu_+(l)&:=\#\{\chi_{\vec n}^{(l')}, a(l')=1, 0\leq l' \leq l \}, & \nu_-(l)&:=\#\{\chi_{\vec n}^{(l')}, a(l')=2, 0 \leq l' \leq l \}.
\end{align}
We will use the notation
\begin{align*}
  \vec \nu&=(\nu_+,\nu_-),& |\vec\nu|&:=\nu_++\nu_-,& |\vec n|&:=n_++n_-.
\end{align*}
Additional expressions for \eqref{def.nu+-}
can be found using the Euclidean division,  (\cite{afm-2}, \cite{bergvelt}), as we now explain. For any given $l \geq 0$ there exists unique non-negative integers $q(l)$ and $r(l)$ such that
\begin{align*}
l&=q(l)|\vec n|+r(l), & 0 &\leq r(l) < n_+, & \text{if} \quad a(l)&=1,\\
l&=q(l)|\vec n|+n_+ + r(l), & 0&\leq r(l)<n_-, &\text{if} \quad a(l)&=2,
\end{align*}
so that
\begin{align*}
\nu_+(l)&=\begin{cases} q(l)n_++r(l)+1, & a(l)=1, \\ (q(l)+1)n_+, &
a(l)=2, \end{cases} & \nu_-(l)&=\begin{cases} q(l)n_- ,& a(l)=1, \\
q(l)n_-+r(l)+1, & a(l)=2, \end{cases}
\end{align*}
from where $|\vec \nu (l)|=l+1$. If   $\{e_k\}_{k=0}^{\infty}$ is the formal canonical basis of $\R^{\infty}$ we consider $\{e_a(k)\}_{k=0}^{\infty}$, with $a=1,2$, defined as
\begin{align*}
e_1(\nu_+(l)-1)&:=e_l, & e_2(\nu_-(l)-1)&:=e_l,
\end{align*}
these are new labels for $\{e_k\}_{k=0}^{\infty}$ adapted to $\vec n$. Given a non-negative integer $l$ there exist a unique non-negative integer $k$ and a number $a \in \{1,2\}$ such that $e_a(k)=e_l$. This basis allows for a natural decomposition of $\chi_{\vec n}$ using positive and negative powers. In particular
\begin{align*}
\chi_{\vec n,a}(z)&:=\sum_{k=0}^{\infty}e_a(k)z^k, & a&=1,2, & \chi_{\vec n}&=\chi_{\vec n,1}+\chi_{\vec n, 2}^*.
\end{align*}

With those sequences we define the extended CMV moment matrix
\begin{definition}
 The extended moment matrix is the following  semi-infinite matrix
\begin{align}\label{def.g}
g_{\vec n}:=\oint_{\T} \chi_{\vec n}(z) \chi_{\vec n}(z)^{\dagger} \d \mu(z).
\end{align}
\end{definition}
Notice that  this moment matrix is a definite positive Hermitian matrix if $\mu$ is positive.
The Gaussian factorization for the semi-infinite matrix
$g_{\vec n}$ is
\begin{align*}
  g_{\vec n}=(S_{\vec n,1})^{-1} S_{\vec n,2},
\end{align*}
where $S_{\vec n,1}$ is a normalized lower triangular matrix and $S_{\vec n,2}$ is an upper triangular matrix.
The associated sequences of Laurent polynomials are
\begin{align*}
\Phi_{\vec n,1}(z):&=S_{\vec n, 1}\chi_{\vec n}(z), &  \Phi_{\vec n,2}(z):&=(S_{\vec n, 2}^{-1})^{\dagger} \chi_{\vec n}(z),
\end{align*}
where
\begin{align*}
\Phi_{\vec n,1}(z)&=
\begin{pmatrix}
\varphi_{\vec n,1}^{(0)}(z)\\
\varphi_{\vec n,1}^{(1)}(z) \\
\vdots
\end{pmatrix},&
\Phi_{\vec n,2}(z)&=
\begin{pmatrix}
 \varphi_{\vec n,2}^{(0)}(z)\\
 \varphi_{\vec n,2}^{(1)}(z) \\
\vdots
\end{pmatrix}.
\end{align*}
As in the CMV case, the sets of Laurent polynomials $\{\varphi_{\vec n,1}^{(l)}\}_{l=0}^\infty$ and $\{\varphi_{\vec n,2}^{(l)}\}_{l=0}^\infty$ are bi-orthogonal with respect to the sesquilinear form $\langle {\cdot},{\cdot}\rangle_{\L}$, that is
 \begin{align}\label{vec.bi-orthogonality}
\langle \varphi_{\vec n,1}^{(l)},\varphi_{\vec n,2}^{(k)}\rangle_{\L}=\oint_{\T} \varphi_{\vec n,1}^{(l)}(z) \bar \varphi_{\vec n,2}^{(k)}(z^{-1}) \d \mu(z) &=\delta_{l,k}, &
l,k=0,1,\dots.
\end{align}
In addition if the measure $\mu$ is positive we have that $\varphi^{(l)}_{\vec n,2}=h_l^{-1}\varphi_{\vec n,1}^{(l)}$ and both families are proportional Laurent polynomials. In addition
\begin{align}\label{orthogonality2}
\langle \varphi_{\vec n,1}^{(l)},\varphi_{\vec n,1}^{(k)}\rangle_{\L}=\oint_{\T} \varphi_{\vec n,1}^{(l)}(z) \bar \varphi_{\vec n,1}^{(k)}(z^{-1}) \d \mu(z) &=\delta_{l,k}h_l, &
l,k=0,1,\dots.
\end{align}
so in this case $\{\varphi_{\vec n,1}^{(l)}\}_{l=0}^\infty$  is a set of orthogonal Laurent polynomials with $\varphi_{\vec n,1}^{(l)}(z) \in \Lambda_{[\nu_-(l),\nu_+(l)-1]}$. The orthogonality relations \eqref{orthogonality2} read as follows
\begin{align}
\begin{aligned}\label{orth2}
\langle \varphi_{\vec n,1}^{(l)},z^k\rangle_{\L}=\oint_{\T} \varphi_{1, \vec n}^{(l)}(z) z^{-k} \d \mu(z)&=0, &k=-\nu_-(l-1),\dots,\nu_+(l-1)-1. \\
\end{aligned}
\end{align}
In terms of the Szeg\H{o} polynomials we have
\begin{pro} \label{grados.vec}
  For a positive measure $\mu$ we have the following identifications between the extended CMV Laurent polynomials, the Szeg\H{o} polynomials and their reciprocals
  \begin{align}\begin{aligned}\label{ex.rec.rel}
   z^{\nu_-(l)}\varphi_{\vec n,1}^{(l)}(z)&= P_{l}(z),& a(l)&=1,\\
  z^{\nu_-(l)}\varphi_{\vec n,1}^{(l)}(z)&= P_{l}^*(z),& a(l)&=2.\end{aligned}
  \end{align}
\end{pro}

\begin{proof}
See Appendix \ref{proofs}.

\end{proof}
The reader should notice that in the CMV case, $n_+=n_-=1$, for $l=2k$ we have $a(l)=1$, $\nu_-(l)=k$ and $\nu_+(l)=k+1$, and when  $l=2k+1$ then $a(l)=2$, $\nu_-(l)=\nu_+(l)=k+1$.
\subsection{Functions of the second kind}
Here we just give  a brief description account of this extended case
\begin{definition}
The partial second kind  sequences with the extended ordering are
given by
\begin{align*}
  C_{\vec n,1,1}(z)&:=(S_{\vec n,1}^{-1})^\dagger \chi^*_{\vec n,1}(z), & C_{\vec n,1,2}(z)&:=(S_{\vec n,1}^{-1})^\dagger \chi_{\vec n,2}(z),&
  C_{\vec n,2,1}(z)&:=S_{\vec n, 2}\chi^*_{\vec n,1}(z),& C_{\vec n,2,2}(z)&:=S_{\vec n,2}\chi_{\vec n,2}(z).
\end{align*}
and the second kind sequences
\begin{align*}
  C_{\vec n,1}(z)&:=(S_{\vec n,1}^{-1})^\dagger \chi^*_{\vec n}(z), &
  C_{\vec n,2}(z)&:=S_{\vec n,2}\chi^*_{\vec n}(z).
\end{align*}
\end{definition}
Generalized determinantal formulae can be obtained
\begin{pro}\label{n.det.Cauchy}
The extended second kind functions have the following determinantal expressions for $l \geq 1$.
\begin{align}
\overline{ C^{(l)}_{\vec n, 1}(z)}=\frac{1}{\det g_{\vec n}^{[l+1]}}\det
  \left(\begin{BMAT}{cccc}{cccc|c}
    g_{\vec n,0,0}&g_{\vec n,0,1}&\cdots& g_{\vec n,0,l}\\
     g_{\vec n,1,0}&g_{\vec n,1,1}&\cdots& g_{\vec n,1,l} \\
     \vdots &\vdots&   &\vdots\\
        g_{\vec n,l-1,0}&g_{\vec n,l-1,1}&\cdots&g _{\vec n,l-1,l}\\
        \bar \Gamma_{\vec n,2,0}^{(l)}& \bar \Gamma_{\vec n,2,1}^{(l)} & \dots & \bar \Gamma_{\vec n,2,l}^{(l)}
\end{BMAT}\right),
\end{align}
and
\begin{align}
C_{\vec n,2}^{(l)}(z)=\frac{1}{\det g_{\vec n}^{[l]}}\det
  \left(\begin{BMAT}{cccc|c}{cccc}
    g_{\vec n,0,0}&g_{\vec n,0,1}&\cdots&g_{\vec n,0,l-1}&\Gamma_{\vec n,1,0}^{(l)}\\
     g_{\vec n,1,0}&g_{\vec n,1,1}&\cdots&g_{\vec n,1,l-1}&\Gamma_{\vec n,1,1}^{(l)}\\
     \vdots &\vdots&            &\vdots&\vdots\\
          g_{\vec n,l,0}&g_{\vec n,l,1}&\cdots&g_{\vec n,l,l-1}&\Gamma_{\vec n,1,l}^{(l)}
\end{BMAT}\right),
\end{align}
where $\Gamma_{\vec n,1,j}^{(l)}:=\sum_{k \geq l} g_{\vec n,jk}\chi^{*(k)}_{\vec n}$ and $\Gamma_{\vec n,2,j}^{(l)}:=\sum_{k \geq l} g^{\dag}_{\vec n,jk}\chi^{*(k)}_{\vec n}.$
\end{pro}
The same can be said about the relationship between the second kind functions, the Fourier series of the measures, and the integral representation, that can be found in
\begin{pro}\label{n.proC}
 The partial second kind functions can be expressed as
 \begin{align*}
    C_{\vec n, 1,1}^{(l)}&= 2\pi \sum_{|k|\ll\infty}\varphi^{(l)}_{\vec n,2,k}z^{-k-1}\bar F^{(-)}_{\mu,-k-1}(z),& R_-&<|z|< \infty &
     C_{\vec n,1,2}^{(l)}&=  2\pi \sum_{|k|\ll\infty}\varphi^{(l)}_{\vec n,2,k}z^{-k-1} \bar F^{(+)}_{\mu,-k-1}(z), & 0&<|z|<R_+\\
     C_{\vec n,2,1}^{(l)}&= 2\pi \sum_{|k|\ll\infty}\varphi^{(l)}_{\vec n,1,k}z^{-k-1}F^{(+)}_{\mu,k}(z^{-1}),& R_+^{-1}&<|z|< \infty &
     C_{\vec n,2,2}^{(l)}&=  2\pi \sum_{|k|\ll\infty}\varphi^{(l)}_{\vec n,1,k}z^{-k-1}  F^{(-)}_{\mu,k}(z^{-1}),& 0&<|z|<R_-^{-1}
  \end{align*}
  and the second kind functions as
  \begin{align}\label{n.skf}
    C_{\vec n,1}^{(l)}&=2\pi\varphi^{(l)}_{\vec n,2}(z^{-1})z^{-1}\bar F_{\mu}(z),& R_-&<|z|< R_+& C_{\vec n,2}^{(l)}&=
     2\pi\varphi^{(l)}_{\vec n,1}(z^{-1})z^{-1}F_{\mu}(z^{-1}) & R_+^{-1}&<|z|< R_-^{-1}.
  \end{align}
\end{pro}
and in
\begin{pro}\label{n.pro:cauchy}
Assume a positive measure $\mu$ or a complex measure $w(\theta)\d\theta$ with $w$ a continuous function. Then, the second kind sequences can be written as the following Cauchy integrals
\begin{align*}
 C_{\vec n,1,1}^{(l)}&=z^{-1}\oint_\T\frac{u\varphi_{\vec n,2}^{(l)}(u)}{u-z^{-1}}\overline{\d\mu (u)},
 &C_{\vec n,2,1}^{(l)}&=z^{-1}\oint_\T\frac{u\varphi_{\vec n,1}^{(l)}(u)}{u-z^{-1}}\d\mu (u),& |z|&>1,\\
  C_{\vec n,1,2}^{(l)}&=-z^{-1}\oint_\T\frac{u\varphi_{\vec n,2}^{(l)}(u)}{u-z^{-1}}\overline{\d\mu (u)},&
  C_{\vec n,2,2}^{(l)}&=-z^{-1}\oint_\T\frac{u\varphi_{\vec n,1}^{(l)}(u)}{u- z^{-1}}\d\mu (u),& |z|&<1.
\end{align*}
\end{pro}

\subsection{Recursion relations}
As we already commented above the recursion relations among extended Laurent polynomials are more involved than in the CMV case.

\begin{definition}Given $ \vec n \in \Z_+^2$ we define the projections
 \begin{align*}
   \Pi_{\vec n,a}&:=\sum_{k=0}^\infty e_a(k) e_a(k)^{\top},& a&=1,2,
 \end{align*}
 and the shift matrices
\begin{align*}
\Lambda_{\vec n,a}&:=\sum_{k=0}^{\infty}e_a(k) e_a(k+1)^{\top},\\
\Upsilon_{\vec n}&:=\Lambda_{\vec n,1} +\Lambda_{\vec n,2}^{\top}+E_{n_+,n_+}(\Lambda^{\top})^{n_+}.
\end{align*}
\end{definition}
In the context of section \ref{genCMV} recursion relations can be obtained using the same technique. Our objective is to have an expression for the multiplication by $z$ and by $z^{-1}$ using the shift operators.
\begin{pro}\label{pro.gsymz}
\begin{enumerate}
  \item The moment matrix $g_{\vec n}$ has the following symmetry
\begin{align} \label{gsymz}
\Upsilon_{\vec n}g_{\vec n}=g_{\vec n}\Upsilon_{\vec n}.
\end{align}
\item The operator of multiplication by z has a $|\vec n |+3$ diagonal band structure in the basis given by $\Phi_{\vec n, 1}$ or $\Phi_{\vec n, 2}$
\end{enumerate}
\end{pro}
\begin{proof}
See Apendix \ref{proofs}
\end{proof}
Now we introduce the following associate integers. For $a=1,2$ we shall call $l_{+a}$ and $l_{-a}$ to the smallest (largest) integer $l'$ that verifies $l' \geq l$ ($l' \leq l$) and $a_1(l')=a$. For instance, in the previous case with $n_+=n_-=1$, if $l$ is an even number $l_{+1}=l_{-1}=l$, $l_{+2}=l+1$, and $l_{-2}=l-1$, in the case that $l$ is an odd number then $l_{+2}=l_{-2}=l$, $l_{+1}=l+1$, and $l_{-1}=l-1$. This structure leads to the following recursion laws for $k \geq 1$ and $0\leq l \leq |\vec n|-1$ (the $k=0$ case corresponds to the truncated recurrence relations).
\begin{align*}
z \varphi^{(|\vec n|k+l)}_{\vec n,1}=J^l_{\vec n, 0}(k)\varphi^{((|\vec n|k+l+1)_{+1})}_{\vec n,1}+J^l_{\vec n, 1}(k)\varphi^{((|\vec n|k+l+1)_{+1}-1)}_{\vec n,1}+\cdots+J^l_{\vec n, m_{\vec n}(l)-1}(k)\varphi^{((|\vec n|k+l-1)_{-2})}_{\vec n,1},
\end{align*}
where $m_{\vec n}(l)$ is the number of non vanishing terms in the recurrence formulae. The coefficients $J_{\vec n, j}^l(k)$ are again labeled with the index $j$ that accounts for the $m_{\vec n}(l)$ non vanishing coefficients for each $l$; as there are only $|\vec n|$ ``different'' recursion relations (the equivalent to the recurrences for the odd and even polynomials) then $l=0,1,\dots|\vec n|-1$. The connection with the elements of the Jacobi operator is the following $J_{\vec n, j}^{l}(k)=J_{\vec n, |\vec n|k+l,(|\vec n|k+l)_{+1}-j} $. The reader can check that $m_{\vec n}(l)=(|\vec n|k+l+1)_{+1}-(|\vec n|k+l-1)_{-2}+1$. Due to the fact that $(|\vec n|k+l+1)_{+1} \leq |\vec n|k+l+n_-+1$ and that $(|\vec n|k+l-1)_{-2} \geq |\vec n|k+l-1-n_+$ then $m_{\vec n}(l) \leq |\vec n|+3$ that agrees with the structure of $|\vec n|+3$ diagonals. Furthermore it is possible to calculate $m_{\vec n}(l)$ more explicitly and show that actually it does not depend on $k$. We can see that
\begin{align*}
(|\vec n|k+l+1)_{+1}&=\begin{cases} |\vec n|k+l+1, & 0 \leq l \leq n_+-2, \\ |\vec n| (k+1), & n_+-1  \leq l \leq |\vec n|-1, \end{cases}
\end{align*}
\begin{align*}
(|\vec n|k+l-1)_{-2}&=\begin{cases} |\vec n|k-1, & 0\leq l \leq n_+, \\ |\vec n|k+l-1 & n_++1 \leq l \leq |\vec n| -1, \end{cases}
\end{align*}
and consequently
\begin{align*}
m_{\vec n}(l)=\begin{cases} l+3 & l=0,\dots,n_+-2, \\ |\vec n| +2 & l=n_+-1,n_+, \\ |\vec n|-l+2 & l=n_++1,\dots,|\vec n|-1, \end{cases}
\end{align*}
so in fact $m_{\vec n}(l) \leq |\vec n| +2$.

The expressions for the coefficients $J^l_{\vec n, p}(k)$ with $l=0,\dots,n_+$ are
\begin{align*}
J^l_{\vec n,p}(k)&=\begin{cases} (S_{\vec n,1}^{-1})_{(|\vec n|k+l+1)_{+1},(|\vec n|k+l+1)_{+1}-p} &  p=0,\dots,(|\vec n|k+l+1)_{+1}-(|\vec n|k+l)-1 \\
 (S_{\vec n,2}^{-1})_{|\vec n|k+l,(|\vec n|k+l)_{+2}}(S_{\vec n,2}^{-1})_{(|\vec n|k+l)_{-2},(|\vec n|k+l+1)_{+1}-p}^{-1} & p=(|\vec n|k+l+1)_{+1}-(|\vec n|k+l), \dots,m_{\vec n}(l)-1,\end{cases}
 \end{align*}
 and for $l=n_+,\dots,|\vec n|-1,$ the expressions are
 \begin{align*}
 J^l_{\vec n, p}(k)&=\begin{cases} (S_{\vec n, 1}^{-1})_{|\vec n|k+l,(|\vec n|k+l)_{-1}}(S_{\vec n,1}^{-1})_{(|\vec n|k+l)_{+1},(|\vec n|k+l)_{+1}-p} & p=0,\dots,(|\vec n|k+l)_{+1}-(|\vec n|k+l) \\
 (h_{\vec n})_{|\vec n|k+l}(S_{\vec n,2}^{-1})_{(|\vec n|k+l-1)_{-2},(|\vec n|k+l)_{+1}-p} & p=(|\vec n|k+l+1)_{+1}-(|\vec n|k+l+1) \dots,m_{\vec n}(l)-1.\end{cases}
\end{align*}
The particular case of $n_-=n_+=1$ gives the 5 diagonal CMV matrix with only four non-vanishing coefficients. As a consequence, the standard CMV case has the shortest possible recurrence formula.

\subsection{The   Christoffel--Darboux kernel}
We discuss the CD kernel  for this extended case
\begin{definition}\label{CD projection}
 For each non-negative integer $l$ we define the set of truncated Laurent polynomial subspace as the following span
\begin{align*}
\Lambda_{\vec n}^{[l]}&:=\C\big\{\chi_{\vec n}^{(0)},\dots,\chi_{\vec n}^{(l-1)}\big\}=\Lambda_{[\nu_-(l-1),\nu_+(l-1)-1]}.
\end{align*}
\end{definition}
As before, we define  \emph{quasi-orthogonal} subspaces 
\begin{align*}
(\Lambda^{[l]}_{\vec n})^{\bot_1}&:=\Big\{\sum_{l\leq k \ll  \infty} c_k \varphi_{\vec n,2}^{(k)},c_k\in\C\Big\},& (\Lambda^{[l]}_{\vec n})^{\bot_2}&:=\Big\{\sum_{l\leq k \ll  \infty} c_k \varphi_{\vec n,1}^{(k)},c_k\in\C\Big\},
\end{align*}
so that the following\emph{ bi-quasi-orthogonality} relations are satisfied
\begin{align*}
\langle \Lambda^{[l]}_{\vec n},(\Lambda^{[l]}_{\vec n})^{\bot_1} \rangle_{\L}&=0, & \langle(\Lambda^{[l]}_{\vec n})^{\bot_2},\Lambda^{[l]}_{\vec n}\rangle_{\L}&=0,
\end{align*}
and the corresponding splittings
\begin{align*}
 \Lambda_{[\infty]}&=\Lambda^{[l]}_{\vec n}\oplus (\Lambda^{[l]}_{\vec n})^{\bot_1}=\Lambda^{[l]}_{\vec n}\oplus (\Lambda^{[l]}_{\vec n})^{\bot_2},
\end{align*}
that induce the associated projections
\begin{align*}
  \pi^{(l)}_{\vec n,1}&:\Lambda_{[\infty]}\to\Lambda^{[l]}_{\vec n},&  \pi^{(l)}_{\vec n,2}&:\Lambda_{[\infty]}\to\Lambda^{[l]}_{\vec n},
\end{align*}
hold.
In the positive definite case this extended version allows for the projection in more general spaces of truncated Laurent polynomials. Recall that for the CMV case the space of truncated polynomials given in \eqref{truncated Laurent CMV} includes only a very particular class of these truncations, excluding the majority of cases. The introduction of the extended ordering allows to include in the discussion all the possible situations of truncation. In fact the space $\Lambda_{[p,q]}$ can be achieved in a number of ways, and always by the choice $n_+=q+1$, $n_-=p$.

\begin{definition}
The  CD kernel is
\begin{align}
  \label{def.CD.n}
  K_{\vec n}^{[l]}(z,z')&:=\sum_{k=0}^{l-1}\varphi_{\vec n,1}^{(k)}(z')\bar \varphi_{\vec n,2}^{(k)}(\bar z),
\end{align}
and, whenever the measure $\mu$ is positive definite, we have the equivalent expressions
\begin{align*}
K_{\vec n}^{[l]}(z,z')&=\sum_{k=0}^{l-1}h_k^{-1}\varphi_{\vec n,1}^{(k)}(z')\bar \varphi_{\vec n,1}^{(k)}(\bar z)=\sum_{k=0}^{l-1}h_k\varphi_{\vec n,2}^{(k)}(z')\bar \varphi_{\vec n,2}^{(k)}(\bar z)=\sum_{k=0}^{l-1} \tilde \varphi_{\vec n}^{(k)}(z')\overline{\tilde \varphi_{\vec n}^{(k)}(z)}.
\end{align*}
\end{definition}
Proceeding as in the CMV ordering we conclude the following results
\begin{pro}
\begin{enumerate}
\item This is the integral kernel of the integral representation of the projections $\pi_{\vec n,1}^{(l)}$, $\pi_{\vec n,2}^{(l)}$
  \begin{align*}
 ( \pi^{(l)}_{\vec n,1}f)(z')&=\oint_{\T}  K_{\vec n}^{[l]}(z,z')f(z)\d \mu(z), & \forall f\in\Lambda_{[\infty]},\\
  \overline{( \pi^{(l)}_{\vec n,2}f)(z)}&=\oint_{\T}  K_{\vec n}^{[l]}(z,z')\bar f(\bar z')\d \mu(z'), &\forall f\in\Lambda_{[\infty]}.
\end{align*}
\item This CD kernel  $K_{\vec n}^{[l]}(z,z')$ has the reproducing property
  \begin{align*}
K^{[l]}_{\vec n}(z,z')=\oint_{\T}  K^{[l]}_{\vec n}(z,u)K^{[l]}_{\vec n}(u,z')\d\mu( u).
\end{align*}
\item The following version of the ABC theorem
\label{ABC.th.vec} holds
\begin{align*}
K^{[l]}_{\vec n}(z,z')&=\chi^{[l]}_{\vec n}(z)^{\dag}(g_{\vec n}^{[l]})^{-1}\chi_{\vec n}^{[l]}(z').
\end{align*}
\item
 \label{primera.CD.vec}The CD kernel has the following expression
\begin{align*}
(z'-\bar z^{-1})K^{[l]}(z,z')&=\chi^{[l]}_{\vec n}(z)^{\dag}(g_{\vec n}^{[l]})^{-1}z'\chi^{[l]}_{\vec n}(z')-\bar z^{-1} \chi_{\vec n}^{[l]}(z)^{\dag}(g_{\vec n}^{[l]})^{-1}\chi_{\vec n}^{[l]}(z')\\
&=(\chi_{\vec n}^{[l]}(z)^{\dag}(g_{\vec n}^{[l]})^{-1}g_{\vec n}^{[l,\geq l]}-\chi_{\vec n}^{[\geq l]}(z)^{\dag})\Upsilon_{\vec n}^{[\geq l, l]}(g_{\vec n}^{[l]})^{-1}\chi_{\vec n}^{[l]}(z')-\\
&-\chi_{\vec n}^{[l]}(z)^{\dag}(g_{\vec n}^{[l]})^{-1}\Upsilon_{\vec n}^{[l, \geq l]}(g_{\vec n}^{[\geq l,l]}(g_{\vec n}^{[l]})^{-1}\chi_{\vec n}^{[l]}(z')-\chi_{\vec n}^{[\geq l]}(z')).
\end{align*}
\end{enumerate}
\end{pro}

The integers $l_{\pm a}$ can be used to calculate the $\Upsilon$ blocks in this case. The reader can check
\begin{pro}\label{upsilons.vec}The formula for $\Upsilon_{\vec n}$ is the following
\begin{align*}
\Upsilon_{\vec n}^{[l,\geq l]}&=e_{(l-1)_{-1}}e_{l_{+1}-l}^{\top}, & \Upsilon_{\vec n}^{[\geq l, l]}&=e_{l_{+2}-l}e_{(l-1)_{-2}}^{\top}.
\end{align*}
\end{pro}
Thus, the expression of the CD kernel is
\begin{align}\begin{aligned}\label{CD.vec.aux.1}
(\bar z^{-1}-z')K_{\vec n}^{[l]}(z,z')&=(\chi_{\vec n}^{[\geq l]}(z)^{\dag}-\chi_{\vec n}^{[l]}(z)^{\dag}(g_{\vec n}^{[l]})^{-1}g_{\vec n}^{[l,\geq l]})e_{l_{+2}-l}e_{(l-1)_{-2}}^{\top}(g_{\vec n}^{[l]})^{-1}\chi_{\vec n}^{[l]}(z')-\\
&-\chi_{\vec n}^{[l]}(z)^{\dag}(g_{\vec n}^{[l]})^{-1}e_{(l-1)_{-1}}e_{l_{+1}-l}^{\top}(\chi_{\vec n}^{[\geq l]}(z')-g_{\vec n}^{[\geq l,l]}(g_{\vec n}^{[l]})^{-1}\chi_{\vec n}^{[l]}(z')),\end{aligned}
\end{align}
that suggests the definition of the following associated polynomials
\begin{definition}\label{def.vec.associated}
The associated Laurent polynomials are defined by
\begin{align*}
\varphi_{\vec n,1,+a}^{(l)}&:=\chi^{(l_{+a})}_{\vec n}-\begin{pmatrix}
    g_{\vec n,l_{+a},0}&g_{\vec n,l_{+a},1}&\cdots &g_{\vec n,l_{+a},l-1}
  \end{pmatrix}(g_{\vec n}^{[l]})^{-1}\chi_{\vec n}^{[l]}, & \varphi_{\vec n,1,-a}^{(l)}&:=e_{l_{-a}}^{\top}(g_{\vec n}^{[l+1]})^{-1}\chi_{\vec n}^{[l+1]}, \\
  \varphi_{\vec n,2,+a}^{(l)}&:=\chi^{(l_{+a})}_{\vec n}-\begin{pmatrix}
    \bar g_{\vec n,0,l_{+a}}& \bar g_{\vec n,1,l_{+a}}&\cdots &\bar g_{\vec n,l-1,l_{+a}}
  \end{pmatrix}((g_{\vec n}^{[l]})^{-1})^{\dag}\chi_{\vec n}^{[l]}, & \varphi_{\vec n,2,-a}^{(l)}&:=e_{l_{-a}}^{\top}((g_{\vec n}^{[l+1]})^{-1})^{\dag}\chi_{\vec n}^{[l+1]},
\end{align*}
where  $a=1,2$.
\end{definition}
It  is easy to see that  $\varphi_{\vec n,1,+a(l)}^{(l)}=\varphi_{\vec n,1}^{(l)}, \varphi_{\vec n,2,+a(l)}^{(l)}=(\bar S_2)_{ll}\varphi_{\vec n,2}^{(l)}$,$\varphi_{\vec n,1,-a(l)}^{(l)}=(S_2)_{ll}^{-1 }\varphi_{\vec n,1}^{(l)}$ and $\varphi_{\vec n,2,-a(l)}^{(l)}=\varphi_{\vec n,2}^{(l)}$.
\begin{theorem}\label{asso.vec}
For the associated Laurent polynomials $\varphi^{(l)}_{\vec n,+a},\varphi_{\vec n, -a}^{(l)}$ we have two alternative expressions
\begin{enumerate}
  \item The reciprocal type form (valid for positive definite cases)
  \begin{align}\begin{aligned}
\varphi^{(l)}_{\vec n,1,+2}(z)&=\varphi^{(l)}_{\vec n,2,+2}(z)=z^{\nu_+(l)-\nu_-(l)-2}\bar \varphi^{(l)}_{\vec n,1}(z^{-1}), \\
\varphi^{(l)}_{\vec n,1,-2}(z)&=\varphi^{(l)}_{\vec n,2,-2}(z)=z^{\nu_+(l)-\nu_-(l)-1}\bar \varphi^{(l)}_{\vec n,2}(z^{-1}),\end{aligned}
\end{align}
when $a(l)=1$ and
\begin{align}\begin{aligned}
\varphi^{(l)}_{\vec n,1,+1}(z)&=\varphi^{(l)}_{\vec n,2,+1}(z)=z^{\nu_+(l)-\nu_-(l)}\bar \varphi^{(l)}_{\vec n,1}(z^{-1}), \\
\varphi^{(l)}_{\vec n,1,-1}(z)&=\varphi^{(l)}_{\vec n,2,-1}(z)=z^{\nu_+(l)-\nu_-(l)-1}\bar \varphi^{(l)}_{\vec n,2}(z^{-1}),\end{aligned}
\end{align}
for $a(l)=2$.
\item The linear combination form (valid for quasi-definite cases)
\begin{align}\label{estr.ass}
 \varphi^{(l)}_{\vec n,1,+a} &= (S_1^{-1})_{l_{+a},l_{+a}}\varphi^{(l_{+a})}_{\vec n,1}+(S_1^{-1})_{l_{+a},l_{+a}-1}\varphi^{(l_{+a}-1)}_{\vec n,1}+\cdots+(S_1^{-1})_{l_{+a},l}\varphi^{(l)}_{\vec n,1},\\
 \varphi^{(l)}_{\vec n,2,+a} &= (\bar S_2)_{l_{+a},l_{+a}}\varphi^{(l_{+a})}_{\vec n,2}+(\bar S_2)_{l_{+a}-1,l_{+a}}\varphi^{(l_{+a}-1)}_{\vec n,2}+\cdots+(\bar S_2)_{l,l_{+a}}\varphi^{(l)}_{\vec n,2},
 \\\varphi^{(l)}_{\vec n,1,-a} &= (S_2^{-1})_{l_{-a},l_{-a}} \varphi^{(l_{-a})}_{\vec n,1}+( S_2^{-1})_{l_{-a},l_{-a}+1}\varphi^{(l_{-a}+1)}_{\vec n,1}+\dots+(S_2^{-1})_{l_{-a},l}\varphi^{(l)}_{\vec n,1},
 \\\label{estr.dual.ass}
\varphi^{(l)}_{\vec n,2,-a} &= (\bar S_1)_{l_{-a},l_{-a}} \varphi^{(l_{-a})}_{\vec n,2}+(\bar S_1)_{l_{-a}+1,l_{-a}}\varphi^{(l_{-a}+1)}_{\vec n,2}+\dots+(\bar S_1)_{l,l_{-a}}\varphi^{(l)}_{\vec n,2}.
\end{align}
\end{enumerate}
\end{theorem}
\begin{proof}
\begin{enumerate}
\item Let us suppose that $a(l)=1$. In that case we have $\varphi^{(l)}_{\vec n,1,+1}=\varphi^{(l)}_{\vec n,1}$ and $\varphi^{(l)}_{\vec n,1,+2}\in \Lambda_{[\nu_-(l)+1,\nu_+(l)-2]}$. Consequently, $z^{\nu_-(l)-\nu_+(l)+2}\varphi^{(l)}_{\vec n,1,+2}\in \Lambda_{[\nu_+(l)-1,\nu_-(l)]}$ and $a(l)=1$, $\nu_+(l)=\nu_+(l-1)+1$, $\nu_-(l)=\nu_-(l-1)$. For the dual polynomials $\varphi^{(l)}_{\vec n,2,-1}=\varphi^{(l)}_{\vec n,2}$ and $\varphi^{(l)}_{\vec n,2,-2}\in \Lambda_{[\nu_-(l),\nu_+(l)-1]}$, hence $z^{\nu_-(l)-\nu_+(l)+1}\varphi^{(l)}_{\vec n,2,-2}\in \Lambda_{[\nu_+(l)-1,\nu_-(l)]}$. Using \eqref{lo que he borrado}  we conclude that the following orthogonality relations hold true
\begin{align*}
\oint_{\T} z^{\nu_-(l)-\nu_+(l)+2} \varphi_{\vec n,1,+2}^{(l)}(z)z^{-k} \d \mu(z)&=0, & k&=-\nu_+(l-1)+1,\dots,\nu_-(l-1),\\
\oint_{\T} z^{\nu_-(l)-\nu_+(l)+1} \varphi_{\vec n,2,-2}^{(l)}(z)z^{-k} \d \mu(z)&=0, & k&=-\nu_+(l-1)+1,\dots,\nu_-(l-1),\\
\oint_{\T} z^{\nu_-(l)-\nu_+(l)+1} \varphi_{\vec n,2,-2}^{(l)}(z) z^{\nu_+(l-1)}\d \mu(z)&=1,
\end{align*}
and we get the result.

Now let us suppose that $a(l)=2$. In this case we have  $\varphi^{(l)}_{\vec n,1,+2}=\varphi^{(l)}_{\vec n,1}$ and $\varphi^{(l)}_{\vec n,1,+1}\in \Lambda_{[\nu_-(l)-1,\nu_+(l)]}$. Consequently, $z^{\nu_-(l)-\nu_+(l)}\varphi^{(l)}_{\vec n,1,+2}\in \Lambda_{[\nu_+(l)-1,\nu_-(l)]}$. Now, as $a(l)=2$, we have $\nu_+(l)=\nu_+(l-1)$ and $\nu_-(l)=\nu_-(l-1)+1$. For the dual polynomials $\varphi^{(l)}_{\vec n,2,-2}=\varphi^{(l)}_{\vec n,2}$ and $\varphi^{(l)}_{\vec n,2,-1}\in \Lambda_{[\nu_-(l),\nu_+(l)-1]}$, so $z^{\nu_-(l)-\nu_+(l)+1}\varphi^{(l)}_{\vec n,2,-1}\in \Lambda_{[\nu_+(l)-1,\nu_-(l)]}$. Now using again \eqref{lo que he borrado} we get
\begin{align*}
\oint_{\T} z^{\nu_-(l)-\nu_+(l)} \varphi_{\vec n,1,+1}^{(l)}(z)z^{-k} \d \mu(z)&=0, & k&=-\nu_+(l-1)+1,\dots,\nu_-(l-1),\\
\oint_{\T} z^{\nu_-(l)-\nu_+(l)} \varphi_{\vec n,2,-1}^{(l)}(z)z^{-k} \d \mu(z)&=0, & k&=-\nu_+(l-1)+1,\dots,\nu_-(l-1),\\
\oint_{\T} z^{\nu_-(l)-\nu_+(l)+1} \varphi_{\vec n,2,-1}^{(l)}(z) z^{-\nu_-(l-1)-1}\d \mu(z)&=1.
\end{align*}

\item
For $\varphi_{\vec n,1,+a}^{(l)}$ direct computation gives
\begin{align*}
\oint_{\T} \varphi_{\vec n,1,+a}^{(l)}(z)\chi^{[l]}_{\vec n}(z)^{\dag} \d \mu(z) &=\oint_{\T} (\chi^{(l_{+a})}_{\vec n}(z)-\begin{pmatrix}
    g_{\vec n,l_{+a},0}&g_{\vec n,l_{+a},1}&\cdots &g_{\vec n,l_{+a},l-1}
  \end{pmatrix}(g^{[l]}_{\vec n})^{-1}\chi^{[l]}_{\vec n}(z)) \chi^{[l]}_{\vec n}(z)^{\dag} \d \mu(z) \\
&=\begin{pmatrix}
    g_{\vec n,l_{+a},0}&g_{\vec n,l_{+a},1}&\cdots &g_{\vec n,l_{+a},l-1}
  \end{pmatrix}-\begin{pmatrix}
    g_{\vec n,l_{+a},0}&g_{\vec n,l_{+a},1}&\cdots &g_{\vec n,l_{+a},l-1}
  \end{pmatrix}\\
&=0,
\end{align*}
and for $\varphi_{\vec n,2,-a}^{(l)}$ we have
\begin{align*}
\oint_{\T} \chi^{[l+1]}_{\vec n}(z) \bar \varphi_{\vec n,2,-a}(\bar z) \d \mu(z) &=\oint_{\T} \chi^{[l+1]}_{\vec n}(z) \chi^{[l+1]}_{\vec n}(z)^{\dag}(g^{[l+1]}_{\vec n})^{-1} e_{l_{-a}} \d \mu(z)=e_{l_{-a}},
\end{align*}
so that we get
orthogonality relations for the associated polynomials
\begin{align}\label{lo que he borrado}
\begin{aligned}
  \oint_{\T} \varphi_{\vec n,1,+a}^{(l)}(z)z^{-k} \d \mu(z)&=0, & k&=-\nu_-(l-1),\dots,\nu_+(l-1)-1,\\
\oint_{\T} \chi_{\vec n}^{(k)}(z) \bar \varphi_{\vec n,2,-a}(\bar z) \d \mu(z)&=\delta_{k,l_{-a}}, & k&=0,1,\dots,l.
\end{aligned}
\end{align}

Therefore,
\begin{align*}
\varphi^{(l)}_{\vec n,1,+a} \in \text{span} \{\varphi^{(l)}_{\vec n,1},\varphi^{(l+1)}_{\vec n,1},\dots,\varphi^{(l_{+a})}_{\vec n,1}\},
\end{align*}
i.e., $\varphi^{(l)}_{\vec n,1,+a}=\sum_{j=l}^{l_{+a}} A^{(l)}_j\varphi^{(j)}_{\vec n,1}$ for a set of coefficients $\{A^{(l)}_j\}$. Comparing the powers of $z$ that appear  in the subsequence $\{\chi^{(j)}\}_{l\leq j \leq l_{+a}}$ on both sides of the equation,  the following linear system of equations is obtained
\begin{align*} \begin{pmatrix} 1 & 0 & 0 & 0 & 0 & 0\\
(S_1)_{l_{+a},l_{+a}-1} & 1 & 0 & 0 & 0 & 0\\
(S_1)_{l_{+a},l_{+a}-2} & (S_1)_{l_{+a}-1,l_{+a}-2} & 1 & 0 & 0 & 0 \\
(S_1)_{l_{+a},l_{+a}-3} & (S_1)_{l_{+a}-1,l_{+a}-3} & \cdots & 1 & 0 & 0\\
\vdots & \vdots & &\vdots & \vdots & \vdots \\
(S_1)_{l_{+a},l} & (S_1)_{l_{+a}-1,l} & \cdots & (S_1)_{l+2,l} & (S_1)_{l+1,l} & 1
\end{pmatrix}
\begin{pmatrix}
A^{(l)}_{l_{+a}} \\ A^{(l)}_{l_{+a}-1} \\ A^{(l)}_{l_{+a}-2}\\ A^{(l)}_{l_{+a}-3}\\ \vdots \\A^{(l)}_{l}
\end{pmatrix}
=\begin{pmatrix}
1 \\ 0 \\ 0 \\ 0 \\ \vdots \\ 0
\end{pmatrix}
\end{align*}
calling $\mathcal{M}$ the coefficient matrix, the solution can be written as
\begin{align*}
\begin{pmatrix}
A^{(l)}_{l_{+a}} \\ A^{(l)}_{l_{+a}-1} \\ \vdots \\A^{(l)}_{l}
\end{pmatrix}
=\begin{pmatrix}
(\mathcal{M}^{-1})_{0,0} \\(\mathcal{M}^{-1})_{1,0}  \\ \vdots \\ (\mathcal{M}^{-1})_{l_{+a}-l,0}
\end{pmatrix}
\end{align*}
From the structure of $\mathcal{M}$ we conclude that
\begin{align*}
\begin{pmatrix}
0 & 0 & \cdots & 0 &1 \\
0 & 0 & \cdots & 1 &0 \\
\vdots & \vdots &  & \vdots & \vdots \\
0 & 1 & \cdots & 0 & 0 \\
1 & 0 & \cdots & 0 & 0
\end{pmatrix}
\mathcal{M}^{\top}
\begin{pmatrix}
0 & 0 & \cdots & 0 &1 \\
0 & 0 & \cdots & 1 & 0\\
\vdots & \vdots &  &\vdots & \vdots \\
0 & 1 & \cdots & 0 & 0 \\
1 & 0 & \cdots & 0 & 0
\end{pmatrix}
:=\mathcal{M}'=
\begin{pmatrix}
1 & 0 & \cdots & 0 & 0 \\
(S_1)_{l+1,l} & 1 & \cdots & 0 & 0 \\
\vdots & \vdots & & \vdots & \vdots \\
(S_1)_{l_{+a}-1,l} & (S_1)_{l_{+a}-1,l+1} & \cdots & 1 & 0 \\
(S_1)_{l_{+a},l} & (S_1)_{l_{+a},l+1} & \cdots & (S_1)_{l_{+a},l_{+a}-1} & 1
\end{pmatrix}.
\end{align*}
From the triangular structure of $S_1$ we deduce that $(\mathcal{M'}^{-1})_{i,j}=(S_1^{-1})_{i+l,j+l}$ for $i,j=0,1,\dots,l_{+a}-l$ and consequently $(\mathcal{M}^{-1})_{j,0}=(\mathcal{M'}^{-1})_{l_{+a}-l,l_{+a}-l-j}=(S_1^{-1})_{l_{+a},l_{+a}-j}$ for $i,j=0,1,\dots,l_{+a}-l$,
which proves \eqref{estr.ass}.
The expression for \eqref{estr.dual.ass} is obtained using a similar technique. Using again \eqref{lo que he borrado} we conclude that $ \varphi^{(l)}_{\vec n,2,-a} \in \text{span} \{\varphi^{(l_{-a})}_{\vec n,2},\varphi^{(l_{-a}+1)}_{\vec n,2},\dots,\varphi^{(l)}_{\vec n,2}\}$; i.e., $\varphi^{(l)}_{\vec n,2,-a}=\sum_{j=l_{-a}}^{l} B^{(l)}_j \varphi^{(j)}_{\vec n,2}$.
Bi-orthogonality and normalization properties imply
\begin{align*}
\bar B^{(l)}_j&=\int_{\T}\varphi^{(j)}_{\vec n,1}(z)\bar \varphi^{(l)}_{\vec n,2,-a}(z^{-1}) \d \mu(z)=(S_{1})_{j,l_{-a}}, & j&=l_{-a},\dots,l,
\end{align*}
that proves \eqref{estr.dual.ass}. The other two equations are obtained using the same idea.
\end{enumerate}

\end{proof}

The polynomials that appear in the CD formula are now clearly identified as
\begin{align}\label{CD.vec.aux.2}
(\chi_{\vec n}^{[\geq l]}(z)^{\dag}-\chi_{\vec n}^{[l]}(z)^{\dag}(g_{\vec n}^{[l]})^{-1}g_{\vec n}^{[l,\geq l]})e_{l_{+2}-l}&=\bar \varphi_{\vec n,2,+2}^{(l)}(\bar z), &
e_{(l-1)_{-2}}^{\top}(g_{\vec n}^{[l]})^{-1}\chi_{\vec n}^{[l]}(z')& =\varphi_{\vec n,1,-2}^{(l-1)}(z'), \\
e_{l_{+1}-l}^{\top}(\chi_{\vec n}^{[\geq l]}(z')-g_{\vec n}^{[\geq l,l]}(g_{\vec n}^{[l]})^{-1}\chi_{\vec n}^{[l]}(z'))&=\varphi_{\vec n,1,+1}^{(l)}(z'), & \chi_{\vec n}^{[l]}(z)^{\dag}(g_{\vec n}^{[l]})^{-1}e_{(l-1)_{-1}}&=\bar \varphi_{\vec n, 2,-1}^{(l-1)}(\bar z),
\end{align}
and consequently the final result is the following
\begin{theorem}
The CD formula for the extended ordering is the following
\begin{align}\label{CD.vec}
K_{\vec n}^{[l]}(z,z')& =\frac{\bar z\bar \varphi_{\vec n, 2,+2}^{(l)}(\bar z)\varphi_{\vec n, 1,-2}^{(l-1)}(z')-\varphi_{\vec n, 1,+1}^{(l)}(z')\bar z\bar \varphi_{\vec n, 2,-1}^{(l-1)}(\bar z)}{(1-z'\bar z)}.
\end{align}
\end{theorem}

We get the followings corollaries when we have a positive Borel measure $\mu$
\begin{cor}
Given a positive measure $\mu$, the CD  kernel can be expressed using
\begin{align*}
K_{\vec n}^{[l]}(z,z')& =\frac{\bar z^{\nu_+(l)-\nu_-(l)-1}\varphi_{\vec n, 1}^{(l)}(\bar z^{-1})z'^{\nu_+(l)-\nu_-(l)-2}\bar \varphi_{\vec n, 2}^{(l-1)}(z'^{-1})-\varphi_{\vec n, 1}^{(l)}(z')\bar z\bar \varphi_{\vec n, 2}^{(l-1)}(\bar z)}{(1-z'\bar z)},
\end{align*}
in the case $a(l)=a(l-1)=1$,
\begin{align*}
K_{\vec n}^{[l]}(z,z')& =\frac{\bar z^{\nu_+(l)-\nu_-(l)-1} \varphi_{\vec n, 1}^{(l)}(\bar z^{-1})\varphi_{\vec n, 2}^{(l-1)}(z')-\varphi_{\vec n, 1}^{(l)}(z')\bar z^{\nu_+(l)-\nu_-(l)-1} \varphi_{\vec n, 2}^{(l-1)}(\bar z^{-1})}{(1-z'\bar z)},
\end{align*}
in the case $a(l)=1,a(l-1)=2$,
\begin{align*}
K_{\vec n}^{[l]}(z,z')& =\frac{\bar z\bar \varphi_{\vec n, 1}^{(l)}(\bar z)z'^{\nu_+{(l)}-\nu_-(l)}\bar \varphi_{\vec n, 2}^{(l-1)}(z'^{-1})-z'^{\nu_+(l)-\nu_-(l)}\bar \varphi_{\vec n, 1}^{(l)}(z'^{-1})\bar z\bar \varphi_{\vec n, 2}^{(l-1)}(\bar z)}{(1-z'\bar z)},
\end{align*}
in the case $a(l)=2,a(l-1)=1$,
\begin{align*}
K_{\vec n}^{[l]}(z,z')& =\frac{\bar z\bar \varphi_{\vec n, 1}^{(l)}(\bar z)\varphi_{\vec n, 2}^{(l-1)}(z')-z'^{\nu_+(l)-\nu_-(l)}\bar\varphi_{\vec n, 1}^{(l)}(z'^{-1})\bar z^{\nu_+(l)-\nu_-(l)+1} \varphi_{\vec n, 2}^{(l-1)}(\bar z^{-1})}{(1-z'\bar z)},
\end{align*}
in the case $a(l)=a(l-1)=2$.
\end{cor}
and
\begin{cor}
The CD formula for a positive Borel measure $\mu$ can be expressed in terms on the Szeg\H{o} polynomials as
\begin{align}
K_{\vec n}^{[l]}(z,z')&=h^{-1}_{l-1} \bar z^{a(l)+\nu_+(l)-2}z'^{a(l)-1-\nu_-(l)}\frac{P_l(\bar z^{-1})P^{*}_{l-1}(z')-P_l(z')P^{*}_{l-1}(\bar z^{-1})}{1-z'\bar z}.
\end{align}
\end{cor}

\section{Associated 2D Toda type hierarchies}
Here we analyze the link between the previous constructions on OLPUC and integrable systems of Toda type. Our driving idea is the presence of the Borel-Gauss factorization problem in the theoretical construction of both, OLPUC and Toda.

\label{Toda}
\subsection{2D Toda flows}
We will consider a set of  complex deformation parameters\footnote{In the framework of the theory on integrable systems these parameters are understood as an infinite set of times, being the independent variables  in an associated nonlinear hierarchy of partial differential-difference equations} $t=\{t_{1j},t_{2j}\}_{j \in \N}$ and with it two semi-infinite matrices that we will define now.\footnote{We shall drop the subindex  $\vec n$ from $g$ and $\Upsilon$ as the definitions are valid for any value of $\vec n$. It will be supposed that a particular $\vec n$ is chosen and the whole \S\ref{Toda} will be built using that $\vec n$.}
\begin{definition}\begin{enumerate}
  \item The  deformation matrices are defined as follows
\begin{align} \label{def.def}
W_{1,0}(t)&:=\exp\Big(\sum_{j=1}^{\infty} t_{1j} \Upsilon^j\Big), & W_{2,0}(t)&:=\exp\Big(\sum_{j=1}^{\infty}t_{2j} (\Upsilon^{\top})^j\Big),
\end{align}
\item For each $t$ we will consider the matrix $g(t)$
\begin{align*}
g(t)&:=W_{1,0}(t)g(W_{2,0}(t))^{-1},
\end{align*}
\item and the corresponding time dependant Gauss--Borel factorization\footnote{For the sake of notation simplicity in some situations we  drop the time dependence of $S_1, S_2$ and they will not denote the factors within the Gauss--Borel factorization of the initial condition but for the ``deformed'' one. Consequently,  in this section $S_1, S_2$ will always depend on ``time'' parameters.}
\begin{align*}
g(t)&:=W_{1,0}(t)g(W_{2,0}(t))^{-1}, & g(t)&=(S_1(t))^{-1} S_2(t).
\end{align*}
\end{enumerate}
\end{definition}
As we show now the deformed moment matrix is a moment matrix of a deformed measure.
\begin{pro}
The ``deformed'' moment matrix can be understood as a moment matrix for a ``deformed'' (that is, a ``time'' dependant) measure given by
\begin{equation}
\d \mu (t,z) := \exp \Big( \sum_{j=1}^{\infty} t_{1j} z^j - t_{2j} z^{-j} \Big) \d \mu(z).
\end{equation}
\end{pro}
\begin{proof}
First expanding the exponentials in \eqref{def.def} we obtain
\begin{align*}
W_{1,0}(t)&=\sum_{k=0}^{\infty} \sigma_1^{k}(t) \Upsilon^k, & (W_{2,0}(t))^{-1}&=\sum_{l=0}^{\infty}\sigma_2^{l} (t) (\Upsilon^{\top})^l,
\end{align*}
then using the definition of $g$ and $g(t)$ we get the desired result
\begin{align*}
W_{1,0}(t) g W_{2,0}(t)^{-1}&=\sum_{k,l=0}^{\infty}\sigma_1^k(t)\Upsilon^k g (\Upsilon^{\top})^l \sigma_2^l(t)=\oint_{\T} \sum_{k=0}^{\infty} \sigma_1^k(t) z^j\chi(z)\chi(z)^{\dag} \sum_{l=0}^{\infty} \sigma_2^l(t) z^{-l} \d \mu(z) \\
&=\oint_{\T}\chi(z) \chi(z)^{\dag} \exp \Big(\sum_{j=0}^{\infty}( t_{1j} z^j- t_{2j} z^{-j}) \Big) \d \mu(z).
\end{align*}
\end{proof}
From this result we conclude at least for absolutely continuous measures
\begin{align*}
  F_{\mu(t)}=\exp \Big( \sum_{j=1}^{\infty} t_{1j} z^j - t_{2j} z^{-j} \Big)F_\mu(z),
\end{align*}
from where we deduce that the radii defining the annulus of convergence is time independent; i.e., $R_\pm(t)=R_\pm$

Given a positive definite initial measure $\mu$ in order to ensure that the evolved measure $\mu(t)$ is also definite positive for all times is enough to request to the exponential to be real; i.e., setting $t_{2j}=-\bar t_{1j}$, so that
\begin{align*}
\exp \Big( \sum_{j=0}^{\infty}( t_{1j} z^j + \bar t_{1j} z^{-j}) \Big)=\exp \Big( \sum_{j=0}^{\infty} 2  \Re (t_{1j} z^j) \Big).
\end{align*}

\subsection{Integrable Toda equations}
\begin{definition}
Associated with the deformed Gauss--Borel factorization we consider
   \begin{enumerate}
\item  Wave semi-infinite  matrices
\begin{align}
  W_1(t)&:=S_1(t)W_{1,0}(t),&   W_2(t)&:=S_2(t) W_{2,0}(t).
\end{align}
\item Partial wave and partial adjoint (denoted the adjoint by $^*$) wave
semi-infinite vector functions\footnote{Also called Baker, or Baker--Akhiezer, functions},
\begin{align}\begin{aligned}\label{defbaker}
  \Psi_{1,1}(z,t)&:=W_1(t)\chi_1(z),& \Psi_{2,1}^*(z,t)&:= (W_2(t)^{-1})^{\dag}\chi_1(z),\\
  \Psi_{1,2}(z,t)&:=W_1(t)\chi_2^*(z),& \Psi_{2,2}^*(z,t)&:= (W_2(t)^{-1})^{\dag}\chi_2^*(z),\\
  \Psi_{1,1}^*(z,t)&:=(W_1(t)^{-1})^{\dag}\chi^*_1(z),& \Psi_{2,1}(z,t)&:=W_2(t)\chi^*_1(z), \\
  \Psi_{1,2}^*(z,t)&:=(W_1(t)^{-1})^{\dag}\chi_2(z),& \Psi_{2,2}(z,t)&:=W_2(t)\chi_2(z), \\
  \end{aligned}
\end{align}
and wave and adjoint wave functions
\begin{align}\label{defbakertotal}\begin{aligned}
\Psi_{1}(z,t)&:=W_1(t)\chi(z)=(\Psi_{1,1}+\Psi_{1,2})(z,t), & \Psi_{2}^*(z,t)&:=(W_2(t)^{-1})^{\dag}\chi(z)=(\Psi_{2,1}^*+\Psi_{2,2}^*)(z,t),\\
\Psi_{1}^*(z,t)&:=(W_1(t)^{-1})^{\dag}\chi^*(z)=(\Psi_{1,1}^*+\Psi_{1,2}^*)(z,t), & \Psi_{2}(z,t)&:=W_2(t)\chi^*(z)=(\Psi_{2,1}+\Psi_{2,2})(z,t)
.\end{aligned}
\end{align}
 \item Lax semi-infinite matrices
\begin{align}\label{deflax}
  L_1(t)&:=S_1(t)\Upsilon S_1(t)^{-1},&
  L_2(t)&:=S_2(t)\Upsilon ^\top  S_2(t)^{-1}.
\end{align}
\item Zakharov--Shabat semi-infinite matrices
\begin{align}\label{defzs}
B_{1,j}&:=(L^j_1)_+,& B_{2,j}&:=(L^j_{2})_-,
  \end{align}
  where the subindex $+$ indicates the projection in the upper triangular matrices while the subindex
   $-$ the projection in the strictly lower triangular matrices.
     \end{enumerate}
\end{definition}
\begin{theorem}
For $j,j'=1,2,\dots$,
the following differential relations hold
  \begin{enumerate}
    \item Auxiliary linear systems for the wave matrices
    \begin{align}\label{auxlinsys}
      \frac{\partial W_1}{\partial t_{1j}}&=B_{1,j} W_1, &\frac{\partial W_1}{\partial t_{2j}}&=B_{2,j} W_1,&
      \frac{\partial W_2}{\partial t_{1j}}&=B_{1,j} W_2, &\frac{\partial W_2}{\partial
      t_{2j}}&=B_{2,j}W_2.
    \end{align}
    \item Linear systems for the wave and adjoint wave semi-infinite functions
       \begin{align}\begin{aligned}\label{linbaker}
      \frac{\partial \Psi_{1}}{\partial t_{1j}}&=B_{1,j}\Psi_{1}, &\frac{\partial \Psi_{1}}{\partial t_{2j}}&=B_{2,j} \Psi_{1},&
         \frac{\partial \Psi_{2}^*}{\partial t_{1j}}&=-B_{1,j}^\dag \Psi_{2}^*, &\frac{\partial
      \Psi_{2}^*}{\partial t_{2j}}&=-B_{2,j}^\dag
      \Psi_{2}^*,\\
      \frac{\partial \Psi_{1}^*}{\partial t_{1j}}&=-B_{1,j}^{\dag}\Psi_{1}, &\frac{\partial \Psi_{1}^*}{\partial t_{2j}}&=-B_{2,j}^{\dag} \Psi_{1},&\frac{\partial \Psi_{2}}{\partial t_{1j}}&=B_{1,j} \Psi_{2}, &\frac{\partial
      \Psi_{2}}{\partial t_{2j}}&=B_{2,j}
      \Psi_{2}.
      \end{aligned}
    \end{align}
    \item Lax equations
      \begin{align}\label{laxeq}
        \frac{\partial L_{1}}{\partial t_{1j}}&=[B_{1,j},L_{1}], &\frac{\partial L_{1}}{\partial t_{2j}}&=[B_{2,j},L_{1}], &
      \frac{\partial L_{2}}{\partial t_{1j}}&=[B_{1,j},L_{2}], &\frac{\partial L_{2}}{\partial t_{2j}}&=[B_{2,j}, L_{2}].
    \end{align}
    \item Zakharov--Shabat equations
    \begin{align}\label{zseq}
      \frac{\partial B_{1,j}}{\partial t_{1j'}}-
          \frac{\partial B_{1,j'}}{\partial t_{1j}} +[B_{1,j},B_{1,j'}]&=0,\\
           \frac{\partial B_{2,j}}{\partial t_{2j'}}-
          \frac{\partial \bar B_{2,j'}}{\partial t_{2j}} +[B_{2,j}, B_{2,j'}]&=0,\\
           \frac{\partial B_{1,j}}{\partial t_{2j'}}-
          \frac{\partial B_{2,j'}}{\partial t_{1j}} +[B_{1,j}, B_{2,j'}]&=0.
    \end{align}
  \end{enumerate}
\end{theorem}

\begin{proof}
The proof can be made using the same idea used in \cite{afm-2}, so we do not repeat them here again.
\end{proof}
From the definition it is clear that the wave functions are associated to the OLPUC for the evolved measure
\begin{pro}
 The wave functions are linked to the OLPUC and the Fourier series of the measure trough
  \begin{align}\label{evol.baker}
&\begin{aligned}
  \Psi_{1}^{(n)}(z,t)&=\varphi^{(n)}_{1}(z,t) \Exp{\sum_{j=1}^{\infty}t_{1j} z^j}, \\ (\Psi_{2}^*)^{(n)}(z,t)&=\varphi^{(n)}_{2}(z,t) \Exp{-\sum_{j=1}^{\infty} \bar t_{2j} z^{j}},
\end{aligned}
\\\label{evol.cauchy}
&\begin{aligned}
 (\Psi_{1}^*)^{(n)}(z,t)&=2\pi \varphi_2^{(n)}(z^{-1},t) z^{-1}\bar F_{\mu(t)}(z)\Exp{-\sum_{j=1}^{\infty} \bar t_{1j} z^{j}}\\
 &=2\pi \varphi_2^{(n)}(z^{-1},t) z^{-1}\bar F_{\mu}(z)\Exp{-\sum_{j=1}^{\infty} \bar t_{2j} z^{-j}}, \\
\Psi_{2}^{(n)}(z,t)&=2\pi\varphi_2^{(n)}(z^{-1},t) z^{-1}F_{\mu(t)}(z^{-1})\Exp{\sum_{j=1}^{\infty} t_{2j} z^{j}}\\
&=2\pi\varphi_2^{(n)}(z^{-1},t) z^{-1}F_{\mu}(z^{-1})\Exp{\sum_{j=1}^{\infty} t_{1j} z^{-j}}.
\end{aligned}
\end{align}
Moreover, the wave functions are eigen-functions of the Lax and adjoint Lax matrices
\begin{align*}
L_1\Psi_1&=z\Psi_1, & L_2^{\dag}\Psi_2^*&=z\Psi_2^*,\\
L_1^\dag \Psi_1^*&=z\Psi_1^*, & L_2\Psi_2&=z\Psi_2.
\end{align*}
\end{pro}

\subsection{CMV matrices and the Toeplitz lattice}
For the CMV ordering of the Laurent basis, Lax equations \eqref{laxeq} can be we written as a nonlinear dynamical system that is a version in the CMV context of the Toeplitz lattice discussed by Mark Adler and Pierre van Moerbeke \cite{Adler-Van-Moerbecke-Toeplitz}
\begin{pro}\label{red.Toeplitz}
For the case $\vec n= (1,1)$ Lax equations \eqref{laxeq} have as a consequence the following nonlinear dynamical system for the Verblunsky coefficients
\begin{align}\begin{aligned}\label{CMV.Lat}
\frac{\partial \alpha_{k}^{(1)}}{\partial t_{11}}&=\alpha_{k+1}^{(1)}(1-\alpha_{k}^{(1)}\bar \alpha_{k}^{(2)}),& \frac{\partial \bar \alpha_{k}^{(2)}}{\partial t_{11}}&=-\bar \alpha_{k-1}^{(2)}(1-\alpha_{k}^{(1)}\bar \alpha_{k}^{(2)}),\\
\frac{\partial \alpha_{k}^{(1)}}{\partial t_{21}}&=\alpha_{k-1}^{(1)}(1-\alpha_{k}^{(1)}\bar \alpha_{k}^{(2)}),& \frac{\partial \bar \alpha_{k}^{(2)}}{\partial t_{21}}&=-\bar \alpha_{k+1}^{(2)}(1-\alpha_{k}^{(1)}\bar \alpha_{k}^{(2)}),
\end{aligned}
\end{align}
with $k=1,2,\dots$.

\end{pro}
\begin{proof}
See Appendix \ref{proofs}.

\end{proof}

If the initial measure $\mu$ is positive definite then the sequences $\{\alpha_k^{(1)}\}$ and $\{\bar \alpha_k^{(2)}\}$ are identical in $t=0$. Furthermore, if we set $t_{21}=-\bar t_{11}$, the evolved measure is always real and there is only one family of time-dependent functions. That is the reduction studied by L. B. Golinskii \cite{Golinskii} in the context of Schur flows.

We have obtained a CMV version of the Toeplitz lattice, but for any $\vec n$ the integrable hierarchy obtained is always equivalent to the one discussed in \cite{Adler-Van-Moerbecke-Toeplitz}. The reason for this fact relies in the observation that for any positive Borel measure $\mu$ and for any $\vec n$, there is a bijection between the set of OLPUC  $\{\varphi^{(l)}_{\vec n, 1}\}$ and the set of OPUC $\{P_l\}$. All the coefficients of any $P_l$ are determined in terms of the set of reflection coefficients $\{ \alpha_l \}$, so the time evolution for the Szeg\H{o} polynomials under Toda-type flows is determined by the evolution of the reflection coefficients. As the measure evolution does not depend on $\vec n$, it is natural to always obtain the very same evolution for the family $\{\alpha_l\}$ under the Toda flows. We conjecture that a similar result holds for the quasi-definite case.

\subsection{Discrete flows}
We now consider discrete flows associated to the moment matrix. Given two integers $s_1,s_2$ and $s:=(s_1,s_2)$ is then possible to make a new deformation of the moment matrix that depends on $s$.
\begin{definition}
We introduce for each $s$ and the deformed moment matrix $g(s)$ and its deformed Gauss--Borel factorization
\begin{align*}
g(s)&:=D_{1,0}(s)g(D_{2,0}(s))^{-1}, & g(s)&=S_1^{-1}(s)S_2(s),
\end{align*}
where $D_{1,0}(s),D_{2,0}(s)$ are discrete deformation operators to be determined later on.\end{definition}
We consider the operator $T_1$  responsible of the shift $s_1 \mapsto s_1+1$ and $T_2$ corresponding to the shift $s_2 \mapsto s_2+1$.
Let us suppose that matrices $q_1,q_2$ exist and satisfy
\begin{align*}
T_1(D_{1,0})&=q_1D_{1,0}, & T_1(D_{2,0})&=D_{2,0}, \\
T_2(D_{1,0})&=D_{1,0}, & T_2(D_{2,0})&=q_2D_{2,0},
\end{align*}
then we define
\begin{align*}
\delta_1&:=S_1(s)q_1S_1(s)^{-1}, & \delta_2&:=S_2(s)q_2S_2(s)^{-1}.
\end{align*}
If $\delta_1,\delta_2$ can be \emph{LU} factorized, then there exist semi-infinite matrices $\delta_{1,+},\delta_{1,-},\delta_{2,+},\delta_{2,-}$ such that
\begin{align*}
\delta_1&=\delta_{1,-}^{-1}\delta_{1,+}, & \delta_2^{-1}&=\delta_{2,-}^{-1}\delta_{2,+}.
\end{align*}
In this case we introduce
\begin{align*}
\omega_1&:=\delta_{1,+}, & \omega_2&:=\delta_{2,-}.
\end{align*}
\begin{pro}\label{discrete.delta}
The operators $T_1,T_2$ and the matrices $S_1(s),S_2(s)$ satisfy the following equations
\begin{align*}
T_1(S_1(s))(S_1(s))^{-1}&=\delta_{1,-} & T_1(S_2(s))(S_2(s))^{-1}&=\delta_{1,+}\\
T_2(S_1(s))(S_1(s))^{-1}&=\delta_{2,-} & T_2(S_2(s))(S_2(s))^{-1}&=\delta_{2,+}
\end{align*}
\end{pro}
\begin{proof}
First using $T_1$ and $T_2$ on the factorization $D_{1,0}gD_{2,0}^{-1}=S_1^{-1}S_2$ we obtain
\begin{align*}
T_1(D_{1,0})gT_1(D_{2,0}^{-1})&=T_1(S_1^{-1})T_1(S_2) \quad \Rightarrow & (T_1(S_1)S_1^{-1})^{-1}T_1(S_2)S_2^{-1}&=\delta_1,\\
T_2(D_{1,0})gT_2(D_{2,0}^{-1})&=T_2(S_1^{-1})T_2(S_2) \quad \Rightarrow & (T_2(S_1)S_1^{-1})^{-1}T_2(S_2)S_2^{-1}&=\delta_2^{-1},
\end{align*}
then using the factorization for $\delta_1$ and $\delta_2^{-1}$ and its uniqueness we can identify the upper and lower triangular parts and prove the claimed result.
\end{proof}
It is also possible to define wave matrices $W_1$ and $W_2$ in this discrete context,
\begin{align*}
W_1&:=S_1D_{1,0} & W_2&:=S_2D_{2,0}.
\end{align*}
To ensure the consistency between both continuous and discrete flows we only need to replace $D_{1,0}\mapsto D_{1,0}W_{1,0}$ and $D_{2,0}\mapsto D_{2,0}W_{2,0}$. The next results are valid in case continuous evolution is also present or not, for the proof one only needs a slight modification of the one in  \cite{afm-2}.
\begin{theorem}
\begin{enumerate}
  \item The next linear system for $W_1$ and $W_2$ is satisfied
  \begin{align}\label{disc.lin.sys}
  T_a(W_{a'})&=\omega_a W_{a'} & a,a'&=1,2.
  \end{align}
  \item The discrete versions of the Lax equations are the following
  \begin{align}\label{disc.lax.eq}
  T_a(L_{a'})&=\omega_a L_{a'} \omega_a^{-1} & a,a'&=1,2.
  \end{align}
  \item The compatibility equations for the discrete flows in the linear system \eqref{disc.lin.sys} are
  \begin{align} \label{disc.comp.pure}
  T_1(\omega_2)\omega_1&=T_2(\omega_1)\omega_2
  \end{align}
  if there are also continuous deformation parameters the mixed compatibility equations are
  \begin{align}\begin{aligned}\label{disc.comp.mixed}
  T_a(B_{1,j})&=\frac{\partial \omega_a}{\partial t_{1j}}\omega_a^{-1}+\omega_a B_{1,j} \omega_a^{-1} & a&=1,2 & j&=1,2,\dots \\
  T_a(B_{2,j})&=\frac{\partial \omega_a}{\partial t_{2j}}\omega_a^{-1}+\omega_a B_{2,j} \omega_a^{-1} & a&=1,2 & j&=1,2,\dots \end{aligned}
  \end{align}
\end{enumerate}
\end{theorem}

Now we give examples of some discrete flows operators. Let be $\{\lambda_1(j)\}_{j \in \Z}$ and $\{\lambda_2(j)\}_{j \in \Z}$ two complex sequences with $\lambda_1(j),\lambda_2(j) \in \mathbb{D}$, then
\begin{align*}
D^{(1)}_{1,0}&:= \begin{cases} \Pi_{j=0}^{n_1}(\Upsilon-\lambda_1(j) \mathbb{I}) & n_1>0 \\ \mathbb{I} & n_1=0 \\ \Pi_{j=0}^{|n_1|}(\Upsilon-\lambda_1(-j) \mathbb{I})^{-1} & n_1<0 \end{cases} & (D^{(1)}_{2,0})^{-1}&:=\begin{cases}\Pi_{j=0}^{n_2}(\Upsilon^{\top}- \lambda_2(j) \mathbb{I}) & n_2>0\\ \mathbb{I} & n_2=0 \\ \Pi_{j=0}^{|n_2|}(\Upsilon^{\top}-\lambda_2(-j) \mathbb{I})^{-1} & n_2<0 \end{cases}
\end{align*}
the evolution of the measure is then
\begin{align*}
\d \mu(z,s)=\mathcal{D}^{(1)}_1(z,s_1) (\mathcal{D}_2^{(1)})^{-1}(z,s_2) \d \mu(z)
\end{align*}
where
\begin{align*}
\mathcal{D}_1^{(1)}(z,s)&=\begin{cases} \Pi_{j=0}^{n_1}(z-\lambda_1(j)) & s_1>0 \\ 1 & s_1=0 \\ \Pi_{j=0}^{|s_1|}(z-\lambda_1(-j) )^{-1} & s_1<0 \end{cases} & (\mathcal{D}_2^{(1)})^{-1}(z,s)&=\begin{cases}\Pi_{j=0}^{n_2}(z^{-1}- \lambda_2(j))  & s_2>0 \\ 1 & s_2=0 \\ \Pi_{j=0}^{|s_2|}(z^{-1}-\lambda_2(-j))^{-1}  & s_2<0 \end{cases}
\end{align*}
in that case
\begin{align*}
q^{(1)}_1&=\Upsilon-\lambda_1(s_1+1) \mathbb{I}   & q^{(1)}_2&=\Upsilon^{\top}- \lambda_2(s_2+1) \mathbb{I} \\
\delta^{(1)}_1&=L_1-\lambda_1(s_1+1) \mathbb{I} & \delta^{(1)}_2&=L_2- \lambda_2(s_2+1) \mathbb{I}\label{disc.delta.1}
\end{align*}
The evolution of the wave functions is associated to the evolved Laurent polynomials
\begin{align*}
\Psi_{1}(z,s)&=W_1(s)\chi(z)=S_1(s) D_{1,0}(s) \chi(z)=\Phi_1(z,s)\mathcal{D}_1^{(1)}(z,s) , \\
\Psi_{2}^*(z,s)&=(W_2(s)^{-1})^{\dag}\chi(z)=(S_2(s)^{-1})^\dag (D_{2,0}(s)^{-1})^\dag \chi(z)=\Phi_2(z,s) (\mathcal{D}_2^{(1)})^{-1}(z,s),
\end{align*}
where $\Phi_1(z,t)$ and $\Phi_2(z,t)$ are the Laurent polynomials associated to the evolved measure. 

\begin{lemma} \label{lema.omegas}We have the following structure for the matrices $\omega_1, \omega_2$
\begin{align*}
\omega_1&=\omega_{1,0}+\omega_{1,1}\Lambda+\cdots+\omega_{1,n_-+1}\Lambda^{n_-+1}\\
\omega_2&=\omega_{2,0}+\omega_{2,1}\Lambda^{\top}+\cdots+\omega_{1,n_++1}(\Lambda^{\top})^{n_++1}\\
\omega_1^{\dag}&=\rho_{1,0}+\rho_{1,1}\Lambda^{\top}+\cdots+\rho_{1,n_-+1}(\Lambda^{\top})^{n_-+1}\\
\omega_2^{\dag}&=\rho_{2,0}+\rho_{2,1}\Lambda+\cdots+\rho_{1,n_++1}\Lambda^{n_++1}
\end{align*}
for some semi-infinite matrices
\begin{align*}
\omega_{1,j}&=\diag (\omega_{1,j}(0),\omega_{1,j}(1),\dots) & j=0,\dots,n_-+1\\
\omega_{2,j}&=\diag (\omega_{2,j}(0),\omega_{2,j}(1),\dots) & j=0,\dots,n_++1\\
\rho_{1,j}&=\diag (\rho_{1,j}(0),\rho_{1,j}(1),\dots) & j=0,\dots,n_-+1\\
\rho_{2,j}&=\diag (\rho_{2,j}(0),\rho_{2,j}(1),\dots) & j=0,\dots,n_++1
\end{align*}
\end{lemma}
\begin{proof}
Immediate from \eqref{disc.delta.1}.
\end{proof}
Defining
\begin{align*}
\gamma_1(z,s)&:=z-\lambda_1(s_1+1)   & \gamma_2(z,s)&:=z^{-1}- \lambda_2(s_2+1)
\end{align*}
 the previous Lemma allows us to compute the action of the operators $T_1$ and $T_2$ on the OLPUC  $\varphi_1^{(l)}(z,s),\varphi_2^{(l)}(z,s)$
\begin{pro}
The following equations hold
\begin{align*}
(T_1\varphi_1^{(l)})\gamma_1&=\omega_{1,0}(l)\varphi_1^{(l)}+\omega_{1,1}(l)\varphi_1^{(l+1)}+\cdots+\omega_{1,n_-+1}(l)\varphi_1^{(l+n_-+1)}\\
(T_2\varphi_1^{(l)})&=\omega_{2,0}(l)\varphi_1^{(l)}+\omega_{2,1}(l)\varphi_1^{(l-1)}+\cdots+\omega_{1,n_++1}(l)\varphi_1^{(l-n_+-1)}\\
\varphi_2^{(l)}&=\rho_{1,0}(l)(T_1\varphi_2^{(l)})+\rho_{1,1}(l)(T_1\varphi_2^{(l-1)})+\cdots+\rho_{1,n_-+1}(l)(T_1\varphi_2^{(l-n_--1)})\\
\varphi_2^{(l)}&=\Big(\rho_{2,0}(l)(T_2\varphi_2^{(l)})+\rho_{2,1}(l)(T_2\varphi_2^{(l+1)})+\cdots+\rho_{1,n_++1}(l)(T_2\varphi_2^{(l+n_++1)})\Big)\gamma_2
\end{align*}
\end{pro}
\begin{proof}
The first two equations come from \eqref{disc.lin.sys} and Lemma \ref{lema.omegas}. For the last two equations we use that
\begin{align*}
\omega_a^{\dag}T_a((W_{a'}^{-1})^{\dag})&=(W_{a'}^{-1})^{\dag} & a,a'&=1,2.
\end{align*}
\end{proof}
Another possible option that preserves the reality of the measure is using pairs of conjugate transforms as follows
\begin{align*}
D^{(2)}_{1,0}&:= \begin{cases} \Pi_{j=0}^{s_1}(\Upsilon-\lambda_1(j) \mathbb{I})(\Upsilon^{\top}-\bar \lambda_1(j) \mathbb{I})  & s_1>0 \\ \mathbb{I} & s_1=0 \\ \Pi_{j=0}^{|s_1|}(\Upsilon-\lambda_1(-j) \mathbb{I})^{-1}(\Upsilon^{\top}-\bar \lambda_1(-j) \mathbb{I})^{-1} & s_1<0 \end{cases} \\ (D^{(2)}_{2,0})^{-1}&:=\begin{cases}\Pi_{j=0}^{s_2}(\Upsilon^{\top}- \lambda_2(j) \mathbb{I})(\Upsilon- \bar \lambda_2(j) \mathbb{I}) & s_2>0\\ \mathbb{I} & s_2=0 \\ \Pi_{j=0}^{|s_2|}(\Upsilon^{\top}-\lambda_2(-j) \mathbb{I})^{-1}(\Upsilon -\bar \lambda_2(-j) \mathbb{I})^{-1} & s_2<0 \end{cases}
\end{align*}
the evolution of the measure is then
\begin{align*}
\d \mu(z,s)=\mathcal{D}^{(2)}_1(z,s_1) (\mathcal{D}^{(2)}_2)^{-1}(z,s_2) \d \mu(z)
\end{align*}
where
\begin{align}
\mathcal{D}_1^{(2)}(z,s)&=\begin{cases} \Pi_{j=0}^{s_1}|z-\lambda_1(j)|^2 & s_1>0 \\ 1 & s_1=0 \\ \Pi_{j=0}^{|s_1|}|z-\lambda_1(-j)|^{-2} & s_1<0 \end{cases} & (\mathcal{D}_2^{(2)})^{-1}(z,s)&=\begin{cases}\Pi_{j=0}^{s_2}|z^{-1}- \lambda_2(j)|^2  & s_2>0 \\ 1 & s_2=0 \\ \Pi_{j=0}^{|s_2|}|z^{-1}-\lambda_2(-j)|^{-2}  & s_2<0 \end{cases}
\end{align}
in that case
\begin{align}\begin{aligned}
q^{(2)}_1&=(\Upsilon-\lambda_1(s_1+1) \mathbb{I})(\Upsilon^{\top}-\bar \lambda_1(s_1+1) \mathbb{I})   & q^{(2)}_2&=(\Upsilon^{\top}- \lambda_2(s_2+1) \mathbb{I})(\Upsilon- \bar \lambda_2(s_2+1) \mathbb{I}) \\
\delta^{(2)}_1&=(L_1-\lambda_1(s_1+1) \mathbb{I})(L_1^{-1}-\bar \lambda_1(s_1+1) \mathbb{I}) & \delta^{(2)}_2&=(L_2- \lambda_2(s_2+1) \mathbb{I})(L_2^{-1}- \bar \lambda_2(s_2+1) \mathbb{I})\end{aligned}
\end{align}

Observe that these discrete flows leads to extended Geronimus  transformations \cite{marcellan2}. When the sequences $\{\lambda_1(j)\},\{\lambda_2(j)\}$ are constant and thus $q_1,q_2$ are invariant under the action of $T_1$ and $T_2$ we can make an interpretation in terms of Darboux transformations in the context of \cite{adler-vanmoerbeke-2}. In that case what we obtain is
\begin{align*}
\delta_1&=\delta_{1,-}^{-1}\delta_{1,+}& T_1\delta_1&=T_1(W_1)q_1T_1(W_1^{-1})=\omega_1 \delta_1 \omega_1^{-1}=\delta_{1,+} \delta_{1,-}^{-1}\delta_{1,+}\delta_{1,+}^{-1}=\delta_{1,+} \delta_{1,-}^{-1}\\
\delta_2^{-1}&=\delta_{2,-}^{-1}\delta_{2,+}& T_2\delta_2^{-1}&=T_2(W_2)q_2^{-1}T_2(W_2^{-1})=\omega_2 \delta_2^{-1} \omega_2^{-1}=\delta_{2,-} \delta_{2,-}^{-1}\delta_{2,+}\delta_{2,-}^{-1}=\delta_{2,+} \delta_{2,-}^{-1}
\end{align*}
that is a change in the $LU$ factorization into $UL$.
\subsection{$\tau$-functions}
As is well known $\tau$- functions is an essential ingredient of the theory of integrable systems. Not only for the use of Hirota of these functions in the construction of soliton solutions \cite{hirota} but also for its relevant geometrical insight \cite{date1}-\cite{date3}, also the bilinear equations discussed in the mentioned papers are fundamental in the construction of solutions. The determinantal expressions for the OLPUC, and the associated Laurent polynomials and the corresponding second kind functions lead to a $\tau$-function representation of these objects. For that aim one consider the action of coherent shifts in the time variables, the so called Miwa shifts

\begin{definition}
\begin{enumerate}
  \item The Miwa shifts are the following time shifts
  \begin{align*}
  t \mapsto t\pm[w]_1:=t_{1j} \mapsto t_{1j}\pm\frac{w^j}{j}, \quad t_{2j} \mapsto t_{2j},
  \end{align*}
  \item And the Miwa  dual shifts
  \begin{align*}
  t \mapsto t\pm[w]_2&:=t_{1j} \mapsto t_{1j},\quad  t_{2j} \mapsto t_{2j}\pm\frac{w^j}{j}.
  \end{align*}
\end{enumerate}
\end{definition}
A very important property of this Miwa shifts is the form in what they act on the deformed measure
\begin{pro}\label{Miwa.mu}
The evolved measure $\mu(z,t)$ has the following behaviour
\begin{align*}
\mu(z,t\pm[w^{-1}]_1)&=\left(1-\frac{z}{w}\right)^{\mp1}\mu(z,t), & |z|&< |w|,
\\
\mu(z,t\pm[w]_2)=&\left(1-\frac{w}{z}\right)^{\pm1}\mu(z,t)& |z|&> |w|.
\end{align*}
\end{pro}
\begin{proof}
Using the series expansion of the logarithm, the evolution factors change under Miwa time shifts like
\begin{align*}
\exp(\sum_{j=1}^{\infty}(t_{1j}z^j-t_{2j}z^{-j})) &\mapsto \exp(\sum_{j=1}^{\infty}((t_{1j}\pm\frac{1}{jw^j})z^j-t_{2j}z^{-j}))
=\left(1-\frac{z}{w}\right)^{\mp1}\exp(\sum_{j=1}^{\infty}(t_{1j}z^j-t_{2j}z^{-j})), & |z|&<|w|\\
\exp(\sum_{j=1}^{\infty}(t_{1j}z^j-t_{2j}z^{-j})) &\mapsto \exp(\sum_{j=1}^{\infty}(t_{1j}z^j-(t_{2j}\pm\frac{w^j}{j})z^{-j}))=\left(1-\frac{w}{z}\right)^{\pm1} \exp(\sum_{j=1}^{\infty}(t_{1j}z^j-t_{2j}z^{-j})),& |z|&>|w|.
\end{align*}
\end{proof}
We introduce the main and associated $\tau$-functions as determinants
\begin{definition}
The $\tau$-function is
\begin{align*}
\tau^{(0)}(t)&:=1, & \tau^{(l)}(t)&:=\det g^{[l]}(t), & l&=1,2,\dots,
\end{align*}
while the associated $\tau$-functions are
\begin{align*}
\tau^{(l)}_{1,-a}(t)&:=(-1)^{l+l_{-a}}\det \left( \begin{BMAT}{cccc:ccc}{cccc} g_{ 0,0} & g_{0,1} & \cdots & g_{0,l_{-a}-1} & g_{0,l_{-a}+1} & \cdots & g_{0,l} \\
g_{1,0} & g_{1,1} & \cdots & g_{1,l_{-a}-1} & g_{1,l_{-a}+1} & \cdots & g_{1,l} \\
\vdots & \vdots &  & \vdots & \vdots & & \vdots \\
g_{l-1,0} & g_{l-1,1} & \cdots & g_{l-1,l_{-a}-1} & g_{l-1,l_{-a}+1} & \cdots & g_{l-1,l} \end{BMAT} \right)& l&=|\vec n|, |\vec n|+1,\dots,\\\tau^{(l)}_{2,-a}(t)&:=(-1)^{l+l_{-a}}\det \left( \begin{BMAT}{cccc}{cccc:ccc} g_{ 0,0} & g_{0,1} & \cdots & g_{0,l} \\
g_{1,0} & g_{1,1} & \cdots & g_{1,l} \\
\vdots & \vdots &  &  \vdots \\
g_{l_{-a}-1,0} & g_{l_{-a}-1,1} & \cdots & g_{l_{-a}-1,l}\\
g_{l_{-a}+1,0} & g_{l_{-a}+1,1} & \cdots &  g_{l_{-a}+1,l}\\
\vdots & \vdots &  &  \vdots \\
g_{l,0} & g_{l,1} & \cdots & g_{l,l}\\
 \end{BMAT} \right),& l&=|\vec n|, |\vec n|+1,\dots,
\end{align*}
\begin{align*}
\tau^{(l)}_{2,+a}(t)&:=\det \left( \begin{BMAT}{cccc:c}{cccc} g_{ 0,0} & g_{0,1} & \cdots & g_{0,l-2} & g_{0,(l-1)_{+a}} \\
g_{1,0} & g_{1,1} & \cdots & g_{1,l-2} & g_{1,(l-1)_{+a}} \\
\vdots & \vdots &  & \vdots &\vdots \\
g_{l-1,0} & g_{l-1,1} & \cdots &  g_{l-1,l-2} & g_{l-1,(l-1)_{+a}} \end{BMAT} \right),& l&=|\vec n|, |\vec n|+1,\dots,\\\tau^{(l)}_{1,+a}(t)&:=\det \left( \begin{BMAT}{cccc}{cccc:c} g_{ 0,0} & g_{0,1} & \cdots & g_{0,l-1} \\
g_{1,0} & g_{1,1} & \cdots & g_{1,l-1} \\
\vdots & \vdots &  &  \vdots \\
g_{l-2,0} & g_{l-2,1} & \cdots & g_{l-2,l-1}\\
g_{(l-1)_{+a},0} & g_{(l-1)_{+a},1} & \cdots & g_{(l-1)_{+a
},l-1}\\
 \end{BMAT} \right),& l&=|\vec n|, |\vec n|+1,\dots.
\end{align*}
\end{definition}
To find expressions for the OLPUC in terms of these $\tau$-functions we need the following
\begin{lemma} \label{lemaMiwa}
Let be $r_j(t)$ the j-th row of the matrix $g(t)$, then for $j \in \Z_+ \setminus \{0,n_+\}$
\begin{align*}
r_j(t-[z^{-1}]_1)&=\begin{cases} r_j(t)-z^{-1}r_{(j+1)_{+1}}(t) & a(j)=1,\\
                               r_j(t)-z^{-1}r_{(j-1)_{-2}}(t) & a(j)=2,
                               \end{cases} \\
r_j(t+[z]_2)&=\begin{cases} r_j(t)-z r_{(j-1)_{-1}}(t) & a(j)=1,\\
                               r_j(t)-z r_{(j+1)_{+2}}(t) & a(j)=2,
                               \end{cases} \\
\end{align*}
is satisfied.
For $j=0$ or $j=n_+$ one has
\begin{align*}
r_0(t-[z^{-1}]_1)&=r_0(t)-z^{-1} r_{1_{+1}}(t)\\
r_0(t+[z]_2)&=r_0(t)-z r_{0_{+2}}(t)\\
r_{n_+}(t-[z^{-1}]_1)&=r_{n_+}(t)-z^{-1}r_{0}(t)\\
r_{n_+}(t+[z]_2)&=r_{n_+}(t)-zr_{(n_++1)_{+2}}(t)
\end{align*}
\end{lemma}
\begin{proof}
Immediate from Proposition \ref{Miwa.mu}.
\end{proof}

Let us  recall the skew multi-linear character of  determinants and the consequent formulation in terms  of wedge
products of covectors. Observe that
\begin{lemma} \label{lemacovectors}
Given a set of covectors $\{r_1,\dots,r_n\}$ it can be  shown that
     \begin{align}\label{covectors}
       \bigwedge_{j=1}^n(zr_j-r_{j+1})=\sum_{j=1}^{n+1}(-1)^{n+1-j}z^{j-1} \;r_1\wedge r_2\wedge\dots\wedge \hat
       r_j\wedge\dots \wedge r_{n+1},
     \end{align}
     where the notation $ \hat r_j$ means that we have erased the covector $r_j$ in the wedge product
     $r_1\wedge\dots\wedge r_{n+1}$.
\end{lemma}
\begin{proof}
It can be done directly by induction.
\end{proof}

These two lemmas are the key property to characterize deformed OLPUC using $\tau$-functions
\begin{theorem} \label{th.tau}
Given $l \geq |\vec n|$, one has the following  $\tau$-function representation of  the OLPUC
\begin{align}\label{tau.pol}
\begin{aligned}
\varphi_{1}^{(l)}(z,t)&=\varphi_{1,+1}^{(l)}(z,t)=(S_2)_{ll}\varphi_{1,-1}^{(l)}(z,t)=z^{\nu_+(l)-1}\frac{\tau^{(l)}(t-[z^{-1}]_1)}{\tau^{(l)}(t)},& a(l)&=1,\\
\varphi_{1}^{(l)}(z,t)&=\varphi_{1,+2}^{(l)}(z,t)=(S_2)_{ll}\varphi_{1,-2}^{(l)}(z,t)=z^{-\nu_-(l)}\frac{\tau^{(l)}(t+[z]_2)}{\tau^{(l)}(t)},& a(l)&=2,
\end{aligned}
\end{align}
for the other ``$+$''associated polynomials we have
\begin{align}
\begin{aligned}\label{tau.asspol.plus}
\varphi_{1,+2}^{(l)}(z,t)&=z^{-\nu_-(l_{+2})}\frac{\tau^{(l)}(t+[z]_2)}{\tau^{(l)}(t)},& a(l)&=1,\\
\varphi_{1,+1}^{(l)}(z,t)&=z^{\nu_+(l_{+1})-1}\frac{\tau^{(l)}(t-[z^{-1}]_1)}{\tau^{(l)}(t)},& a(l)&=2,
\end{aligned}
\end{align}
and finally, the remaining ``$-$'' polynomials can be written as follows
\begin{align}
\begin{aligned}\label{tau.asspol.minus}
\varphi_{1,-2}^{(l)}(z,t)&=z^{\nu_+(l)-1}\frac{\tau^{(l)}_{1,-2}(t-[z^{-1}]_1)}{\tau^{(l+1)}(t)},& a(l)&=1,\\
\varphi_{1,-1}^{(l)}(z,t)&=z^{-\nu_-(l)}\frac{\tau^{(l)}_{1,-1}(t+[z]_2)}{\tau^{(l+1)}(t)},& a(l)&=2.
\end{aligned}
\end{align}
\end{theorem}
\begin{proof}
 Let us prove \eqref{tau.pol}. If $a(l)=1$ we can use Lemma \ref{lemacovectors} with $r_1=r_{(l-1)_{-2}}$ and $r_n=r_{(l-1)_{-1}}$ to expand
 \begin{align*}
 z^{\nu_+(l)-1} \tau^{(l)}(t-[z^{-1}]_1)&=z^{\nu_+(l)-1}(M_{ll}^{(l+1)}+(-1)^{l+(l-1)_{-1}}z^{-1}M_{(l-1)_{-1}l}^{(l+1)} +\cdots+(-1)^{l+(l-1)_{-2}}z^{-l}M_{(l-1)_{-2}l}^{(l+1)})\\
 &=\sum_{j=0}^{l}(-1)^{l+j}M_{jl}^{(l+1)}\chi^{(j)}(z)\\
 &=\det (g^{[l]}(t)) \varphi_{1}^{(l)}(z,t)\\
 &=\tau^{(l)}(t) \varphi_{1}^{(l)}(z,t).
 \end{align*}
If $a(l)=2$ the same procedure works with $r_1=r_{(l-1)_{-1}}$ and $r_n=r_{(l-1)_{-2}}$. Now the expansion is
\begin{align*}
z^{-\nu_-(l)} \tau^{(l)}(t+[z]_2)&=z^{-\nu_-(l)} (M_{ll}^{(l+1)}+(-1)^{l+(l-1)_{-2}}z M_{(l-1)_{-2}l}^{(l+1)} +\cdots+(-1)^{l+(l-1)_{-1}}z^l M_{(l-1)_{-1}l}^{(l+1)})\\
&=\sum_{j=0}^{l}(-1)^{l+j}M_{jl}^{(l+1)}\chi^{(j)}(z)\\
&=\det (g^{[l]}(t)) \varphi_{1}^{(l)}(z,t)\\
&=\tau^{(l)}(t) \varphi_{1}^{(l)}(z,t).
\end{align*}
The proof of expressions in \eqref{tau.asspol.plus} can be performed using the same technique, expanding the right hand side. If $a(l)=1$ we have to consider Lema \ref{lemacovectors} with the rows $r_1=r_{(l-1)_{-1}}$ and $r_n=r_{(l-1)_{-2}}$. In the case $a(l)=2$ the rows $r_1=r_{(l-1)_{-2}}$ and $r_n=r_{(l-1)_{-1}}$ should be considered.

To conclude we prove \eqref{tau.asspol.minus}, repeating previous arguments with adequate $\tau$-functions, that is, in the case $a(l)=1$
\begin{align*}
 z^{\nu_+(l)-1} \tau^{(l)}_{1,-2}(t-[z^{-1}]_1)=&z^{\nu_+(l)-1}\Big((-1)^{l+l_{-2}}M_{ll_{-2}}^{(l+1)}+(-1)^{l_{-2}+(l-1)_{-1}}z^{-1}M_{(l-1)_{-1}l_{-2}}^{(l+1)} +\\
 &+\cdots+(-1)^{l_{-2}+(l-1)_{-2}}z^{-l}M_{(l-1)_{-2}l_{-2}}^{(l+1)}\Big)\\
 =&\sum_{j=0}^{l}(-1)^{l_{-2}+j}M_{jl_{-2}}^{(l+1)}\chi^{(j)}(z)\\
 =&\tau^{(l+1)}(t) \varphi_{1,-2}^{(l)}(z,t).
 \end{align*}
 and in the case $a(l)=2$ then
 \begin{align*}
z^{-\nu_-(l)} \tau^{(l)}_{1,-1}(t+[z]_2)=&z^{-\nu_-(l)} \Big((-1)^{l+l_{-1}}M_{ll_{-1}}^{(l+1)}+(-1)^{l_{-1}+(l-1)_{-2}}z M_{(l-1)_{-2}l_{-1}}^{(l+1)} +\\
&+\cdots+(-1)^{l_{-1}+(l-1)_{-1}}z^l M_{(l-1)_{-1}l_{-1}}^{(l+1)}\Big)=\\
=&\sum_{j=0}^{l}(-1)^{l_{-1}+j}M_{jl_{-1}}^{(l+1)}\chi^{(j)}(z)\\
=&\tau^{(l+1)}(t) \varphi_{1,-1}^{(l)}(z,t).
\end{align*}
\end{proof}
To obtain the $\tau$-function representation of the dual Laurent polynomials $\varphi_{ 2}^{(l)}$ and their associated ones we would proceed using again Lemma \ref{lemaMiwa} and Lemma \ref{lemacovectors} interchanging the role of rows and columns, that would lead to the final expression
\begin{theorem}\label{th.dual.tau}
For any $l \geq |\vec n|$ the dual Laurent polynomials $\varphi_{2}$ have the following expressions in terms of $\tau$-functions
\begin{align*}
\begin{aligned}
\overline {\varphi_{2}^{(l)}(z,t)}&=(S_2)_{ll}^{-1}\overline{ \varphi_{2,+1}^{(l)}(z,t)}=\overline{ \varphi_{2,-1}^{(l)}(z,t)}=\bar z^{\nu_+(l)-1}\frac{\tau^{(l)}(t+[\bar z^{-1}]_2)}{\tau^{(l+1)}(t)},& a(l)&=1,\\
\overline{ \varphi_{2}^{(l)}(z,t)}&=(S_2)_{ll}^{-1} \overline{ \varphi_{2,+2}^{(l)}(z,t)}=\overline{ \varphi_{2,-2}^{(l)}(z,t)}=\bar z^{-\nu_-(l)}\frac{\tau^{(l)}(t-[\bar z]_1)}{\tau^{(l+1)}(t)},& a(l)&=2,
\end{aligned}
\end{align*}
the ``+'' labeled associated polynomials can be written as
\begin{align*}
\begin{aligned}
\overline{ \varphi_{2,+2}^{(l)}(z,t)}&=\bar z^{-\nu_-(l_{+2})}\frac{\tau^{(l)}(t-[\bar z]_1)}{\tau^{(l)}(t)},& a(l)&=1,\\
\overline{ \varphi_{2,+1}^{(l)}(z,t)}&=\bar z^{\nu_+(l_{+1})-1}\frac{\tau^{(l)}(t+[\bar z^{-1}]_2)}{\tau^{(l)}(t)},& a(l)&=2,
\end{aligned}
\end{align*}
to conclude, the ``$-$'' labeled polynomials have the following representation
\begin{align*}
\begin{aligned}
\overline{ \varphi_{2,-2}^{(l)}(z,t)}&=\bar z^{\nu_+(l)-1}\frac{\tau^{(l)}_{2,-2}(t+[\bar z^{-1}]_2)}{\tau^{(l+1)}(t)},& a(l)&=1,\\
\overline{ \varphi_{2,-1}^{(l)}(z,t)}&=\bar z^{-\nu_-(l)}\frac{\tau^{(l)}_{2,-1}(t-[\bar z]_1)}{\tau^{(l+1)}(t)},& a(l)&=2.
\end{aligned}
\end{align*}
\end{theorem}

We will end this section with results regarding the $\tau$-function representation of the second kind functions in the way we did in \cite{afm-2}.
\begin{lemma}\label{lemacovectors2}
The following  identity
\begin{align}\label{covectors.2}
  \bigwedge_{j=1}^n
  \Big(\sum_{i=0}^{\infty}r_{j+i}z^{-i}\Big)=r_1\wedge\dots\wedge r_{n-1}\wedge \Big(\sum_{i=0}^\infty
  r_{n+i}z^{-i}\Big)
\end{align}
holds.
\end{lemma}
\begin{proof}
Use induction in $n$.
\end{proof}
\begin{theorem}\label{th.tau.cauchy}
Let $\mu$ be a positive measure supported in $\T$, then the following statements hold true
\begin{enumerate}
\item The second kind functions have the following representation involving $\tau$-functions
\begin{align}\label{tau.cauchy}
\begin{aligned}
\overline{C_{1,1}^{(l)}(z,t)}&=\bar z^{-\nu_+(l_{+1})}\frac{\tau_{1,+1}^{(l+1)}(t+[\bar z^{-1}]_1)}{\tau^{(l+1)}(t)}, &R_-<&|z|,\\
\overline{C_{1,2}^{(l)}(z,t)}&=\bar z^{\nu_-(l_{+2})-1}\frac{\tau_{1,+2}^{(l+1)}(t-[\bar z]_2)}{\tau^{(l+1)}(t)},&&|z|<R_+,\\
\overline{C_{1}^{(l)}(z,t)}&=\frac{\bar z^{-\nu_+(l_{+1})}\tau_{1,+1}^{(l+1)}(t+[\bar z^{-1}]_1)+\bar z^{\nu_-(l_{+2})-1}\tau_{1,+2}^{(l+1)}(t-[\bar z]_2)}{\tau^{(l+1)}(t)},& R_-<&|z|<R_+.
\end{aligned}
\end{align}
\begin{align}\label{tau.dual.cauchy}
\begin{aligned}
C_{2,1}^{(l)}(z,t)&= z^{-\nu_+(l_{+1})}\frac{\tau_{2,+1}^{(l+1)}(t-[z^{-1}]_2)}{\tau^{(l)}(t)}, &R_+^{-1}<&|z|,\\
C_{2,2}^{(l)}(z,t)&= z^{\nu_-(l_{+2})-1}\frac{\tau_{2,+2}^{(l+1)}(t+[z]_1)}{\tau^{(l)}(t)},&&|z|<R_-^{-1}\\
C_{2}^{(l)}(z,t)&=\frac{z^{-\nu_+(l_{+1})}\tau_{2,+1}^{(l+1)}(t-[z^{-1}]_2)+z^{\nu_-(l_{+2})-1}\tau_{2,+2}^{(l+1)}(t+[z]_1)}{\tau^{(l)}(t)},& R_+^{-1}<&|z|<R_-^{-1}.
\end{aligned}
\end{align}
\item For $ R_-<|z|<R_+$ the Fourier series of the measure can be expressed in terms of $\tau$-functions in the following way
\begin{align}\label{fourier.tau}
\begin{aligned}
F_{\mu(t)}(z)&=\frac{\tau_{2,+1}^{(l+1)}(t-[z]_2)+z^{-\nu_+(l_{+1})-\nu_-(l_{+2})+1}\tau_{2,+2}^{(l+1)}(t+[z^{-1}]_1)}{2\pi \tau^{(l)}(t-[z^{-1}]_1)}\\&=\frac{\tau_{1,+1}^{(l+1)}(t+[ z^{-1}]_1)+ z^{\nu_+(l_{+1})+\nu_-(l_{+2})-1}\tau_{1,+2}^{(l+1)}(t-[ z]_2)}{2\pi \tau^{(l)}(t+[ z]_2)}, & a(l)&=1, \\
F_{\mu(t)}( z)&=\frac{z^{\nu_+(l_{+1})+\nu_-(l_{+2})-1}\tau_{2,+1}^{(l+1)}(t-[z]_2)+\tau_{2,+2}^{(l+1)}(t+[z^{-1}]_1)}{2\pi \tau^{(l)}(t+[z]_2)}\\&=
 \frac{ z^{1-\nu_+(l_{+1})-\nu_-(l_{+2})}\tau_{1,+1}^{(l+1)}(t+[ z^{-1}]_1)+\tau_{1,+2}^{(l+1)}(t-[ z]_2)}{2\pi \tau^{(l)}(t-[ z^{-1}]_1)}, & a(l)&=2,
 \end{aligned}
\end{align}
\end{enumerate}
\end{theorem}
\begin{proof}
We will prove \eqref{tau.dual.cauchy} only, and the proof of \eqref{tau.cauchy} that goes analogously is left to the reader.
The expression from Proposition \ref{det.Cauchy} can be arranged using the truncated columns of the moment matrix, that is $c_j^{[l]}:=\int_{\T}\chi^{[l]}(\chi^{(l)})^\dag \d \mu(t)$. Using this notation
\begin{align*}
g^{[l]}(t)C_{2,1}^{(l)}(z,t)=\tau^{(l)}(t)C_{2,1}^{(l)}(z,t)&=c_0^{[l]} \wedge c_1^{[l]}\wedge \dots \wedge c_{l-1}^{[l]} \wedge \sum_{j=l}^{\infty} c_j^{[l]} (\chi^*_1)^{(j)}\\
&=c_0^{[l]} \wedge c_1^{[l]}\wedge \dots \wedge c_{l-1}^{[l]} \wedge \sum_{j=l_{+1}}^{\infty} z^{-\nu_+(j)} c_j^{[l]}\delta_{a(j),1}.
\end{align*}
We define  $\Z_{+,i}:=\{j \in \Z_+, a(j)=i\}$, $i=1,2$ and notice that $\Z_+=\Z_{+,1}\cup \Z_{+,2}$. The  restrictions $\nu_{+}|_{\Z_{+,1}}, \nu_-|_{\Z_{+,2}}$ of the mappings $\nu_+, \nu_- : \Z_+ \mapsto \N$  are bijections; hence, they have a well defined  inverse, $(\nu_+)^{-1}$ and $(\nu_-)^{-1}$. Therefore,
\begin{align*}
\tau^{(l)}(t)C_{2,1}^{(l)}(z,t)&=c_0^{[l]} \wedge c_1^{[l]}\wedge \dots \wedge c_{l-1}^{[l]} \wedge z^{-\nu_+(l_{+1})}\sum_{j=0}^{\infty} z^{-j} c_{(\nu_+)^{-1}(\nu_+(l_{+1}+j))}^{[l]}\\
&= z^{-\nu_+(l_{+1})} c_0^{[l]} \wedge c_1^{[l]}\wedge \dots \wedge c_{l-1}^{[l]} \wedge \sum_{j=0}^{\infty} z^{-j} c_{(\nu_+)^{-1}(\nu_+(l_{+1}+j))}^{[l]}\\
&=z^{-\nu_+(l_{+1})}\tau_{2,+1}^{(l+1)}(t-[z^{-1}]_2),
\end{align*}
where the last step requires the use of Lemma \ref{lemacovectors2} with $c_0^{[l]}, c_1^{[l]},\dots, c_{l-1}^{[l]}, c_{l_{+1}}^{[l]}$. Proceeding in a very similar way, we have
\begin{align*}
\det (g^{[l]}(t))C_{2,2}^{(l)}(z,t)=\tau^{(l)}(t)C_{2,2}^{(l)}(z,t)&=c_0^{[l]} \wedge c_1^{[l]}\wedge \dots \wedge c_{l-1}^{[l]} \wedge \sum_{j=l}^{\infty} c_j^{[l]} \chi_2^{(j)}\\
&=c_0^{[l]} \wedge c_1^{[l]}\wedge \dots \wedge c_{l-1}^{[l]} \wedge \sum_{j=l_{+2}}^{\infty} z^{\nu_(j)-1} c_j^{[l]}\delta_{a(j),2}\\
&=c_0^{[l]} \wedge c_1^{[l]}\wedge \dots \wedge c_{l-1}^{[l]} \wedge z^{\nu_(l_{+2})-1}\sum_{j=0}^{\infty} z^{j} c_{(\nu_-)^{-1}(\nu_-(l_{+2}+j))}^{[l]}\\
&= z^{\nu_-(l_{+2})-1} c_0^{[l]} \wedge c_1^{[l]}\wedge \dots \wedge c_{l-1}^{[l]} \wedge \sum_{j=0}^{\infty} z^{j} c_{(\nu_-)^{-1}(\nu_-(l_{+2}+j))}^{[l]}\\
&=z^{\nu_-(l_{+2})}\tau_{2,+2}^{(l+1)}(t+[z]_1).
\end{align*}
with $c_0^{[l]},c_1^{[1]},\dots,c_{l-1}^{[l_{+2}]}$ as adequate entries for Lemma \ref{lemacovectors2}. Finally,  \eqref{fourier.tau}  are obtained combining equations \eqref{tau.cauchy} and \eqref{tau.dual.cauchy} with Proposition \ref{proC}, Theorem \ref{th.tau} and Theorem \ref{th.dual.tau}.
\end{proof}
A comment on the convergence of the expressions in Theorem \ref{th.tau.cauchy} and it's proof is needed  at this point. The main tool used in the proof is the series expansion of $(1-z)^{-1}$, that is convergent only for $ |z| <1$ but can be analytically extended outside $\T$. For instance, the $\tau$-function expression for $C_{2,1}^{(l)}$ is only strictly valid outside $\T$. Nevertheless, it can be analytically extended inside the circle up to $R_-$ that is where the series for $C_{2,1}$ is convergent. As this extension is unique, we can talk about the analytically extended ``Miwa-shifted'' $\tau$-function. The same can be said about the other equalities involving $\tau$-functions; they are formally correct and convergent outside or inside $\T$, but there is an analytical continuation for the shifted $\tau$-functions that converges where the Cauchy transforms do. The differences between expressions for the Cauchy transforms \eqref{tau.cauchy} and \eqref{tau.dual.cauchy} and their equivalent in the real line (e.g. \cite{adler-vanmoerbeke-5},\cite{afm-2}) is due to the existence of a positive and a negative part in the Laurent expansion around $z=0$. The series expansion of $(1-z)^{-1}$ generates the positive part and the expansion of $(1-z^{-1})^{-1}$ gives a negative power series that generates the singular part of the Laurent expansion.

\subsection{Bilinear equations}

 For the derivation of a bilinear identity we proceed similarly as we did in \cite{afm-2} proving several lemmas. For the first one, let  $W_1, W_2$ be  the wave matrices associated with the moment matrix $g$; so that, $W_1g=W_2$. Then, we have
\begin{lemma}
  \label{lemma: matrix bilinear}
  The wave matrices associated with different times satisfy
  \begin{align}\label{bilinear-wave}
 W_1(t) W_1(t')^{-1}=W_2(t) W_2(t')^{-1},
\end{align}
  \end{lemma}
  \begin{proof}
    We consider simultaneously the following equations
    \begin{align*}
      W_1(t)g&=W_2(t),\\
      W_1(t') g &= W_2(t'),
    \end{align*}
   and we get
    \begin{align*}
      W_1(t)^{-1}W_2(t)= W_1(t')^{-1} W_2(t')=g,
    \end{align*}
    and the result becomes evident.
  \end{proof}
\begin{lemma}\label{chi-id}
  \begin{enumerate}
    \item For the vectors $\chi,\chi^* $ the following formulae hold
  \begin{align*}
     \operatorname{Res}_{z=0} \big(\chi(\chi^*)^\top\big)=\operatorname{Res}_{z=0} \big(\chi^*\chi^\top\big)=\I,
  \end{align*}
  \item \cite{afm-2} For any couple of semi-infinite matrices $U$ and $V$ we have
\begin{align}\label{residue1}
UV&=     \operatorname{Res}_{z=0}\Big((U\chi) \big(V^\top\chi^*\big)^\top \Big)   \\ \label{residue2}
&=     \operatorname{Res}_{z=0}\Big((U\chi^*)\big(V^\top\chi\big)^\top \Big)
\end{align}
  \end{enumerate}
 \end{lemma}
We have the following
\begin{theorem}\label{theorem: bilinear}
For any $t,t'$
\begin{enumerate}
  \item Wave functions  satisfy
\begin{align*}
     \operatorname{Res}_{z=0}\Big(\Psi_1(z,t) (\Psi_1^*(\bar z,t'))^\dag \Big)=
     \operatorname{Res}_{z=0}\Big(\Psi_2(z,t) (\Psi_2^*(\bar z,t'))^\dag \Big)
\end{align*}
\item OLPUC fulfill
 \begin{multline}\label{bilinear2}
\operatorname{Res}_{z=0}\Big(
\varphi_{1}^{(k)}(z,t) \bar\varphi_{2}^{(l)}(z^{-1},t') z^{-1}F_\mu(z)\Exp{\sum_{j=1}^{\infty}(t_{1j}z^j-t'_{2j}z^{-j})}\Big)\\=
\operatorname{Res}_{z=\infty}\Big(\varphi_{1}^{(k)}(z,t) \bar\varphi_{2}^{(l)}(z^{-1},t')z^{-1}F_\mu(z) \Exp{\sum_{j=1}^{\infty}(t_{1j}z^{j}-t'_{2j}z^{-j})}\Big),
\end{multline}
\end{enumerate}
\end{theorem}
\begin{proof}
  \begin{enumerate}
    \item   First we notice that \eqref{residue1} and \eqref{residue2} can be written as
    \begin{align*}
UV&=     \operatorname{Res}_{z=0}\Big((U\chi(z)) \big(V^\dag \chi^*(\bar z)\big)^\dag \Big)   \\
&=     \operatorname{Res}_{z=0}\Big((U\chi^*(z))\big(V^\dag \chi(\bar z)\big)^\dag \Big)
\end{align*}
    If we set in  \eqref{residue1} $U=W_{1}(t)$ and $V=W_{1}(t')^{-1}$
    and in \eqref{residue2} we put $U= W_{2}(t)$ and $V=
    W_{2}(t')^{-1}$
    attending to
  \eqref{bilinear-wave}, recalling that $\Psi_{1}=W_{1}\chi$,
  $ \Psi_{2}= W_{2}\chi^*$ and observing that
  $\Psi_{1}^*=(W_{1}^{-1})^\dag \chi^*$
  and
  $\Psi_{2}^*=(W_{2}^{-1})^\dag \chi $ we get  the stated  bilinear equation for the wave functions.
\item
 We can substitute the expressions \eqref{evol.baker} and \eqref{evol.cauchy} to prove the second part of the result.
  \end{enumerate}
 We can reformulate this result using the residue theorem
\end{proof}
\begin{cor}
If $R_-=0$ and $R_+=\infty$ and we take $\gamma_0$ and $\gamma_\infty$ small zero-index cycles around $z=0$ and $z=\infty$, respectively; then
\begin{align}\label{bilinear.int}
     \oint_{\gamma_0}\Psi_1^{(n)}(z,t) (\bar \Psi_1^*)^{(m)}(z,t') \d z&=
     \oint_{\gamma_0}\Psi_2^{(n)}(z,t) (\bar \Psi_2^*)^{(m)}(z,t') \d z,
\end{align}
\begin{multline*}
\oint_{\gamma_0}
\varphi_{1}^{(k)}(z,t) \bar\varphi_{2}^{(l)}(z^{-1},t')\Exp{\sum_{j=1}^{\infty}(t_{1j}z^j-t'_{2j}z^{-j})} z^{-1}F_\mu(z)\d z \\=
\oint_{\gamma_\infty} \varphi_{1}^{(k)}(z,t) \bar\varphi_{2}^{(l)}(z^{-1},t')\Exp{\sum_{j=1}^{\infty}(t_{1j}z^j-t'_{2j}z^{-j})} z^{-1}F_\mu(z)\d z,\\
\end{multline*}
Alternatively the bilinear equation can be expressed using $\tau$-functions
\begin{multline*}
\oint_{\gamma_0}\tau_{1,+1}^{(l)}(t-[z^{-1}]_1)\Big(\tau_{1,+1}^{(l+1)}(t+[z^{-1}]_1)+z^{\nu_+(l_{+1})+\nu_-(l_{+2})-1}\tau^{(l+1)}_{1,+2}(t-[z]_2)\Big)\Exp{\sum_{j=1}^{\infty}(t_{1j}-t'_{1j})z^j}\frac{\d z}{z}\\
=\oint_{\gamma_0}\tau^{(l)}_{2,+1}(t+[z^{-1}]_2)\Big(\tau_{2,+1}^{(l+1)}(t-[z^{-1}]_2)+z^{\nu_+(l_{+1})+\nu_-(l_{+2})-1}\tau_{2,+2}^{(l+1)}(t+[z]_1)\Big)\Exp{\sum_{j=1}^{\infty}(t_{2j}-t'_{2j})z^j} \frac{\d z}{z},
\end{multline*}
if a(l)=1 and
\begin{multline*}
\oint_{\gamma_0}\tau_{1,+2}^{(l)}(t+[z]_2)\Big(z^{-\nu_+(l_{+1})-\nu_-(l_{+2})+1}\tau_{1,+1}^{(l+1)}(t+[z^{-1}]_1)+\tau^{(l+1)}_{1,+2}(t-[z]_2)\Big)\Exp{\sum_{j=1}^{\infty}(t_{1j}-t'_{1j})z^j}\frac{\d z}{z}\\
=\oint_{\gamma_0}\tau^{(l)}_{2,+2}(t-[z]_1)\Big(z^{-\nu_+(l_{+1})-\nu_-(l_{+2})+1}\tau_{2,+1}^{(l+1)}(t-[z^{-1}]_2)+\tau_{2,+2}^{(l+1)}(t+[z]_1)\Big)\Exp{\sum_{j=1}^{\infty}(t_{2j}-t'_{2j})z^j} \frac{\d z}{z},
\end{multline*}
if a(l)=2.
\end{cor}

\section*{Acknowledgements}
MM thanks economical support from the Spanish Ministerio de Ciencia e Innovaci\'{o}n, research
project FIS2008-00200.

\begin{appendices}
\section{Proofs}\label{proofs}
\paragraph{Proof of Proposition \ref{pro.1}}
   In this case $g$ is positive definite and Hermitian,   if we write $S_2=h\hat S_2$, where $h=\diag(h_0,h_1,\dots)$ is a diagonal matrix and $\hat S_2=\left(\begin{smallmatrix}
  1& (\hat S_2)_{01}&(\hat S_2)_{02}&\dots\\
  0&1&(\hat S_2)_{12}&\dots\\
  0&0&1&\\
  \vdots&\vdots&
\end{smallmatrix}\right)$, the uniqueness of the factorization implies that $\hat S_2=(S_1^{-1})^\dagger$ and $h_l\in\R$ and we get the stated result.

\paragraph{Proof of Proposition \ref{pro.det.lp}}

Expressions \eqref{syslaurant1}, \eqref{syslaurant2}, \eqref{sysduallaurant1} and \eqref{sysduallaurant2} are obtained expressing the factorization problem as a system of equations. From \eqref{syslaurant2} we deduce
\begin{align*}
(S_1)_{lk}=(S_2)_{ll}((g^{[l+1]})^{-1})_{l,k}=\frac{(S_2)_{ll}}{\det g^{[l+1]}}(-1)^{l+k}M_{k,l}^{(l+1)},
\end{align*}
so that
\begin{align*}
\varphi_1^{(l)}(z)=\sum_{k=0}^l (S_1)_{lk} \chi^{(k)}= \frac{1}{\det g^{[l]}} \sum_{k=0} (-1)^{l+k}M_{k,l}^{(l+1)}\chi^{(k)},
\end{align*}
as stated in \eqref{detlaurant}.  To prove  \eqref{detduallaurant} we  consider
\begin{align*}
(S_2^{-1})_{kl}=((g^{[l+1]})^{-1})_{k,l}=\frac{1}{\det g^{[l+1]}}(-1)^{l+k}M_{l,k}^{(l+1)},
\end{align*}
so that
\begin{align*}
\bar \varphi_2^{(l)} (\bar z)=\sum_{k=0}^l  (S_2^{-1})_{kl} (\chi^{(k)})^\dag = \frac{1}{\det g^{[l+1]}} \sum_{k=0}^l (-1)^{l+k}M_{l,k}^{(l+1)}(\chi^{(k)})^\dag,
\end{align*}
that leads to \eqref{detduallaurant}.
\paragraph{Proof of Proposition \ref{det.Cauchy}}
Using the definition for $C_1$, we have that
\begin{align*}
C_1^{(l)}(z)=\sum_{k \geq l} (S_1^{-1})^{\dag}_{lk}\chi^{*(k)}(z)=\sum_{j=0}^{l} (S_2^{-1})^{\dag}_{lj} \sum_{k \geq l} g_{jk}^{\dag} \chi^{*(k)}(z)= \sum_{j=0}^{l} (S_2^{-1})^{\dag}_{lj} \Gamma_{2,j}^{(l)}(z),
\end{align*}
so
\begin{align*}
\overline{ C_1^{(l)}(z)}=\sum_{j=0}^{l} (S_2^{-1})_{jl} \bar\Gamma_2^{(j)}(\bar z).
\end{align*}
For the other set of functions we have
\begin{align*}
C_2^{(l)}(z)=\sum_{k \geq l} (S_2)_{l,k}\chi^{*(k)}(z)=\sum_{j=0}^{l}(S_1)_{lj} \sum_{k \geq l} g_{jk} \chi^{*(k)}(z)=\sum_{j=0}^{l}(S_1)_{lj} \Gamma_{1,j}^{(l)}(z).
\end{align*}
Comparing the expressions with those in Proposition \ref{pro.det.lp} we see that they are formally identical, so we conclude the stated result.

\paragraph{Proof of Proposition \ref{pro.conv.gamma}}
Using the Fourier coefficients of the measure and the definition for the $\Gamma_{a,j}^{(l)}$ we obtain
\begin{align*}
\Gamma^{(0)}_{1,j}(z)&=\sum_{k \geq 0}g_{j,k}\chi^{*(k)}(z)=\sum_{k \geq 0}g_{j,k}\chi_1^{*(k)}(z)+\sum_{k \geq 0}g_{j,k}\chi_2^{(k)}(z)\\
&=\sum_{k \geq 0} \int_{0}^{2\pi} e^{\im (J(j)-k)\theta} \d \mu(\theta) z^{-k-1}+\sum_{k \geq 0} \int_{0}^{2\pi} e^{\im (J(j)+k+1)\theta} \d \mu(\theta) z^{k}\\
&=2\pi \sum_{k \geq 0}(c_{k-J(j)} z^{-k-1}+c_{-k-J(j)-1} z^k)\\
&=2 \pi z^{-J(j)-1}\Big(\sum_{k \geq -J(j)}c_k (z^{-1})^{k}+\sum_{k< -J(j)}c_k (z^{-1})^k\big)\\
&=2 \pi z^{-J(j)-1} F_{\mu}(z^{-1}).
\end{align*}
\begin{align*}
\Gamma^{(0)}_{2,j}(z)&=\sum_{k \geq 0}g^{\dag}_{j,k}\chi^{*(k)}(z)=\sum_{k \geq 0}g^{\dag}_{j,k}\chi_1^{*(k)}(z)+\sum_{k \geq 0}g^{\dag}_{j,k}\chi_2^{(k)}(z)\\
&=\sum_{k \geq 0} \int_{0}^{2\pi} e^{\im (J(j)-k)\theta} \d \bar\mu(\theta) z^{-k-1}+\sum_{k \geq 0} \int_{0}^{2\pi} e^{\im (J(j)+k+1)\theta} \d \bar\mu(\theta) z^{k}\\
&=2\pi \sum_{k \geq 0}(\overline{c_{J(j)-k}} z^{-k-1}+\overline{c_{k+J(j)+1}} z^k)\\
&=2 \pi z^{-J(j)-1}\Big(\sum_{k < J(j)+1}\overline{c_k} z^k+\sum_{k \geq J(j)+1}\overline{c_k} z^k\Big)\\
&=2 \pi z^{-J(j)-1} \bar F_{\mu}(z).
\end{align*}
From here we obtain the rest of the expressions
\begin{align*}
\Gamma^{(l)}_{1,j}(z)&=2 \pi z^{-J(j)-1}(\sum_{k \geq -J(j)+l}c_k (z^{-1})^{k}+\sum_{k< -J(j)-l}c_k (z^{-1})^k)=2 \pi z^{-J(j)-1} \Big(F^{(+)}_{J(j)-l,\mu}(z^{-1})+F^{(-)}_{J(j)+l,\mu}(z^{-1})\Big),\\
\Gamma^{(l)}_{2,j}(z)&=2 \pi z^{-J(j)-1}\Big(\sum_{k < J(j)+1-l}\overline{c_k} z^k+\sum_{k \geq J(j)+1+l}\overline{c_k} z^k\Big)=2 \pi z^{-J(j)-1} \Big(\bar F^{(-)}_{l-J(j)-1,\mu}(z)+\bar F^{(+)}_{-l-J(j)-1,\mu}(z)\Big).
\end{align*}

\paragraph{Proof of Proposition \ref{scf}}

From the definitions we have \begin{align*}
    (C_{1,1})^\dagger (z)\varphi_{1,1}(z')=\chi_1^*(z)S_1^{-1}S_1\chi_1(z')=(\chi_1^*)^\dagger(z)\chi_1(z')&=\sum_{n=0}^\infty z^{-n-1}(z')^n,\\
      (C_{2,1})^\dagger (z)\varphi_{2,1}(z')=\chi_1^*(z)S_2^\dagger  ( S_2^\dagger)^{-1}\chi_1(z')=(\chi_1^*)^\dagger(z)\chi_1(z')&=\sum_{n=0}^\infty z^{-n-1}(z')^n,\\
        (C_{1,2})^\dagger (z)\varphi_{1,2}(z')=\chi_2(z)S_1^{-1}S_1\chi_2^*(z')=(\chi_2^*)^\dagger(z)\chi_2(z')&=\sum_{n=0}^\infty z^{n}(z')^{-n-1},\\
      (C_{2,2})^\dagger (z)\varphi_{2,2}(z')=\chi_2(z)S_2^\dagger  ( S_2^\dagger)^{-1}\chi^*_2(z')=(\chi_1^*)^\dagger(z)\chi_1(z')&=\sum_{n=0}^\infty z^{n}(z')^{-n-1},
  \end{align*}
  which, after the study the region of convergence of the series involved, leads to the sated result.
  Then other identities derived from $(\chi_1^*)^\dagger(z)\chi_2(z')=(\chi_2^*)^\dagger(z)\chi_1(z')=0$.

\paragraph{Proof of Proposition \ref{explicit-J}}

It can be deduced as follows
  \begin{align*}
J_{2k,2k+2}&=(S_1E_{2k,2k+2}S_1^{-1})_{2k,2k+2}=(S_1)_{2k,2k}(S_1^{-1})_{2k+2,2k+2}=1,\\
J_{2k,2k+1}&=(S_1E_{2k,2k+2}S_1^{-1})_{2k,2k+1}=(S_1)_{2k,2k}(S_1^{-1})_{2k+2,2k+1}=-(S_1)_{2k+2,2k+1}=-\alpha^{(1)}_{2k+2},\\
J_{2k,2k}&=(S_2E_{2k+1,2k-1}S_2^{-1})_{2k,2k}=(S_2)_{2k,2k+1}(S_2^{-1})_{2k-1,2k}=\\
 &=-(S_2)_{2k,2k}(S_2^{-1})_{2k,2k+1}(S_2)_{2k+1,2k+1}(S_2^{-1})_{2k-1,2k}=-\bar \alpha_{2k}^{(2)}\alpha_{2k+1}^{(1)},\\
J_{2k,2k-1}&=(S_2E_{2k+1,2k-1}S_2^{-1})_{2k,2k-1}=(S_2)_{2k,2k+1}(S_2^{-1})_{2k-1,2k-1}=\\
 &=-(S_2)_{2k,2k}(S_2^{-1})_{2k,2k+1}(S_2)_{2k+1,2k+1}(S_2^{-1})_{2k-1,2k-1}=-\rho^2_{2k} \alpha^{(1)}_{2k+1}.
\end{align*}
\begin{align*}
J_{2k+1,2k-1}&=(S_2E_{2k+1,2k-1}S_2^{-1})_{2k+1,2k-1}=(S_2)_{2k+1,2k+1}(S_2^{-1})_{2k,2k}(S_2)_{2k,2k}(S_2^{-1})_{2k-1,2k-1}\\
&=\rho^2_{2k+1}\rho^2_{2k},\\
J_{2k+1,2k}&=(S_2E_{2k+1,2k-1}S_2^{-1})_{2k+1,2k}=(S_2)_{2k+1,2k+1}(S_2^{-1})_{2k-1,2k}\\
&=(S_2)_{2k+1,2k+1}(S_2^{-1})_{2k,2k}(S_2)_{2k,2k}(S_2^{-1})_{2k-1,2k}\\
&=\rho^2_{2k+1} \bar \alpha_{2k}^{(2)},\\
J_{2k+1,2k+1}&=(S_1E_{2k,2k+2}S_1^{-1})_{2k+1,2k+1}=(S_1)_{2k+1,2k}(S_1^{-1})_{2k+2,2k+1}=-(S_1)_{2k+1,2k}(S_1)_{2k+2,2k+1}\\
&=-\bar\alpha_{2k+1}^{(2)}\alpha_{2k+2}^{(1)},\\
J_{2k+1,2k+2}&=(S_1E_{2k,2k+2}S_1^{-1})_{2k+1,2k+2}=(S_1)_{2k+1,2k}(S_1^{-1})_{2k+2,2k+2}=\bar \alpha_{2k+1}^{(2)}.
\end{align*}

\paragraph{Proof of Proposition \ref{pro.alpha.non.hermitian}}

For $k=0$ the result comes from the definition of $\rho_0^2$. For $k=1,2$ we have to use the truncated recursion relations,
\begin{align*}
z^{-1}\varphi_1^{(0)}&=\varphi_1^{(1)}-\bar \alpha_1^{(2)} \varphi_1^{(0)},&
z\varphi_1^{(0)}&=\varphi_1^{(2)}-\alpha_2^{(1)}\varphi_1^{(1)}-\alpha_1^{(1)}\varphi_1^{(0)},\\
z^{-1}\varphi_1^{(1)}&=\varphi_1^{(3)}-\bar \alpha_3^{(2)}\varphi_1^{(2)}-\alpha_1^{(1)}\bar \alpha_2^{(2)}\varphi_1^{(1)}-\rho_1^2\bar \alpha_2^{(2)}\varphi_1^{(0)},&z\varphi_1^{(1)}&=\bar \alpha_1^{(2)} \varphi_1^{(2)}-\alpha_2^{(1)}\bar \alpha_1^{(2)}\varphi_1^{(1)}+\rho_1^2\varphi_1^{(0)},
\end{align*}
multiplying by $z$ and integrating we obtain $h_0=h_1-\bar \alpha_1^{(2)}\int \varphi_1^{(0)}z \d \mu$, and $\int \varphi_1^{(0)}z \d \mu=-\alpha_1^{(1)}h_0 $, from where we have $h_0=h_1+\bar \alpha_1^{(2)}\alpha_1^{(1)}h_0$. Now multiplying by $z^{-1}$ and integrating we obtain $0=\bar \alpha_1^{(2)} h_2-\alpha_2^{(1)}\bar \alpha_1^{(2)}\int\varphi_1^{(1)}z^{-1}\d \mu+\rho_1^2\int \varphi_1^{(0)}z^{-1}\d \mu $, $\int\varphi_1^{(1)}z^{-1}\d \mu=-\rho_1^2\bar \alpha_2^{(2)}h_0$ and $\int \varphi_1^{(0)}z^{-1}\d \mu=-\bar \alpha_1^{(2)}h_0$ leading to $0=h_2+\alpha_2^{(1)}\bar \alpha_2^{(2)}h_1-h_1$.

The other $k\geq 2$ can be proved by induction. For odd case we multiply \eqref{pro.inv.relations.even} by $z^{k+1}$ to obtain
\begin{align*}
0=\alpha_{2k}^{(1)}h_{2k+1}-\alpha_{2k}^{(1)}\bar \alpha_{2k+1}^{(2)}\int_{\T}\varphi_1^{(2k)}z^{k+1}\d \mu+\rho^2_{2k}\alpha_{2k-1}^{(1)}\int_{\T}\varphi_1^{(2k-1)}z^{k+1}\d \mu+\rho_{2k-1}^2\rho^2_{2k}\int_{\T}\varphi_1^{(2k-2)}z^{k+1}\d \mu
\end{align*}
then multiplying by $z^k$ the recurrence expressions \eqref{5terms.even} and \eqref{5terms.odd} for $z\varphi_1^{(2k)}$, $z\varphi_1^{(2k-1)}$, $z\varphi_1^{(2k-2)}$ and integrating we substitute, to get
\begin{align*}
0=h_{2k+1}+\alpha_{2k+1}^{(1)}\bar \alpha_{2k+1}^{(2)}h_{2k}-\alpha_{2k-1}^{(1)}\bar\alpha_{2k-1}^{(2)}h_{2k}-\rho^2_{2k-1}h_{2k},
\end{align*}
from where using the induction principle the result is proven. Should we want to obtain the rest of the equations (those with even $k$), then it is necessary to multiply by $z^{-k-1}$ the odd recurrence relations for $z\varphi_1^{(2k+1)}$ \eqref{5terms.odd} and use the same procedure.

\paragraph{Proof from Proposition \ref{reproducing.CMV}}

From the definition we have
\begin{align*}
\int_{\T}K^{[l]}_{\CMV}(u,z') K^{[l]}_{\CMV}(z,u)\d \mu(u)&:=\int_{\T}\sum_{j=0}^{l-1}\varphi_{1}^{(j)}(z')\bar \varphi_{2}^{(j)}(\bar u) \sum_{k=0}^{l-1}\varphi_{1}^{(k)}(u)\bar \varphi_{2}^{(k)}(\bar z) \d \mu(u) \\
&= \sum_{j=0}^{l-1}\varphi_{1}^{(j)}(z')\sum_{k=0}^{l-1}\bar \varphi_{2}^{(k)}(\bar z)\int_{\T}\varphi_{1}^{(k)}(u)\bar \varphi_{2}^{(j)}(\bar u) \d \mu(u)\\
&=\sum_{j=0}^{l-1}\varphi_{1}^{(j)}(z')\sum_{k=0}^{l-1}\bar \varphi_{2}^{(k)}(\bar z)\delta_{j,k}=K^{[l]}_{\CMV}(z,z').
\end{align*}

\paragraph{Alternate proof for Theorem \ref{exp.CMV.associated}} Due to Gerardo Ariznabarreta.

From the block factorization problem $g^{[l]}=(S_1^{[l]})^{-1}S_2^{[l]}$. As $S_2$ is upper-triangular, we can write $S_1^{[\geq l,l]}g^{[l]}+S_1^{[\geq l,\geq l]}g^{[\geq l,l]}=S_2^{[\geq l,l]}=0$ from where
\begin{equation*}
    S_1^{[\geq l,l]}=-S_1^{[\geq l, \geq l]}g^{[\geq l,l]}(g^{[l]})^{-1},
\end{equation*}
then using the definition for $\varphi^{(l)}_1$ and using the previous formula, we get
\begin{align*}
\varphi^{(l)}_1&=\chi^{(l)}+\sum_{j=0}^{l-1} (S_{1}^{[\geq l,l]})_{0,j}\chi^{(j)}=\chi^{(l)}-\sum_{r,j=0}^{l-1}\sum_{k \geq 0} (S_1^{[\geq l,\geq l]})_{0,k}(g^{[\geq l,l]})_{k,r}(g^{[l]})^{-1}_{r,j}\chi^{(j)}\\
&=\chi^{(l)}-\sum_{r,j=0}^{l-1}(g^{[\geq l,l]})_{0,r}(g^{[l]})^{-1}_{r,j}\chi^{(j)}=\chi^{(l)}-\begin{pmatrix} g_{l,0}&g_{l,1}&\cdots & g_{l,l-1}\end{pmatrix}(g^{[l]})^{-1}\chi^{[l]}.
\end{align*}
In addition, we can express the formula for $\varphi^{(l+1)}$ in the following way as
\begin{align*}
\varphi^{(l+1)}&=\chi^{(l+1)}+(S_{1}^{[\geq l+1,l+1]})_{0l}\chi^{(l)}+\sum_{j=0}^{l-1} (S_{1}^{[\geq l+1,l+1]})_{0j}\chi^{(j)}\\
&=\chi^{(l+1)}+(S_1)_{l+1,l}\chi^{(l)}+\sum_{j=0}^{l-1} (S_{1}^{[\geq l+1,l+1]})_{0,j}\chi^{(j)}\\
&=\chi^{(l+1)}+(S_1)_{l+1,l}\chi^{(l)}+\sum_{j=0}^{l-1} (S_{1}^{[\geq l+1,l+1]})_{1,j}\chi^{(j)}\\
&=\chi^{(l+1)}+(S_1)_{l+1,l}\chi^{(l)}-\sum_{r,j=0}^{l-1}\sum_{k \geq 0} (S_1^{[\geq l,\geq l]})_{1,k}(g^{[\geq l,l]})_{k,r}(g^{[l]})^{-1}_{r,j}\chi^{(j)}\\
&=(S_1)_{l+1,l}(\chi^{(l)}-\begin{pmatrix} g_{l,0}&g_{l,1}&\cdots & g_{l,l-1}\end{pmatrix}(g^{[l]})^{-1}\chi^{[l]})+\\
&+(\chi^{(l+1)}-\begin{pmatrix} g_{l+1,0}&g_{l+1,1}&\cdots & g_{l+1,l-1}\end{pmatrix}(g^{[l]})^{-1}\chi^{[l]}),
\end{align*}
from where we obtain (if $l$ is odd) $\varphi_{1,+1}^{(l)}(z)=\varphi_{1}^{(l+1)}(z)-\alpha_{l+1}\varphi^{(l)}_{1}(z)$.

\paragraph{Proof of Lemma \ref{ABC.th.cmv}}
If we denote by $\Pi^{[l]}=\sum_{i=0}^{l-1}E_{i,i}$ (the projection over the first $l$ components) we find
\begin{align*}
K^{[l]}_{\CMV}(z,z')&=(\Pi^{[l]}\Phi_{2}(z))^{\dag}\Pi^{[l]}\Phi_{1}(z')= \Phi_{2}(z)^{\dag}\Pi^{[l]}\Phi_{1}(z')=\chi_{\CMV}(z)^{\dag}S_2^{-1}\Pi^{[l]}S_1\chi_{\CMV}(z').
\end{align*}
From the block factorization $g_{\CMV}^{[l]}=(S_1^{[l]})^{-1}S_2^{[l]}$ and its inverse $(g_{\CMV}^{[l]})^{-1}=(S_2^{[l]})^{-1}S_1^{[l]}$ we can get an expression for $K_{\CMV}^{[l]}$ using only finite size matrices, that is
\begin{align*}
K^{[l]}_{\CMV}(z,z')&=\chi_{\CMV}(z)^{\dag}\Pi^{[l]}S_2^{-1}\Pi^{[l]}S_1\Pi^{[l]}\chi_{\CMV}(z')\\
&=\chi_{\CMV}^{[l]}(z)^{\dag}(S_2^{[l]})^{-1}S_1^{[l]}\chi_{\CMV}^{[l]}(z')\\
&=\chi_{\CMV}^{[l]}(z)^{\dag}(g_{\CMV}^{[l]})^{-1}\chi_{\CMV}^{[l]}(z').
\end{align*}

\paragraph{Proof of Lemma \ref{primera.CD.CMV}}

The symmetry of $g$ in \eqref{symz} can be expressed using the block structure
\begin{align*}
\Upsilon &= \left(\begin{BMAT}{c|c}{c|c} \Upsilon^{[l]} & \Upsilon^{[l,\geq l]} \\ \Upsilon^{[\geq l, l]} & \Upsilon^{[\geq l]} \end{BMAT}\right), &
g&=\left(\begin{BMAT}{c|c}{c|c} g^{[l]} & g^{[l,\geq l]} \\ g^{[\geq l, l]} & g^{[\geq l]} \end{BMAT}\right),
\end{align*}
With this block structure we get
\begin{align*}
\Upsilon^{[l]} g^{[l]} + \Upsilon^{[l, \geq l]} g^{[\geq l,l]} =  g^{[l]} \Upsilon^{[l]} + g^{[l,\geq l]} \Upsilon^{[\geq l,l]},
\end{align*}
or equivalently, recalling the Gaussian factorization, we arrive to
\begin{align*}
(g^{[l]})^{-1}\Upsilon^{[l]}-\Upsilon^{[l]} (g^{[l]})^{-1} = (g^{[l]})^{-1}(g^{[l,\geq l]} \Upsilon^{[\geq l,l]}-\Upsilon^{[l, \geq l]} g^{[\geq l,l]})(g^{[l]})^{-1}.
\end{align*}
We have also the equations
\begin{align*}
\Upsilon^{[l]} \chi^{[l]}(z)+ \Upsilon^{[l,\geq l]} \chi^{[\geq l]}(z) & = z \chi^{[l]}(z), & \chi^{[l]}(z)^{\dag} \Upsilon^{[l]} +  \chi^{[\geq l]}(z)^{\dag} \Upsilon^{[\geq l, l]}  &=  \bar z^{-1} \chi^{[l]}(z)^{\dag},
\end{align*}
that leads to
\begin{align*}
(z'-\bar z^{-1})K^{[l]}(z,z')&=\chi^{[l]}(z)^{\dag}(g^{[l]})^{-1}z'\chi^{[l]}(z')-\bar z^{-1} \chi^{[l]}(z)^{\dag}(g^{[l]})^{-1}\chi^{[l]}(z')\\
&=\chi^{[l]}(z)^{\dag}(g^{[l]})^{-1}(\Upsilon^{[l]} \chi^{[l]}(z')+ \Upsilon^{[l,\geq l]} \chi^{[\geq l]}(z'))-\\
&-(\chi^{[l]}(z)^{\dag} \Upsilon^{[l]} +  \chi^{[\geq l]}(z)^{\dag} \Upsilon^{[\geq l, l]})(g^{[l]})^{-1}\chi^{[l]}(z')\\
&=\chi^{[l]}(z)^{\dag} ((g^{[l]})^{-1}\Upsilon^{[l]}-\Upsilon^{[l]}(g^{[l]})^{-1})\chi^{[l]}(z')+\\
&+\chi^{[l]}(z)^{\dag}(g^{[l]})^{-1}\Upsilon^{[l,\geq l]} \chi^{[\geq l]}(z')-\chi^{[\geq l]}(z)^{\dag} \Upsilon^{[\geq l, l]}(g^{[l]})^{-1}\chi^{[l]}(z')\\
&=\chi^{[l]}(z)^{\dag} ((g^{[l]})^{-1}(g^{[l,\geq l]} \Upsilon^{[\geq l,l]}-\Upsilon^{[l, \geq l]} g^{[\geq l,l]})(g^{[l]})^{-1}) \chi^{[l]}(z')+\\
&+\chi^{[l]}(z)^{\dag}(g^{[l]})^{-1}\Upsilon^{[l,\geq l]} \chi^{[\geq l]}(z')-\chi^{[\geq l]}(z)^{\dag} \Upsilon^{[\geq l, l]}(g^{[l]})^{-1}\chi^{[l]}(z')\\
&=(\chi^{[l]}(z)^{\dag}(g^{[l]})^{-1}g^{[l,\geq l]}-\chi^{[\geq l]}(z)^{\dag})\Upsilon^{[\geq l, l]}(g^{[l]})^{-1}\chi^{[l]}(z')-\\
&-\chi^{[l]}(z)^{\dag}(g^{[l]})^{-1}\Upsilon^{[l, \geq l]}(g^{[\geq l,l]}(g^{[l]})^{-1}\chi^{[l]}(z')-\chi^{[\geq l]}(z')).
\end{align*}

\paragraph{Proof of Proposition \ref{grados.vec}}

If $a(l)=1$ then $z^{\nu_-(l)} \varphi_1^{(l)}(z)$ is a monic polynomial of degree $\nu_-(l)+\nu_+(l)-1$, while when $a(l)=2$ then $z^{\nu_+(l)-1} \bar \varphi_1^{(l)}(z^{-1})$ is a monic polynomial of degree $\nu_-(l)+\nu_+(l)-1$. The orthogonality relations for $z^{\nu_-(l)} \varphi_1^{(l)}(z)$ and $z^{\nu_+(l)-1} \bar \varphi_1^{(l)}(z^{-1})$ can be obtained from \eqref{orth2}
\begin{align}
\begin{aligned}\label{orth3}
\oint_{\T} z^{\nu_-(l)} \varphi_{\vec n, 1}^{(l)}(z) z^{-k} \d \mu(z)&=0, & k&=0,\dots,|\vec\nu(l)|-1, & a(l)&=1, \\
\oint_{\T} z^{\nu_+(l)-1} \bar \varphi_{\vec n, 1}^{(l)}(z^{-1}) z^{-k} \d \mu(z)&=0, & k&=0,\dots,|\vec\nu(l)|-1, & a(l)&=2,
\end{aligned}
\end{align}
that means that
\begin{align*}
  z^{\nu_-(l)}\varphi_{\vec n,1}^{(l)}(z)&= P_{|\vec\nu(l)|-1}(z),& a(l)&=1,\\
  z^{\nu_+(l)-1}\bar \varphi_{\vec n,1}^{(l)}(z^{-1})&= P_{|\vec\nu(l)|-1}(z),& a(l)&=2,
  \end{align*}
recalling that $|\vec\nu(l)|-1=l$ we get the desired result.

\paragraph{Proof of Proposition \ref{pro.gsymz}}
\begin{enumerate}\item The shift operators defined fulfill
\begin{align*}
\Lambda_{\vec n,1} \chi_{\vec n}(z)&=z\Pi_{\vec n,1}  \chi_{\vec n}(z),&
\Lambda_{\vec n,2}  \chi_{\vec n}(z)&=z^{-1}\Pi_{\vec n,2}  \chi_{\vec n}(z), \\
\Lambda_{\vec n,1}^{\top} \chi_{\vec n}(z)&=(z^{-1}\Pi_{\vec n,1}-E_{0,0}\Lambda^{n^+}) \chi_{\vec n}(z),&
\Lambda_{\vec n,2}^{\top} \chi_{\vec n}(z)&=(z\Pi_{\vec n,2} -E_{n^+,n^+}(\Lambda^{\top})^{n^+}) \chi_{\vec n}(z),
\end{align*}
that means that
\begin{align*}
(\Lambda_{\vec n,1} +\Lambda_{\vec n,2}^{\top}+E_{n_+,n_+}(\Lambda^{\top})^{n_+})\chi_{\vec n}(z)&=z\chi_{\vec n}(z), & (\Lambda_{\vec n,1}^{\top} +\Lambda_{\vec n,2}+E_{0,0}\Lambda^{n_+})\chi_{\vec n}(z)&=z^{-1}\chi_{\vec n}(z),
\end{align*}
from there it follows that
\begin{align*}
\Upsilon_{\vec n} g_{\vec n}&=\oint_{\T}z\chi_{\vec n}(z) \chi_{\vec n}(z)^{\dag} \d \mu(z)\\
&=\oint_{\T}\chi_{\vec n}(z) (z^{-1}\chi_{\vec n}(z))^{\dag} \d \mu(z)\\
&=g_{\vec n}(\Lambda_{\vec n,1}^{\top}+\Lambda_{\vec n,2}+E_{0,0}\Lambda^{n_+})^{\dag}\\
&=g_{\vec n}(\Lambda_{\vec n,1}+\Lambda_{\vec n,2}^{\top}+E_{n_+,n_+}(\Lambda^{\top})^{n_+})\\
&=g_{\vec n}\Upsilon_{\vec n}.
\end{align*}
\item
With the definitions
\begin{align*}
J_{\vec n,1}&:=S_{\vec n,1} \Upsilon_{\vec n} S_{\vec n,1}^{-1}, & J_{\vec n,2}&:=S_{\vec n,2}\Upsilon_{\vec n}S_{\vec n,2}^{-1},
\end{align*}
the use of Proposition \ref{pro.gsymz} leads easily to
\begin{equation}
J_{\vec n}:=J_{\vec n, 1}=J_{\vec n, 2}.
\end{equation}
The matrix $J_{\vec n, 1}$ has $n_-+1$ diagonals over the main diagonal and the matrix $J_{\vec n, 2}$ has $n_++1$ diagonals under the main diagonal (in both computations we have excluded the main diagonal itself), so both $J_{\vec n, 1}, J_{\vec n, 2}$ have $n_++n_-+3$ diagonal band.
\end{enumerate}

\paragraph{Proof of Proposition \ref{red.Toeplitz}}

We have calculated previously $J=L_1$ and using the same method $L_2$ can be calculated, both are five-diagonal matrices given by
\begin{align*}
L_1=J=\begin{pmatrix} - \alpha_1^{(1)} & -\alpha_2^{(1)} & 1 & 0 & 0 & 0 & 0 &0 & \cdots \\
\rho^2_1 & -\alpha_2^{(1)} \bar \alpha_1^{(2)} & \bar \alpha_1^{(2)} & 0 & 0 & 0 & 0 & 0 & \cdots\\
0 & -\rho_2^2  \alpha_3^{(1)} & - \alpha_3^{(1)}\bar \alpha_2^{(2)} & -\alpha_4^{(1)} & 1 & 0 & 0 &0 & \cdots\\
0 & \rho^2_2 \rho^2_3 & \rho_3^2 \bar \alpha_2^{(2)} & -\alpha_4^{(1)} \bar \alpha_3^{(2)} & \bar \alpha_3^{(2)} & 0 &0 &0 & \cdots\\
0 & 0 &0 & -\rho_4^2 \alpha_5^{(1)} & - \alpha_5^{(1)}\bar \alpha_4^{(2)}  & -\alpha_6^{(1)} & 1 & 0 & \cdots\\
0 & 0 & 0 & \rho^2_4 \rho^2_5& \rho^2_5 \bar \alpha_4^{(2)} & -\alpha_6^{(1)}\bar \alpha_5^{(2)}  & \bar \alpha_5^{(2)} & 0 & \cdots\\
\vdots & \vdots & \vdots & \vdots & \vdots & \vdots & \vdots & \vdots & \ddots
\end{pmatrix}
\end{align*}
\begin{align*}
L_2=\begin{pmatrix} -\bar \alpha_1^{(2)} & 1 & 0 & 0 & 0 & 0 & 0 &0 & \cdots \\
-\rho_1^2 \bar \alpha_2^{(2)} & - \alpha_1^{(1)} \bar \alpha_2^{(2)} & -\bar \alpha_3^{(2)} & 1 & 0 & 0 & 0 & 0 & \cdots\\
\rho_1^2 \rho_2^2 & \rho_2^2  \alpha_1^{(1)} & -\alpha_2^{(1)} \bar \alpha_3^{(2)} & \alpha_2^{(1)} & 0 & 0 &0 & 0 & \cdots\\
0 & 0 & -\rho_3^2 \bar \alpha_4^{(2)} & - \alpha_3^{(1)} \bar \alpha_4^{(2)} & -\bar \alpha_5^{(2)} & 1 &0 &0 & \cdots \\
0 & 0 &\rho_3^2 \rho_4^2 & \rho_4^2  \alpha_3^{(1)} & -\alpha_4^{(1)} \bar \alpha_5^{(2)} & \alpha_4^{(1)} & 0 & 0 & \cdots\\
0 & 0 & 0 & 0 & -\rho^2_5 \bar \alpha_6^{(2)} & -  \alpha_5^{(1)} \bar \alpha_6^{(2)} & -\bar \alpha_7^{(2)} & 1 & \cdots\\
\vdots & \vdots & \vdots & \vdots & \vdots & \vdots & \vdots & \vdots & \ddots
\end{pmatrix}
\end{align*}
as we are looking only for time flows associated to $t_{11}$ and $t_{21}$ then $B_{1,1}=(L_1)_+$ and $B_{2,1}=(L_2)_-$. Using \eqref{laxeq} in the matrix elements $(L_1)_{k,k+1}$ for $k \geq 0$  we obtain
\begin{align*}
\frac{\partial \alpha_{2k+2}^{(1)}}{\partial t_{11}}&=-\frac{\partial (L_1)_{2k,2k+1}}{\partial t_{11}}=-[B_{1,1},L_1]_{2k,2k+1}=\alpha_{2k+3}^{(1)}(1-\alpha_{2k+2}^{(1)}\bar \alpha_{2k+2}^{(2)})\\
\frac{\partial \alpha_{2k+2}^{(1)}}{\partial t_{21}}&=-\frac{\partial (L_1)_{2k,2k+1}}{\partial t_{21}}=-[B_{2,1},L_1]_{2k,2k+1}=\alpha_{2k+1}^{(1)}(1-\alpha_{2k+2}^{(1)}\bar \alpha_{2k+2}^{(2)})\\
\frac{\partial \bar \alpha_{2k+1}^{(2)}}{\partial t_{11}}&=\frac{\partial (L_1)_{2k+1,2k+2}}{\partial t_{11}}=[B_{1,1},L_1]_{2k+1,2k+2}=-\bar \alpha_{2k}^{(2)}(1-\alpha_{2k+1}^{(1)}\bar \alpha_{2k+1}^{(2)})\\
\frac{\partial \bar \alpha_{2k+1}^{(2)}}{\partial t_{21}}&=\frac{\partial (L_1)_{2k+1,2k+2}}{\partial t_{21}}=[B_{2,1},L_1]_{2k+1,2k+2}=-\bar \alpha_{2k+2}^{(2)}(1-\alpha_{2k+1}^{(1)}\bar \alpha_{2k+1}^{(2)})
\end{align*}
and now looking at $(L_1)_{2k,2k}$ and $(L_2)_{2k+1,2k+1}$ for $k \geq 0$ we obtain the rest of the equations
\begin{align*}
\frac{\partial (\bar \alpha_{2k}^{(2)} \alpha_{2k+1}^{(1)})}{\partial t_{11}}&=-\frac{\partial (L_1)_{2k,2k}}{\partial t_{11}}=-[B_{1,1},L_1]_{2k,2k} \quad \Rightarrow &\frac{\partial \alpha_{2k+1}^{(1)}}{\partial t_{11}}&=\alpha_{2k+2}^{(1)}(1-\alpha_{2k+1}^{(1)}\bar \alpha_{2k+1}^{(2)})\\
\frac{\partial (\bar \alpha_{2k}^{(2)} \alpha_{2k+1}^{(1)})}{\partial t_{21}}&=-\frac{\partial (L_1)_{2k,2k}}{\partial t_{21}}=-[B_{2,1},L_1]_{2k,2k} \quad  \Rightarrow & \frac{\partial \alpha_{2k+1}^{(1)}}{\partial t_{21}}&=\alpha_{2k}^{(1)}(1-\alpha_{2k+1}^{(1)}\bar \alpha_{2k+1}^{(2)})\\
\frac{\partial (\alpha_{2k+1}^{(1)} \bar \alpha_{2k+2}^{(2)})}{\partial t_{11}}&=-\frac{\partial (L_2)_{2k+1,2k+1}}{\partial t_{11}}=-[B_{1,1},L_2]_{2k+1,2k+1} \quad \Rightarrow &\frac{\partial \bar \alpha_{2k+2}^{(2)}}{\partial t_{11}}&=-\bar \alpha_{2k+1}^{(2)}(1-\alpha_{2k+2}^{(1)}\bar \alpha_{2k+2}^{(2)})\\
\frac{\partial (\alpha_{2k+1}^{(1)} \bar \alpha_{2k+2}^{(2)})}{\partial t_{21}}&=-\frac{\partial (L_2)_{2k+1,2k+1}}{\partial t_{21}}=-[B_{2,1},L_2]_{2k+1,2k+1} \quad \Rightarrow &\frac{\partial \bar \alpha_{2k+2}^{(2)}}{\partial t_{21}}&=-\bar \alpha_{2k+3}^{(2)}(1-\alpha_{2k+2}^{(1)}\bar \alpha_{2k+2}^{(2)})\\
\end{align*}
considering all the equations we obtain \eqref{CMV.Lat}.

\paragraph{Proof of Proposition \ref{dets}}

First let us look to \eqref{det.ass}. Using Definition \ref{def.vec.associated} it can be expressed as
\begin{align*}
\varphi_{1,+a}^{(l)}(z)&=\chi^{(l_{+a})}(z)-\sum_{i,j=0}^{l-1}
    g_{l_{+a},i}(g^{[l]})^{-1}_{ij}\chi^{(j)}(z)\\
    &=\frac{1}{\det g^{[l]}}\Big(\chi^{(l_{+a})}(z)\det g^{[l]}-\sum_{i,j=0}^{l-1}
    g_{l_{+a},i}(-1)^{i+j}M^{(l)}_{ji}\chi^{(j)}(z)\Big)\\
    &=\frac{1}{\det g^{[l]}}\Big(\chi^{(l_{+a})}(z)\det g^{[l]}+\sum_{j=0}^{l-1}(-1)^{j+l}\sum_{i=0}^{l-1}(-1)^{l-1+j}g_{l_{+a},i}M^{(l)}_{ji}(-1)^{i+j}\chi^{(j)}(z)\Big),
\end{align*}
that is the expansion of \eqref{det.ass}. Using the same idea with \eqref{det.dualass}
\begin{align*}
\varphi_{ 2, -a}^{(l)}(z)&=\sum_{j=0}^l(g^{[l+1]})^{-1}_{l_{-a} j}\chi^{(j)}(z)=\frac{1}{\det g^{[l+1]}}\sum_{j=0}^l(-1)^{l+l_{-a}}(-1)^{l+j}M^{(l+1)}_{jl_{-a}}\chi^{(j)}(z),
\end{align*}
we arrive at the expansion of \eqref{det.dualass} taking the complex conjugate. Both \eqref{det.ass.minus} and \eqref{det.dualass.plus} can be proved using the same ideas.
\end{appendices}

\end{document}